\documentclass[11pt]{amsart}

\usepackage{t1enc}
\usepackage{amsthm,amssymb}
\usepackage{latexsym}
\usepackage{amsmath} 
\usepackage{graphics}
\usepackage{epsfig,multicol}
\usepackage{amsfonts}
\usepackage[latin1]{inputenc}
\usepackage[english]{babel}
\usepackage{ mathrsfs }

\newtheorem{theorem}{Theorem}[subsection]
\newtheorem{proposition}{Proposition}[subsection]
\newtheorem{lemma}{Lemma}[subsection]
\newtheorem{definition}{Definition}[subsubsection]

\newtheorem{corollary}{Corollary}[subsection]
\newtheorem{remark}{Remark}[subsection]
\newtheorem{example}{Remark}[subsection]

\def\hpic #1 #2 {\mbox{$\begin{array}[c]{l} \epsfig{file=#1,height=#2}
\end{array}$}}
 
\def\vpic #1 #2 {\mbox{$\begin{array}[c]{l} \epsfig{file=#1,width=#2}
\end{array}$}}

\def\ct {\widetilde T}
\newcommand  {\rmn}\romannumeral
\newcommand{\notetoself}[1]{\marginpar{\tiny #1}}
\newcommand {\ds}{\mathscr{D}}

\newcommand {\diff}{Diff$(S^1)$}

\begin{document}
\title{Some unitary representations of Thompson's groups F and T.}
\author{Vaughan F. R. Jones\\Vanderbilt University}
\thanks{V.J. is supported by the NSF under Grant No. DMS-0301173}
\maketitle
\begin{abstract}
In a ``naive'' attempt to create  algebraic quantum field theories on the circle, we obtain
a family of unitary representations of Thompson's groups T and F for any subfactor. 
The Thompson group elements are the ``local scale transformations'' of the theory.
In a simple case the coefficients of the representations are polynomial invariants
of links. We show that all links  arise and introduce new ``oriented'' subgroups 
$\overrightarrow F <F$ and $\overrightarrow T< T$ which allow us to produce all 
\emph{oriented} knots and links.

\end{abstract}

\section{Introduction}
This paper is part of an ongoing effort to construct a conformal field theory for every finite index subfactor in such a way
that the standard  invariant of the subfactor, or at least its quantum double, can be recovered from the CFT. There is no
doubt that interesting subfactors arise in CFT nor that in some cases the numerical data of the subfactor appears as numerical
data in the CFT. But there are supposedly ``exotic'' subfactors for which no CFT is known to exist, the first of which was 
constructed by Haagerup in a tour de force in \cite{H5},\cite{AH}. But in the last few years ideas of Evans and Gannon (see \cite{EG11}) have made
it seem plausible that CFT's exist for the Haagerup and other exotic subfactors constructed in the Haagerup line (see \cite{JMS}). This
has revived the author's interest in giving a construction of a CFT from subfactor data.

The most complete way to do this would be to extract form the subfactor the Boltzmann weights of  a critical two-dimensional lattice
model then construct  a quantum field theory  from the scaling limit of the $n$-point functions. Looking at the monodromy
representations of the braid group one would then construct a subfactor as in the very first constructions of \cite{jo1}.
This ``royal road'' is paved with many mathematical difficulties and it is probably impossible to complete  with current technology except in the
very simplest examples. 

There are alternatives, however, to using the scaling limit of the $n$-point functions. The algebraic (Haag-Kastler \cite{HK},\cite{DHR1},\cite{H}) approach has been 
quite successful in understanding some aspects of conformal field theory-\cite{FRS},\cite{K17},\cite{ek}. After splitting the CFT into two chiral halves, this approach
predicts the existence of ``conformal nets''- von Neumann algebras $\mathcal A (I)$ on the Hilbert space $\mathcal H$, associated to closed intervals $I\subset S^1$, 
and a continuous projective unitary representation $ \alpha \mapsto u_\alpha \mbox{ of } \mbox {Diff} S^1$, on $\mathcal H$, satisfying four axioms:
$$(\romannumeral 1) \qquad \mathcal A(I)\subseteq \mathcal A(J) \mbox{  if  } I\subseteq J$$
$$(\romannumeral 2) \qquad \mathcal A(I)\subseteq \mathcal A(J)' \mbox{  if  } I\cap J=\emptyset$$
$$(\romannumeral 3) \qquad  u_\alpha \mathcal A(I)u_\alpha^{-1}=\mathcal A(\alpha(I))$$
$$(\romannumeral 4)\qquad \sigma(Rot(S^1))\subset \mathbb Z^+\cup \{0\}$$

Here by $Rot(S^1)$ we mean the subgroup of rotations in  \diff. $Rot(S^1)$  may be supposed to act as an honest representation which can therefore 
be decomposed into eigenspaces (Fourier modes). The eigenvalues, elements of $\hat{S^1}=\mathbb Z$, are the spectrum $\sigma$ of the representation.

There may or may not be a vacuum vector $\Omega$ in $\mathcal H$ which would be fixed by the linear fractional transformations in \diff, and would be cyclic 
and separating for all the $\mathcal A(I)$.

The $\mathcal A(I)$ can be shown to be type III$_1$ factors  so subfactors appear by axiom (\romannumeral 2) as $\mathcal A(I')\subseteq \mathcal A(I)'$ where
$I'$ is the closure of the complement of $I$. 

Non-trivial examples of such conformal nets were constructed in \cite{Wa7},\cite{ToL} by the analysis 
of unitary loop group representations (\cite{PSe}). These examples can be
exploited to construct many more.

On an apparently completely different  front, the study of subfactors for their own sake led to the development of ``planar algebras'' which in their strictest form 
(\cite{jo2})
are an axiomatization of the standard invariant of a subfactor but by changing the axioms slightly they yield an axiomatization of correspondences (bimodules) in
the sense of Connes (\cite{C20}), and systems of such. The most significant ingredient of a planar algebra is the existence of a positive definite inner product which
is interpreted diagrammatically. More precisely a planar algebra is a graded vector space $\mathcal P=(P_n)$ of vector spaces where $n$ is supposed to
count the number of boundary points on a disc into which the elements of $P_n$ can be ``inserted''.  Given a planar tangle - a finite collection of discs inside
a big (output) disc, all discs having boundary points and all boundary points being connected by non-crossing curves called strings, the insertion of
elements of $\mathcal P$ into the internal discs produces an output element in $P_n$, $n$ being the number of boundary points on the output disc.

If $P_0$ is one-dimensional the map $\langle x, y\rangle= \vpic{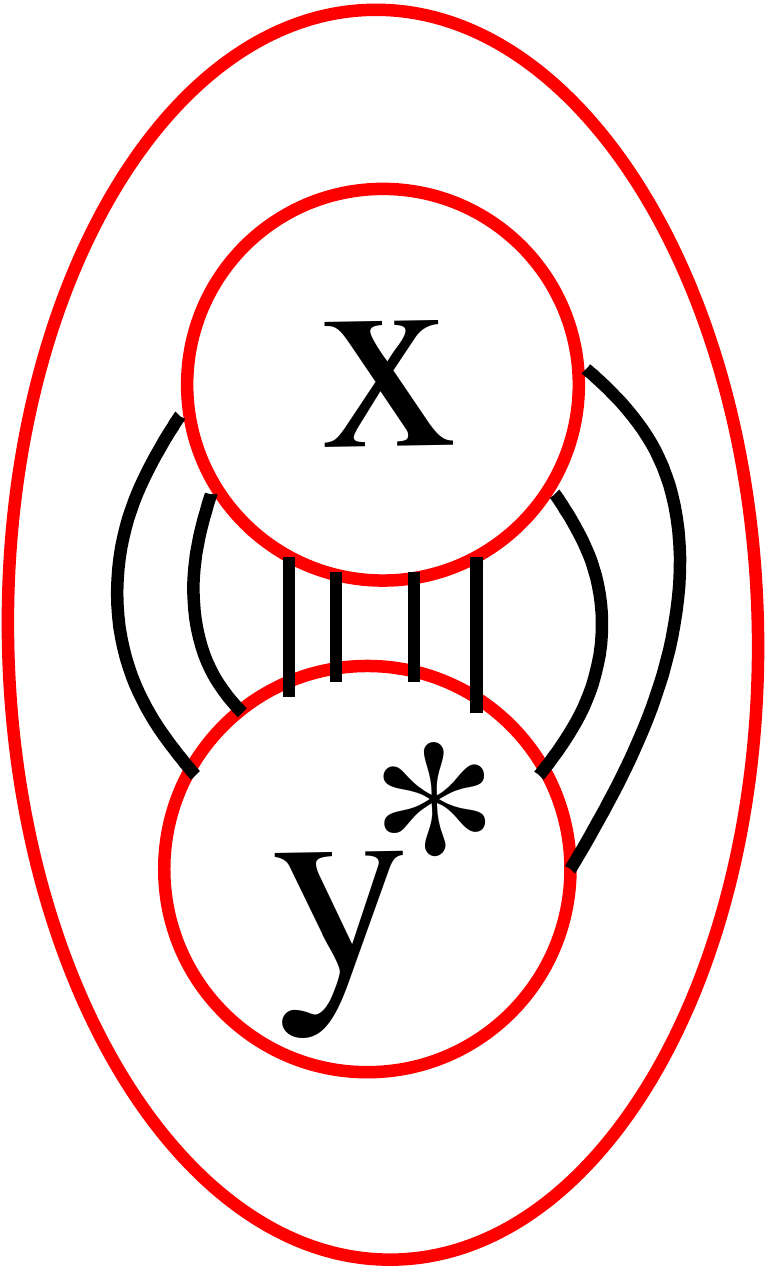} {0.5in} $ defines a sesquilinear map to $\mathbb C$ and this is the inner product.

The idea of obtaining a ``continuum limit'' by letting the number of boundary points on the discs fill out the circle has been around for over 20 years but this
paper is the first one to take a concrete, though by no means big enough, step in that direction. 
Planar algebra is an abstraction of the notion of (planar) manipulations of the tensor powers of a given finite dimensional Hilbert space (thus in some 
sense a planar version of \cite{Pen}),
and our constructions below of limit Hilbert spaces are really versions, aimed at a scaling rather than a thermodynamic limit, of
von Neumann's original infinite tensor product-\cite{vNdirect}. Background for this point of view is detailed in \cite{J22}.

Given the difficulty of following the royal road using the scaling limit, we are trying to construct the local algebras $\mathcal A(I)$ directly from a planar 
algebra. A well known idea in physics is the block spin renormalization procedure (\cite{CV}). Here one groups the spins in a block on one scale and replaces the blocks
by spins of the same kind on a coarser scale. Hamiltonians (interactions) between the spins and blocks of spins are chosen so that the physics on
the block spin scale resembles the physics on the original scale. This procedure is tricky to implement but we shall use the idea. For, however one
plays it, in constructing a continuum limit one must relate the Hilbert space on one scale to the Hilbert space on a finer scale.  It is this relation that
we are trying to produce using structures suggested by planar algebra.  

More precisely, given a planar algebra $\mathcal P$,  for a choice of $n$ points on $S^1$, called $B_n$, we will associate the Hilbert space $\mathcal H_n =P_n$ , and for an inclusion $B_n\subset B_m$ we will use planar algebra data to group boundary points into blocks and construct a projection from $\mathcal H_m$ onto $\mathcal H_n$.
Alternatively we are defining isometries of $\mathcal B_n$ into $\mathcal B_m$ and
the Hilbert space $\mathcal H$ of the theory will then be the direct limit of the $\mathcal H_n$.

The first embeddings of $\mathcal B_n$ into $\mathcal B_m$ will work for \emph{any} finite subsets of $S^1$ by simply adding points using an
essentially trivial element of the planar algebra.  We will briefly describe these embeddings as they are a simplified template for
what follows. In particular one obtains projective unitary representations of \diff (and coherent local algebras if desired). But  this construction fails to be of any interest for two essential reasons. The first is that the unitary representations are independent of the planar algebra data used.  The second is more serious-
the action of \diff is hopelessly discontinuous.  

We have no answer to the discontinuity problem. In a sense it would be surprising if we did since our input is entirely kinematic. We have not constructed
any Hamiltonian, local or otherwise, and we have done nothing to cause positivity of the energy. We hope to construct dynamics in a forthcoming paper.

On the other hand we have managed to take into account more of the planar algebra structure and with it we have been led to unitary (projective)
representations of Thompson's groups of $PL$ homeomorphisms of $S^1$ and $[0,1]$ which play the role of \diff in our not yet continuum limit.
The idea is to use an element of $R\in P_n$ for $n>2$ to embed Hilbert spaces associated with finitely many points into each other by grouping together
$n-1$ ``spins'' on one scale into a single spin on a more coarse scale. This is just what is done in block spin renormalisation. 

Although this block spin idea does not introduce dynamics, we will see that it does produce interesting unitary representations of $T$ and $F$. In particular
these representations do depend on the planar algebra data used to construct them. A perhaps surprising byproduct arises if one uses ``crossings'' from the Conway knot-theoretic skein
theory and knot polynomial theory. For then the coefficient of the ``vacuum vector'' in the representation is an unoriented link. One of our main results
is that all unoriented links arise in this way. 

Thompson group elements are often defined by pairs of bifurcating rooted planar trees with the same number of leaves. There is a simple construction of the link
from the pair of trees.  The stabilization move (canceling carets- see \cite{CFP}) corresponds to adding a distant unknot to the link. But a pair of trees with
no canceling carets is in fact a well defined description of a Thompson group element and the algorithm we  give, which produces a Thompson group
element from a link, will not introduce any spurious unknots. 

\emph{So, at least for unoriented links, the Thompson group $F$ is just as effective at producing links as the braid groups.}

 But one may use simpler planar algebras, the simplest possible being that of an index 2 subfactor. A natural choice
of $R$ then gives representations that are in fact linearized permutation representations (quasi-regular representations). The stabilizers of the vacuum 
in this representation are subgroups $\overrightarrow F <F$ and $\overrightarrow T< T$ which we call the oriented Thompson groups since a certain surface
constructed from them is always orientable. Sapir has shown the delightful result that $\overrightarrow F$ is in fact a copy of Thompson's group $F_3$ inside
$F(=F_2)$!

Another motivation for this work was to try to exploit the list of subfactors of small index \cite{JMS}.  Typically these factors are quite ``supertransitive'', i.e. the new information
they contain does not appear in $P_n$ until $n$ is $>2$. We would like in future work to extend our theory to use elements of these
higher $P_n$'s to block spins together.

\section{Some notions of planar algebra.} \label{planar}

A planar algebra $\mathcal P$ is a vector space $P_n$, graded first and foremost by $\mathbb N\cup \{0\}$ and admitting multilinear operations indexed
by \emph{planar tangles} $T$ which are subsets of the plane consisting of a large (output) circle containing smaller (input circles). There are also 
non-intersecting smooth curves called strings whose end points, if they have any, lie on the  circles where they are called marked points. Elements of $\mathcal P$ are ``inserted'' into 
the input circles with an element of $P_n$ going into a disc with $n$ marked points, and the result of the operation specified 
by the tangle is in $P_k$ where
there are $k$ marked points on the output circle. In order to resolve cyclic ambiguities, each of the circles of $T$ comes with
a privileged interval between marked points which we will denote in pictures by putting a $\$ $ sign near that interval.

Here is an example of a planar tangle: \\

\hpic{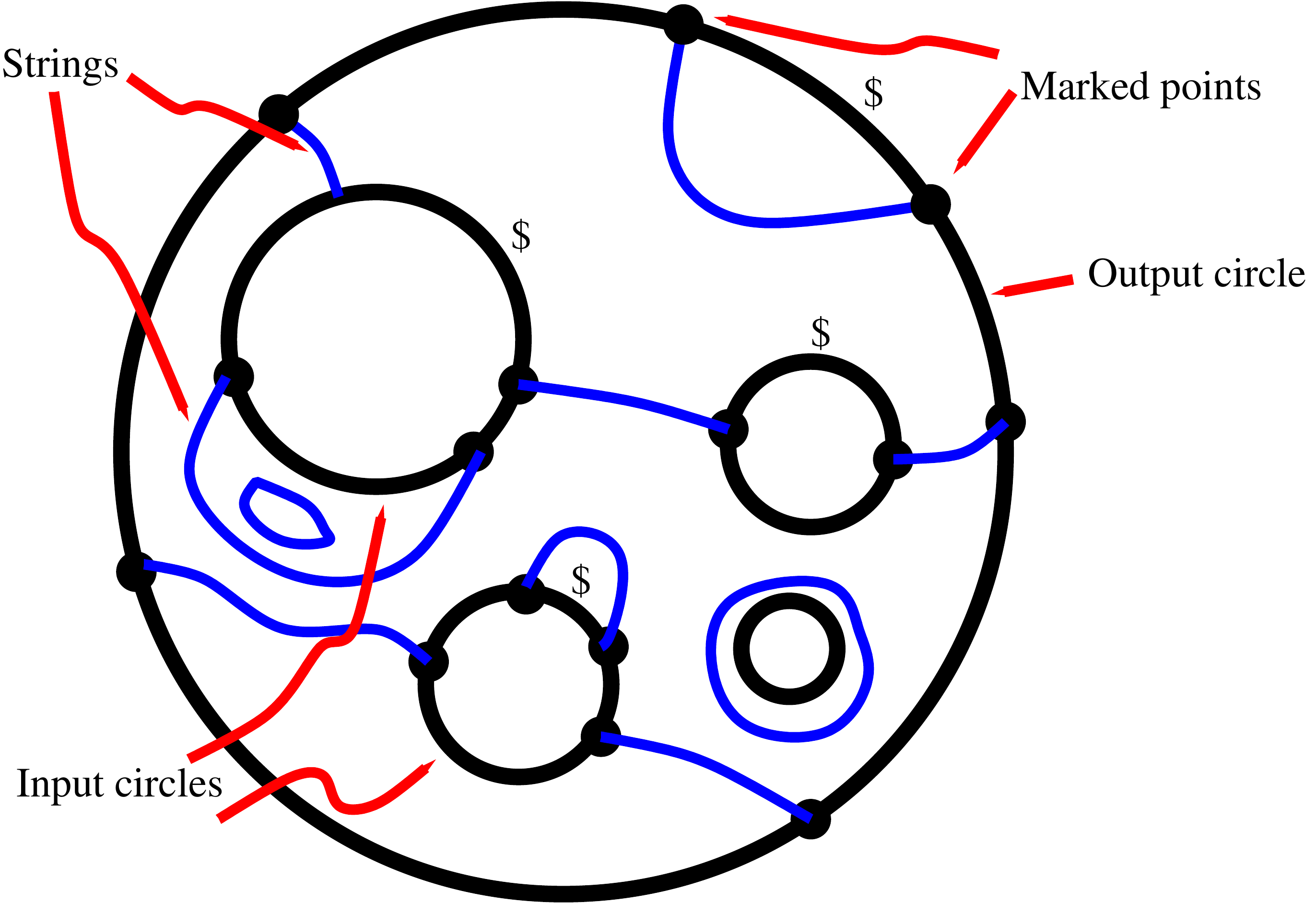} {2in}

The result of the operation indexed by $T$ on elements of $\mathcal P$ is denoted $Z_T(v_1,v_2,\cdots, v_n)$ where there are $n$ input discs. See \cite{jo2} for details.
The operation $Z_T$ depends only on $T$ up to smooth planar isotopy so one has a lot of freedom drawing the tangles, in particular the circles may be replaced
by rectangles when it is convenient. 
Tangles may also be ``labelled'' by actually writing appropriately graded elements of $\mathcal P$ inside some of the internal circles.

\begin{definition} \label{partition function}Given a  planar tangle $T$ all of whose internal circles are labelled by  $v_1,v_2,\cdots, v_n$
 we call  $Z_T(v_1,v_2,\cdots, v_n)$
 the element of $P_k$ which it defines. If $k=0$ and the dimension of $P_{0}$ is one, this may be identified
 with a scalar using the rule that $Z(\rm{empty tangle})=1$.
 \end{definition}

Planar tangles can be glued in an obvious way along input circles and  the operations $Z_T$ are by definition compatible with
the gluing.

Planar algebras come in many varieties according to further decorations of the planar tangles such as oriented strings, labelled strings, coloured 
regions. These decorations are incorporated by further grading of the vector spaces $P_n$. Perhaps the easiest such decoration is a shading, where
the regions of the tangle are shaded or unshaded in such a way that adjacent regions have opposite shadings. This forces the number of 
boundary points of each circle to be even so for historical reasons the grading then becomes half the number of marked points. This can 
cause confusion. When $T$ is shaded, the $\$ $ signs may be in shaded or unshaded
Even in the shaded case it can be cumbersome to drag along all the shadings and $\pm$ signs.
So when it suits us we will ignore them and trust that the reader can fill in the details by copying the instances
where we do, out of necessity, consider the shadings carefully.

 For connections with physics and von Neumann algebras, planar algebras will have more structure, namely an antilinear involution $*$ 
 on each $P_n$ compatible with orientation reversing diffeomorphisms acting on tangles. If $dim P_0=0$ we get a sesquilinear
 inner product $\langle x, y\rangle$ on each $P_n$ given in the introduction. 
 
 A planar algebra will be called \emph{positive definite} if this inner product is. 
 
 In order to accommodate knot-theoretic applications
 we will  not in general restrict to the situation where the $P_n$ are finite dimensional and $\langle, \rangle$ is positive definite.
 
 Our planar algebras will all have a parameter $\delta$ which is the 
value of a homologically trivial closed string which may be removed from any tangle with multiplication by
the scalar $\delta$.

Two examples of planar algebras should be mentioned. The first is the Temperley-Lieb algebra  $TL$ (which has its origins
in \cite{TL} though its appearance here should properly be attributed to \cite{Kff2}, via \cite{Ba2}-see also \cite{J9}) which may be
shaded or unshaded. A basis of $TL_n$ consists of all isotopy classes of systems of non-crossing strings joining
$2n$ points inside the disc. The planar algebra operations are the obvious gluing ones with the rule that any
closed strings that may be formed in the gluing process are discarded but each one counts for a multiplicative
factor of $\delta$.  This planar algebra is positive definite iff $\delta \geq 2$.

The second examples of  planar algebras which we will use are the spin planar algebras. For fixed integer $Q\geq 2$ one considers
$\{1,2,\cdots, Q\}$ as a set of ``spins''. The shaded planar algebra $\mathcal P^{spin}$ is then defined by 
$P^{spin}_0,+=\mathbb C$, $P^{spin}_0,-=\mathbb C^Q$ and for $n\geq 1$,  $P^{spin}_{n,\pm}=\mathbb C^{n}$
(using half the number of marked boundary points).  Shaded tangles then give contraction systems for spin
configurations as explained in \cite{jo2}. This kind of planar algebra is what is used in several 2 dimensional statistical
mechanical models such as the Ising, Potts and Fateev-Zamolodchikov models. It is important to note that closed
strings have parameter either one or $Q$ depending on whether they enclose an unshaded or shaded region respectively.

The various tangles may be used to endow a planar algebra with a large number of algebraic structures such
as filtered algebras, graded algebras and tensor categories. 
We will be interested only in two such structures and their representations, the rectangular and affine categories.
\begin{definition} The \emph{rectangular category} $Rect(\mathcal P)$ of a planar algebra will be 
the (linear) category whose objects are $\mathbb N\cup \{0\}$ (along with other decorations according to shading etc)
and whose set of morphisms from $n$ to $m$ is the vector space $P_{n+m}$, with composition according to the following
concatenation tangle read from top to bottom (which makes each $P_n$ into an algebra):\\

Composition of $x\in Mor(5,3)$ with $y\in Mor(3,7)$: \vpic{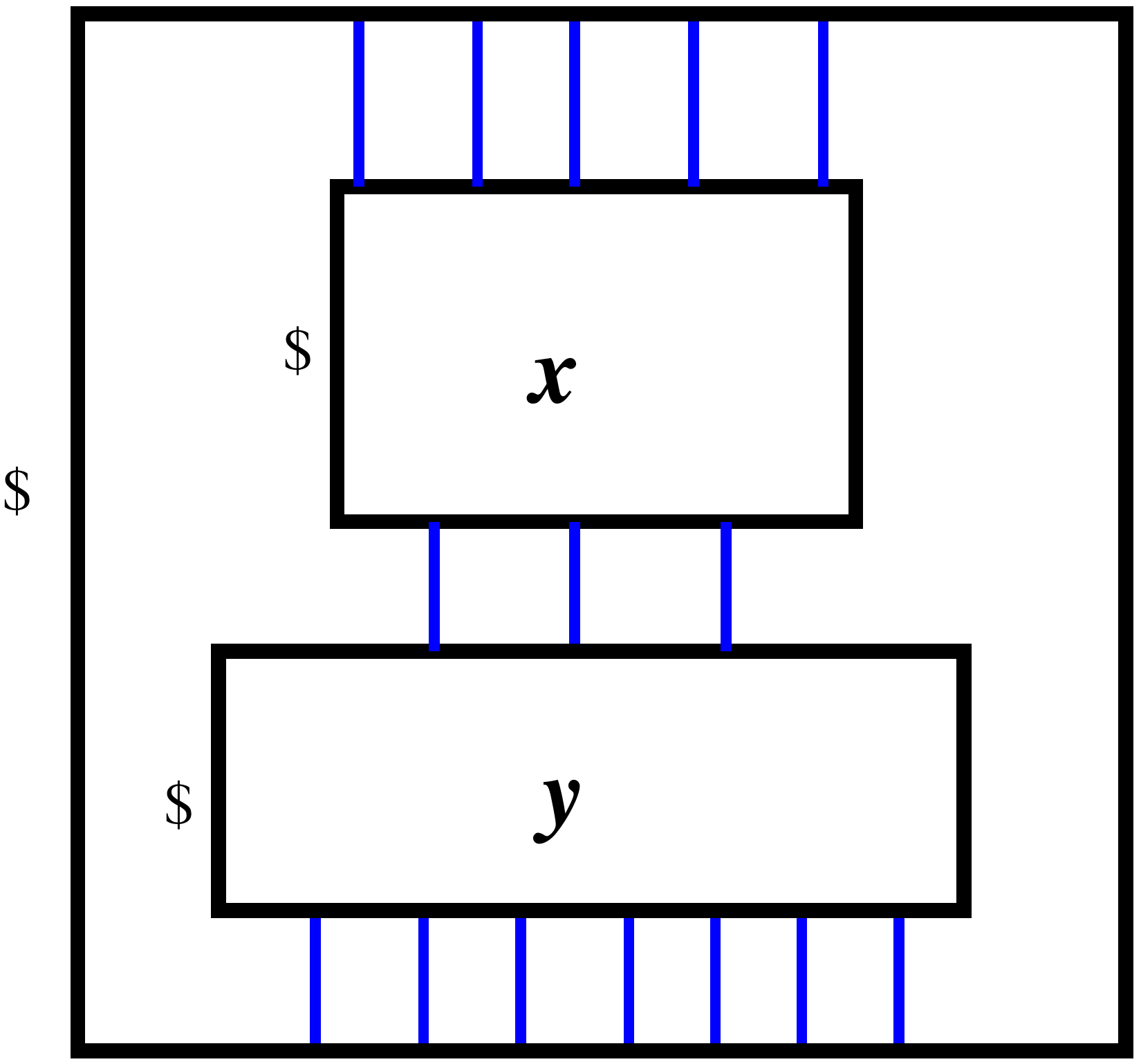} {1.2in}

A representation of $Rect(\mathcal P)$ will be called a rectangular representation.
\end{definition}
A planar algebra always admits a privileged rectangular representation:
\begin{definition} \label{regular} The \emph{regular representation} of the planar algebra $\mathcal P$ is the one with $V_n=P_n$ for all 
$n$ (with $\pm$ in the shaded case) and the action being just concatenation of rectangular tangles.
\end{definition}

The regular representation has a $P_0$-valued invariant inner product given by  $\langle x, y\rangle$ as below. 

\hpic{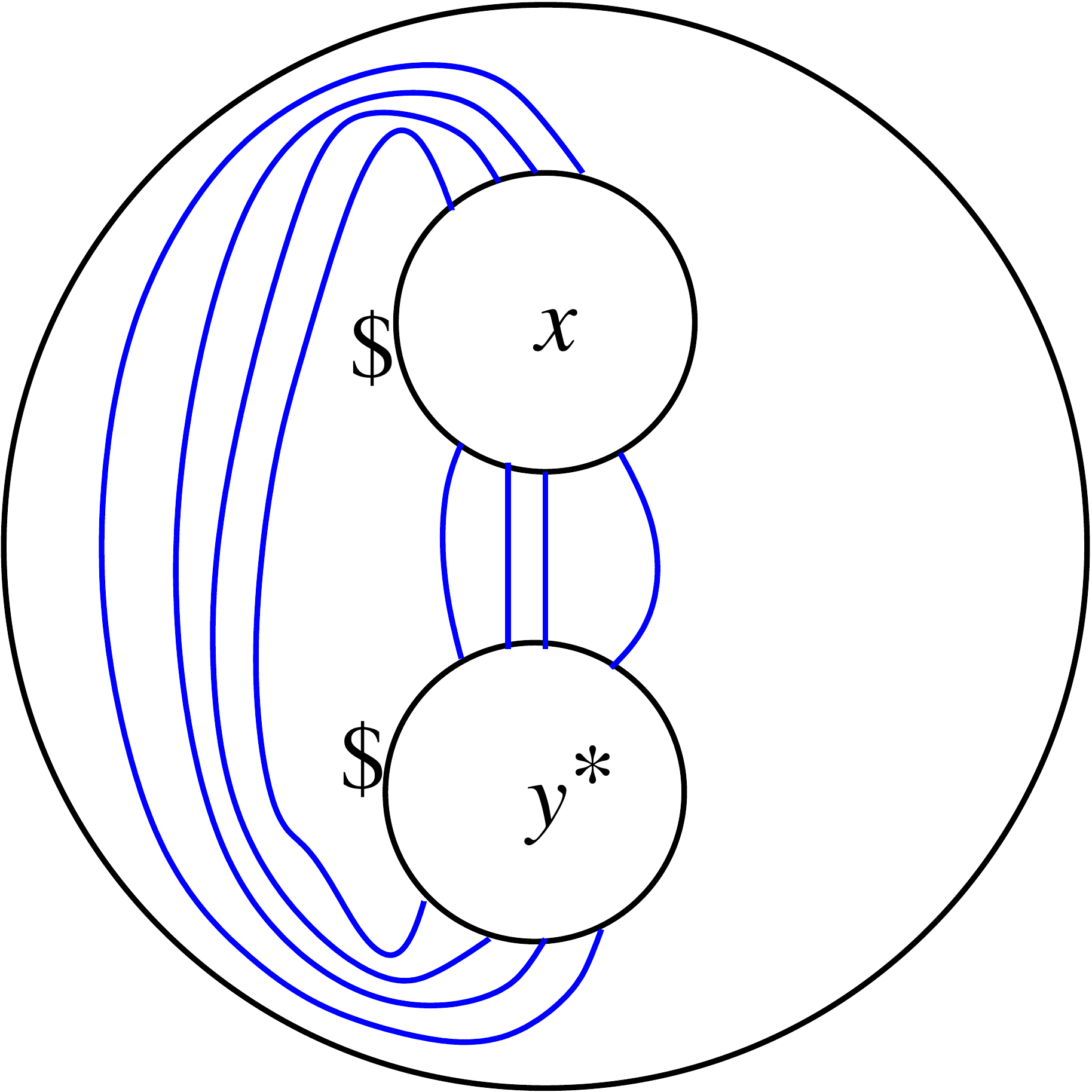} {1.5in}

where  we have been careful to close the tangle with
strings to the left as we will be dealing with a non spherically invariant planar algebra below.

In the subfactor case any irreducible rectangular representation is a subrepresentation of the regular representation but in general we may define
rectangular representations by minimal projections $p$ in the algebras $P_k$, the vector space $V_n$ of the representation being equal to 
$pP_n $ with $p$ considered as the element $pe_{k+1}e_{k+3}\cdots e_{n-1}\in P_n$  for $k\geq n$ and $V_k=0$ for $k<n$, with action of morphisms by
concatenation. For example:\\

Action of $y\in Mor(5,7)$ on $xp\in V_5$:
\vpic{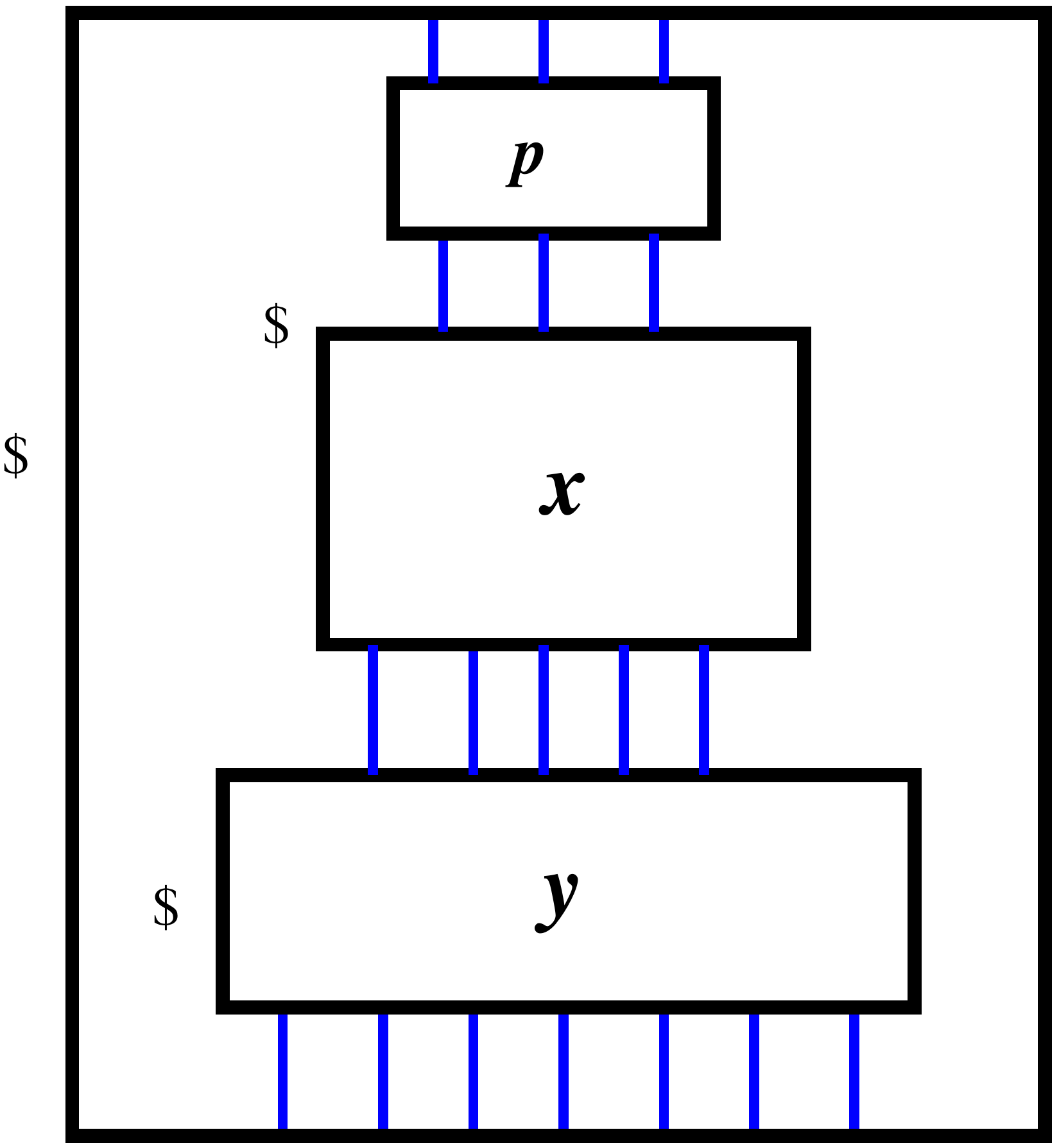} {1.6in}

\begin{definition} The \emph{affine} category $Aff(\mathcal P) $ is the (linear) category whose objects are sets $\bar m$ of $m$ points
on the unit circle in $\mathbb C$ (along with other decorations according to shading etc)
and whose vector space of  morphisms from $\bar m$ to $\bar n$ is the set of linear combinations of labelled tangles 
(with marked boundary points $\bar m\cup \bar n$) between the unit circle and a circle of
larger radius  modulo any relations in $\mathcal P$ 
which occur within contractible discs between the unit circle and the larger circle. Composition of morphisms comes from rescaling and gluing 
the larger circle of the first morphism to the smaller circle of the second. 

\end{definition}
 
Use of $\bar m$ adds to clutter so we will abuse notation by using just $m$ for an object of $Aff(\mathcal P) $ with $m$ points, though when 
we need it later on we will be quite careful to specify what sets of $S^1$ we are talking about.

One needs to be careful with this definition (see \cite{jo3},\cite{gl}). In a representation of $Aff(\mathcal P) $, morphisms may be changed by planar  isotopies 
without affecting the action, but the
isotopies are required to be the identity on the inner and outer circles.   Thus the tangle of rotation by 360 degrees does not necessarily act by the
identity in a representation of $Aff(\mathcal P)$.

The representations we will consider of $Aff(\mathcal P)$ are called lowest weight modules and may be defined as in \cite{jo3} by taking a representation $W$
of the algebra $Mor(n,n)$ for some $n$ (the ``lowest weight'') and inducing it in the obvious way. This may cause problems with positive definiteness but it is known that
subfactor planar algebras possess a host of such representations. The vector spaces $V_k$ of such a lowest weight representation  are zero 
if $k\leq n$ and spanned by  diagrams consisting of a vector $w\in W$ inside a disc with $n$ marked points, surrounded by a labelled planar tangle
of $\mathcal P$ with $k$ marked points on the output circle. 
 
 Here is a vector $w$ in  a $V_6$ created by the action of an affine morphism on $v$ in the lowest weight space $V_2$: 
  and the action of a morphism in $Aff(\mathcal P)$ on it:\\
 
 \hpic{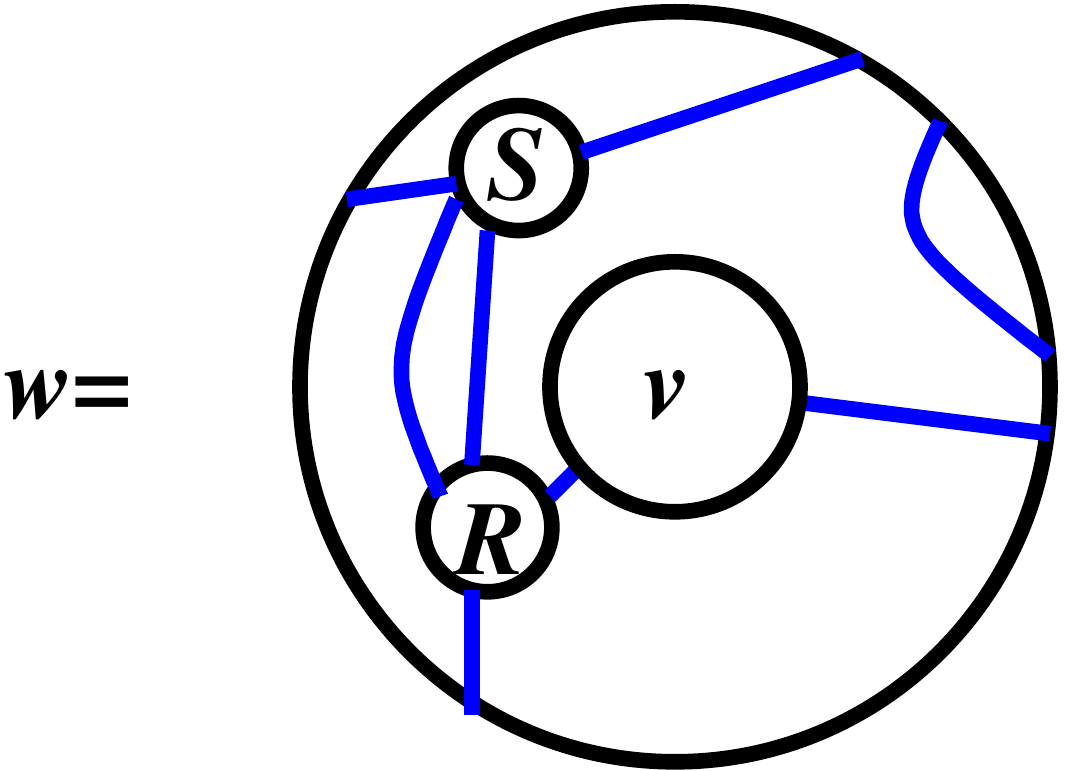} {1in}

 and  here is a diagram with  the action of a morphism in $Mor(6,4)$ on it:\\
 
 \hpic{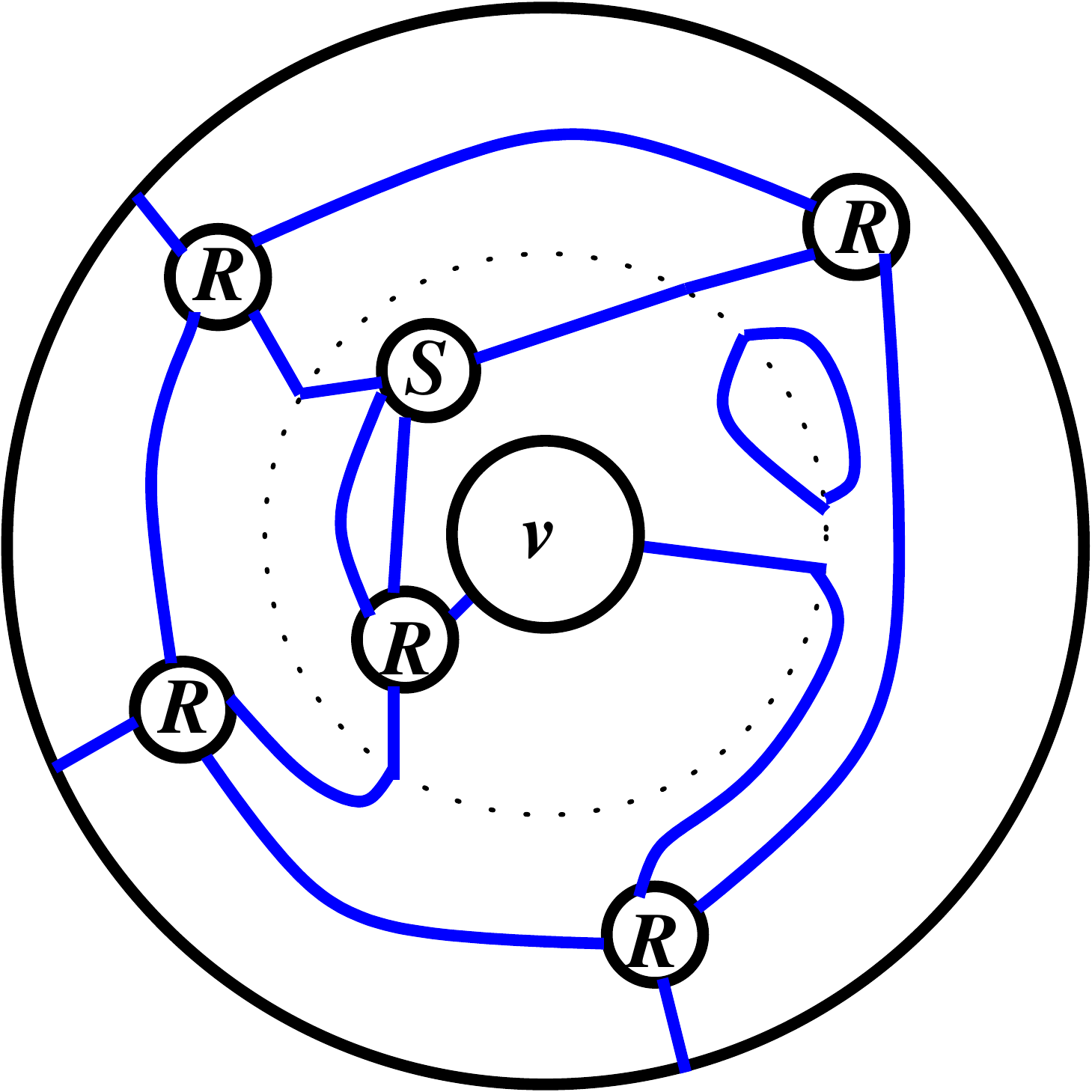} {1.5in}

As with rectangular representations the planar algebra itself defines an affine representation simply by applying annular labelled
tangles to elements of $\mathcal P$ but now the representation  is irreducible  and plays the role of the trivial representation.

In the TL case which is what we will mostly consider, irreducible lowest weight representations are parametrized by their lowest weight (the smallest $n$ 
for which $V_n$ is non-zero), and a complex number of absolute value one which is the eigenvalue for the rotation tangle.
The case $n=0$ is exceptional and the rotation is replaced by the tangle which surrounds an element $v$ of $V_0$ by a circular string (two strings in
the shaded case). $v$ is an eigenvector for this tangle and there are some restrictions on the eigenvalue $\mu$-see \cite{jo3},\cite{JR}.
The case where $\mu=\delta$ is precisely the trivial representation. In this case the vector $v\in V_0$ is the empty diagram so never features
in pictures.

\section{The Hilbert spaces.}
\subsection{One-box version.}\label{1box}
Let $P=\{P_{n}\}$ be an unshaded positive definite planar algebra.
Suppose we are given an element $R\in P_1$ with $R^*R=1$. Suppose we are also given an affine
representation $V=\{V_n\}$of $P$ with lowest weight space $V_k$. 

With no more data than this we can make a naive construction of a Hilbert space ``continuum'' limit.

\begin{definition}. If $\mathcal F$ is the directed set of finite subsets of $S^1$, directed by
inclusion, define the directed system of Hilbert spaces $F\mapsto V_F$ for $F\in \mathcal F$,
with inclusion maps $\iota _{F_1}^{F_2}$ (for $F_1\leq F_2$) defined by the following affine tangles:
\end{definition}

\hpic {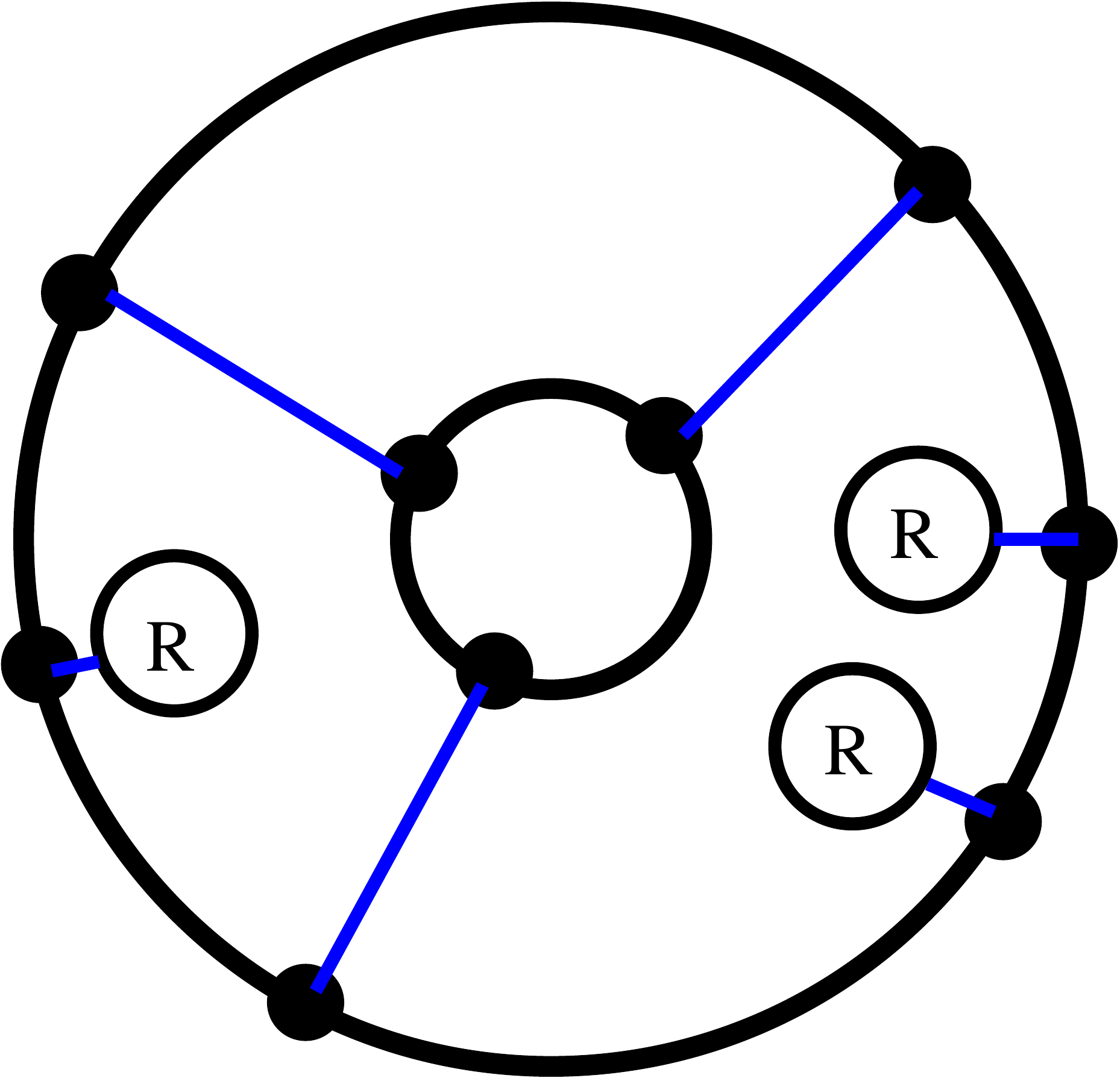} {1.5in}

In this example there are 3 points in $F_1$ and 6 in $F_2$. The points in $F_1$ are to be joined
to themselves on the outside boundary by radial straight lines, and the points in $F_2\setminus F_1$
are attached to an instance of $R$. Since $R^*R=1$ these $\iota$'s define isometric embeddings
of $V_{F_1}$ into $V_{F_2}$ which are obviously compatible with inclusions. Thus the direct limit of these Hilbert spaces can be itself completed to
form a Hilbert space $\mathfrak H=\lim_{F\in \mathcal F} V_F$.

This Hilbert space $\mathfrak H$ obviously carries a projective unitary representation of $Diff(S^1)$ 
and one can define local algebras $\mathcal A (I)$ in such a way that the first three properties of a conformal 
net defined in the introduction are satisfied. But the situation is caricatural. The Hilbert space is not separable
and the action of the diffeomorphism group is hopelessly discontinuous. 

One may remedy the lack of separability by restricting consideration to the net of finite subsets of a countable set
like the rationals but this does not make things any more continuous. If we restrict to the dyadic rationals we 
obtain unitary representations of Thompson's group $T$ and by restriction to Thompson's group $F$.
But we can show that all these representations of $F$ are direct sums of representations on
$\ell^2(\{\mbox{subsets of size }k \mbox{ of the dyadic rationals in  } [0,1]\})$ So this construction has
given very little, both from a physical and mathematical point of view. 

Undeterred, we will make a similar construction which uses more planar algebra data than just an
element of $P_1$. This will at least give something on the mathematical side, namely some interesting
representations of the Thompson groups.

\subsection{Two-box version.} 

\begin{definition}\label{normalized} Let $P=\{P_{n,\pm}\}$ be a subfactor  planar algebra. An element $R\in P_2$ will be said to be
\emph{normalized} if \\

\qquad \qquad \vpic{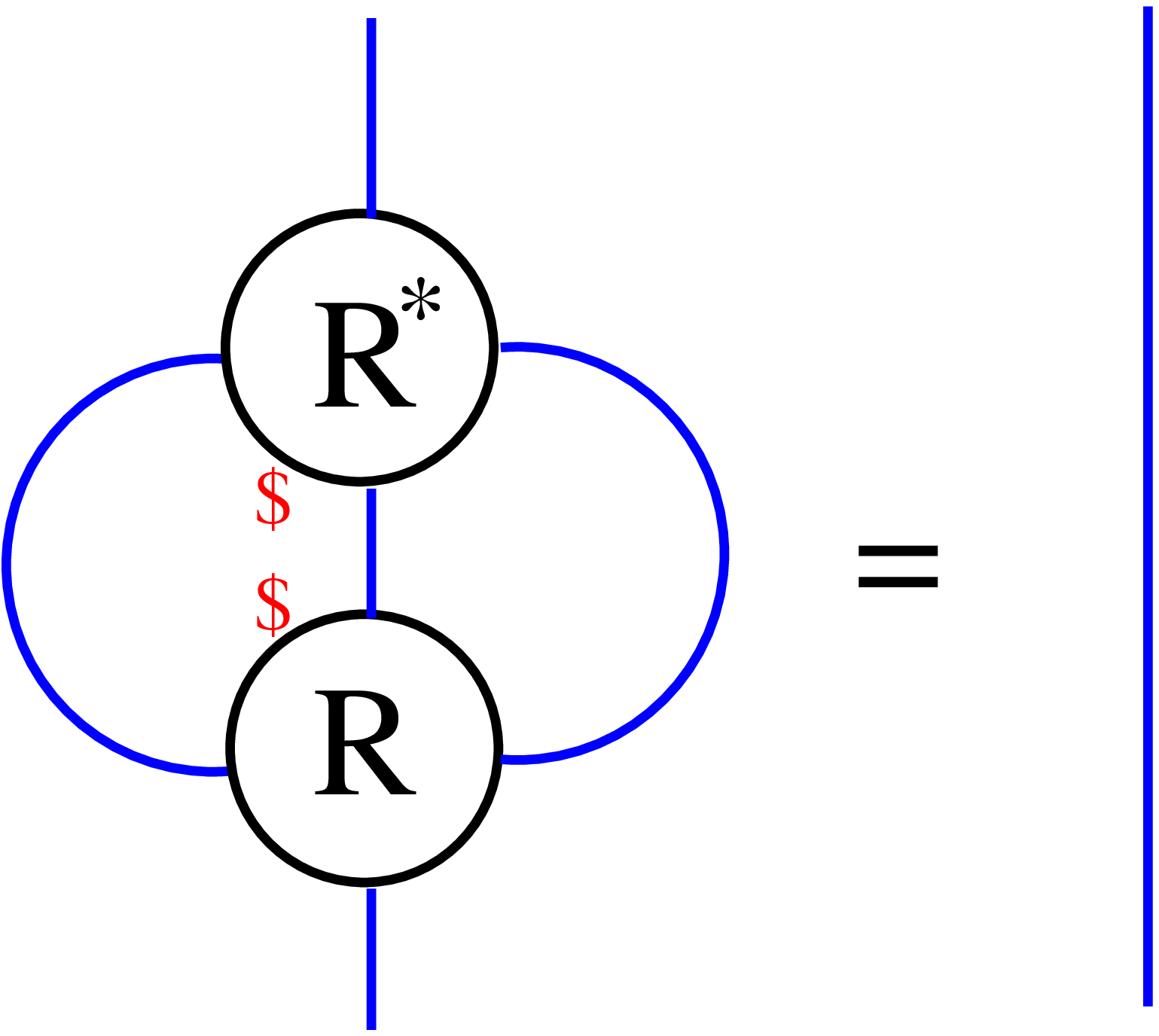} {1.5in}
\end{definition}

 For instance if $\dim P_1=1$ one may take any vector  $v$ in $P_2$  with $\langle v,v\rangle = \delta$.

 \subsubsection{A directed set construction.}
The construction of the Hilbert space in this section was inspired directly by the idea of block spin
renormalisation where groups of spins on a lattice at one scale are combined to give blocks which from
single spins on the same lattice at a coarser scale. Mathematically it has proven  better to give a more flexible construction  where
 one does not have to be block together all the spins at once. This 
approach was suggested by Dylan Thurston.

An interval $[a,b]\subseteq [0,1]$ is said to be \emph{standard dyadic} if there are  non-negative integers $n,p$ such that
$\displaystyle a=\frac{p}{2^n}$ and $\displaystyle b=\frac{p+1}{2^n}$. See \cite{CFP}.

As is often the case in mathematics, a partition of a closed interval $I$ will be a (finite) family of closed subintervals with disjoint
interiors whose union is $I$.And, given two partitions $\mathcal I$ and $\mathcal J$ of $[0,1]$ into closed intervals
we say $\mathcal J$ is a refinement of $\mathcal J$, $\displaystyle \mathcal I \lesssim \mathcal J\mbox{, if for all }I\in \mathcal I, I=\underset{J\subseteq I, J \in \mathcal J}\bigcup J$

\begin{definition}  The directed set $ \ds $ is the set of all partitions of $[0,1]$ into \emph{standard dyadic intervals}, with
order as above. 
For each $\mathcal I\in \ds$  we let $M(\mathcal I)$ be the set of midpoints of the intervals $I\in \mathcal I$, and $E(\mathcal I)$
to be  $M(\mathcal I)\cup E$ where $E$ is the set of all endpoints of intervals in $\mathcal I$ except $0$ and $1$.

\end{definition}

The relation $\{[0,1]\}\lesssim \mathcal I$ for $\mathcal I  \in \ds$ can be represented as a planar rooted bifurcating tree in $[0,1]\times[0,1]$ whose leaves
are $M(\mathcal I)\times\{ 1\}$ and whose root  is $(0,\frac{1}{2})$ in a way described in \cite{CFP}. Thus since a standard dyadic interval is just
a rescaled version of $[0,1]$, a pair $\mathcal I \lesssim \mathcal J$ can be represented as a planar forest whose roots are $M(\mathcal I)\times \{0\}$ and
whose leaves are $M(\mathcal J)\times \{1\}$. We illustrate below:\\

\qquad Tree: \qquad \vpic{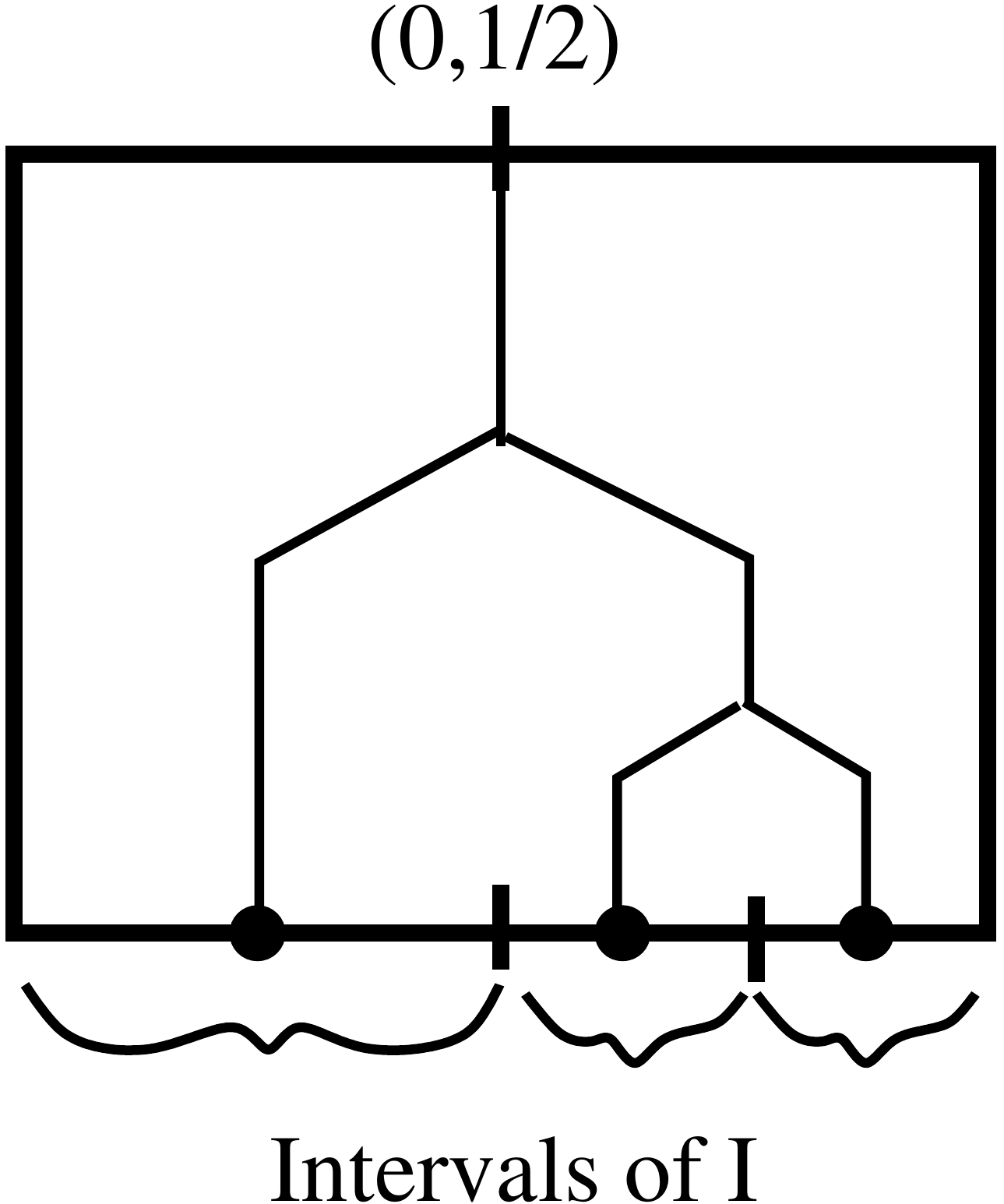} {1.5in}  \\

\qquad Forest: \qquad \vpic{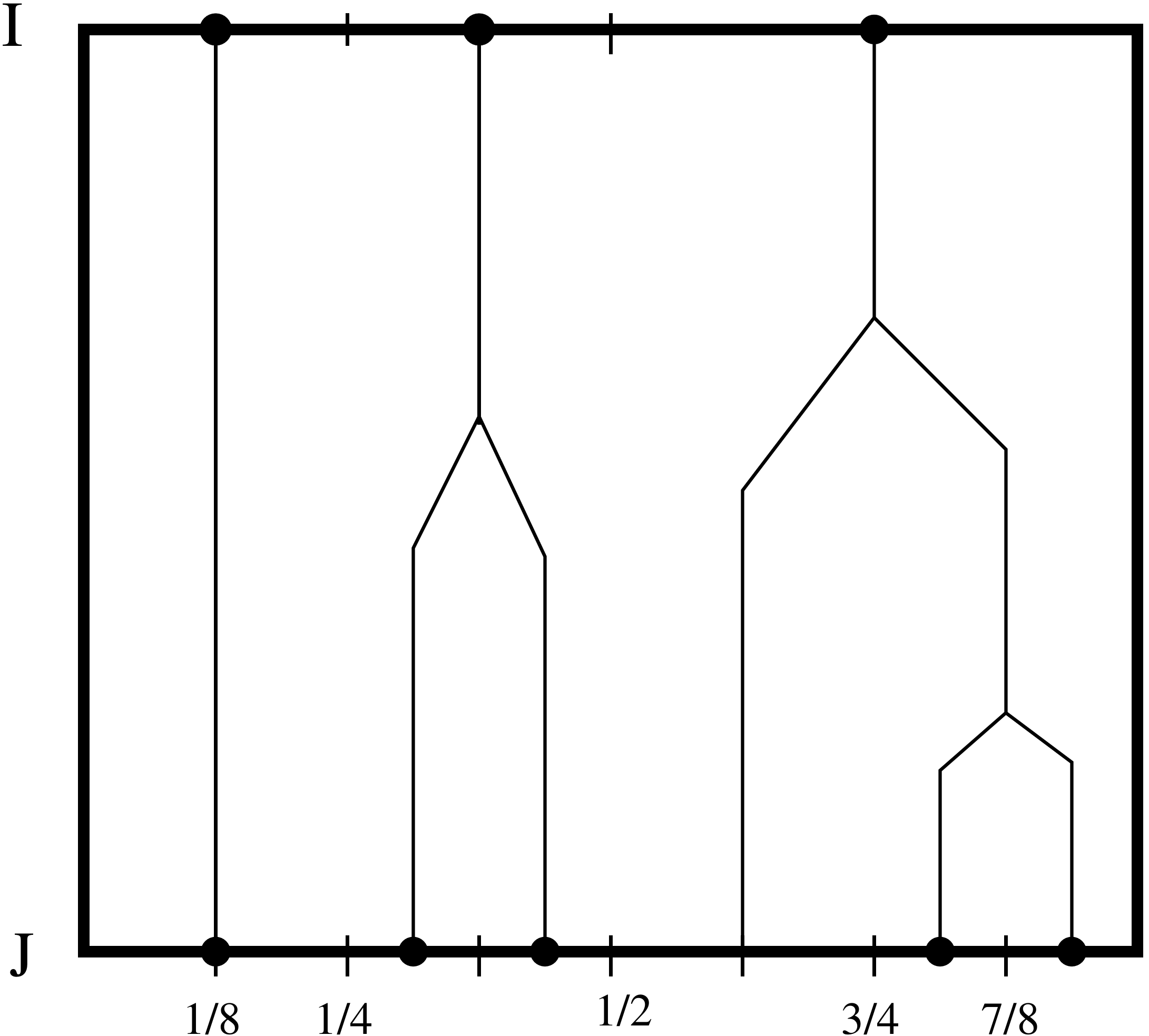} {1.5in} 

\begin{definition} For $\mathcal I \lesssim \mathcal J \in \ds$ we call $F_{\mathcal J}^{\mathcal I}$ the forest defined above. \end{definition}

Note that if $\mathcal I \lesssim \mathcal J $ and $\mathcal J \lesssim \mathcal K$  then  $F^{\mathcal I}_{\mathcal K}$ is obtained
by stacking  $F_{\mathcal K}^{\mathcal J}$ underneath of $F_{\mathcal J}^{\mathcal I}$.

We will define  nets of vector spaces on $\ds$ using rectangular
representations of a planar algebra as defined in section \ref{planar}

\begin{definition}\label{xi} The element $\xi\in P_1$ will be the element given by a tangle with a single straight line connecting
the two boundary points. The rectangular representation $\Xi$ will be the sub representation of the regular representation
generated by $\xi$.
\end{definition}

\begin{definition} Given a planar algebra $\mathcal P$ we define  $Cat(P)$ to be the category with objects $\mathcal I$ for $\mathcal I\in \ds$ and 
morphisms $Mor(\mathcal I,\mathcal J)$ for $\mathcal I\lesssim \mathcal J$
being labelled planar tangles of $\mathcal P$ inside rectangles with marked boundary points $E(\mathcal I)$ on the top and $E(\mathcal J)$ on the bottom.
\end{definition}
Note that a rectangular representation $V$ of a planar algebra gives a representation of $Cat(\mathcal P)$ by taking $Z$ of  labelled tangles.

Now for each $\mathcal I \lesssim \mathcal J$ in $\ds$ we complete the diagram $F_{\mathcal J}^{\mathcal I}$ to a labelled tangle 
$T_{\mathcal J}^{\mathcal I}\in Mor(\mathcal I,\mathcal J)$  of $Cat(\mathcal P)$
inside $[0,1]\times[0,1]$  defined by replacing each 
trivalent vertex of the forest $F_{\mathcal J}^{\mathcal I}$ by an instance of $R$, with a vertical straight line joining the top of the disc containing $R$ to the corresponding common boundary
point of two intervals of $\mathcal J$ directly below it, and vertical lines connecting all the end points of intervals of $\mathcal I$ to those 
endpoints not connected to discs containing $R$'s,  thus:

\qquad \hpic{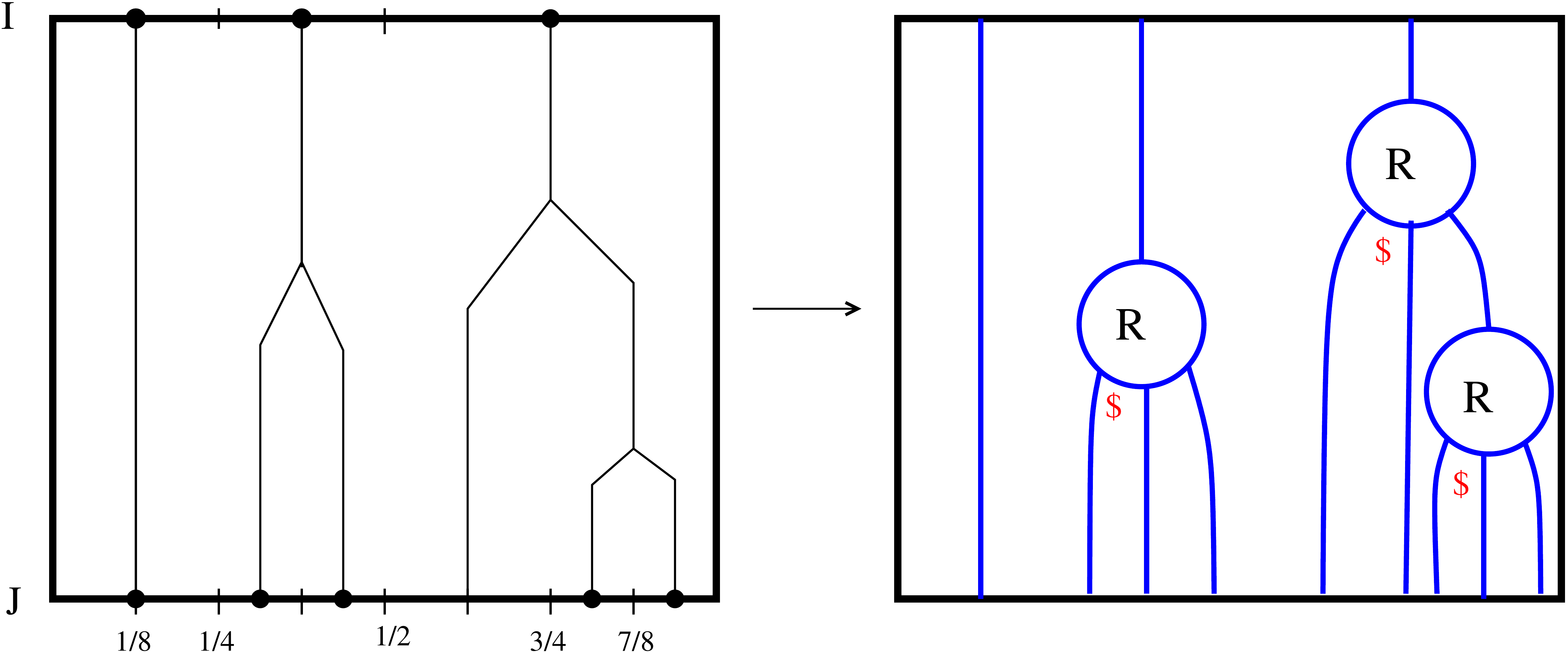} {1.7in}  \\

\hspace {1in} $F_{\mathcal J}^{\mathcal I}$ \hspace{1in}  $\rightarrow$ \hspace{1in}
$T_{\mathcal J}^{\mathcal I}$

\begin{definition} \label{blowup}Given $\mathcal P$ and a rectangular representation $V=(V_n)$ of it, and an element $R\in P_2$ as above we define the directed system of vector spaces
$\mathcal I\mapsto  H(\mathcal I)$ for $\mathcal I\in \ds$ by $H( \mathcal I)=V_{|E(\mathcal I)|}$ and for 
$\mathcal I \lesssim \mathcal J$ in $\ds$ we use the tangle $T_{\mathcal J}^{\mathcal I}$ to define a map from 
$H( \mathcal I)$ to $H( \mathcal J)$ as required by the definition of a directed system.
 \end{definition}

Note that the functorial property of a directed system is that if $\mathcal I \lesssim \mathcal J $ and $\mathcal J \lesssim \mathcal K$
then $T_{\mathcal K}^{\mathcal J}\circ T_{\mathcal J}^{\mathcal I}=T_{\mathcal K}^{\mathcal I}$ which is true since stacking of
the $F_{\mathcal J}^{\mathcal I}$ corresponds to composition of the tangles 
in $Cat(\mathcal P)$.

\begin{definition}\label{hilbert} The direct limit vector space $\mathcal V=\lim_{\mathcal I \in \ds} H_{\mathcal I}$ is called the \emph{dyadic limit}
of $V$. If $\mathcal P$ is a subfactor planar algebra, $\mathcal V$ becomes a pre-Hilbert space since all the $T_{\mathcal J}^{\mathcal I}$
define isometries. Its Hilbert space completion is called the \emph{dyadic Hilbert space} $\mathcal H_{R,V}$ of $V$.
\end{definition}

Thus elements of $\mathcal V$ are equivalence classes of elements of $\underset{\mathcal I}\coprod H(\mathcal I)$ with $x\in H(\mathcal I)$ equivalent
to $y \in H(\mathcal J)$ iff there is a $\mathcal K$, $\mathcal I \lesssim \mathcal K$ and $\mathcal J\lesssim \mathcal K$ with $T_{\mathcal K}^{\mathcal I}(x)=T_{\mathcal K}^{\mathcal J}(y)$.
The $T_{\mathcal I}^{\mathcal J}$ are all injections so each $H(\mathcal I)$ is a vector subspace of $\mathcal V$.

Note that if ${\mathcal I_n}$ is the partition into all standard dyadic intervals of
width $2^{-n}$,
 the sequence $H_{\mathcal I_n}$ is cofinal in $\ds$. Going from $H_{\mathcal I_{n+1}}$ to $H_{\mathcal I_n}$ is supposed to represent
a "block spin renormalization" step where the spins sitting at the points $E(\mathcal I_{n+1})$ are grouped together into spins sitting at
the points $E(\mathcal I_{n})$.

\subsubsection{Unitary representations of Thompson's group $F$ from rectangular $\mathcal P$-modules.}

We will define representations of the Thompson group $F$ on the dyadic limit of a rectangular representation 
$V$ of a planar algebra for normalised $R$  ( \ref{normalized}). They will extend to unitary representations in 
the case of a subfactor planar algebra.

Observe first that  $F$ acts on partitions preserving the order, but does not  preserve $\ds$ since an element of $F$ does not necessarily send a
standard dyadic interval to another one. 

\begin{definition} We say $\mathcal J\in \ds$  is \emph{good for} $g\in F$ if  \, $g(\mathcal J)\in \ds$. \\
And in that case 
we define the tangle $g_{\mathcal J} $ as the rectangular tangle in $Mor( \mathcal J,g(\mathcal J))$ with straight lines connecting all the points in $E(\mathcal J)$
to their images under $g$ thus:\\

\hspace{1in} \hpic {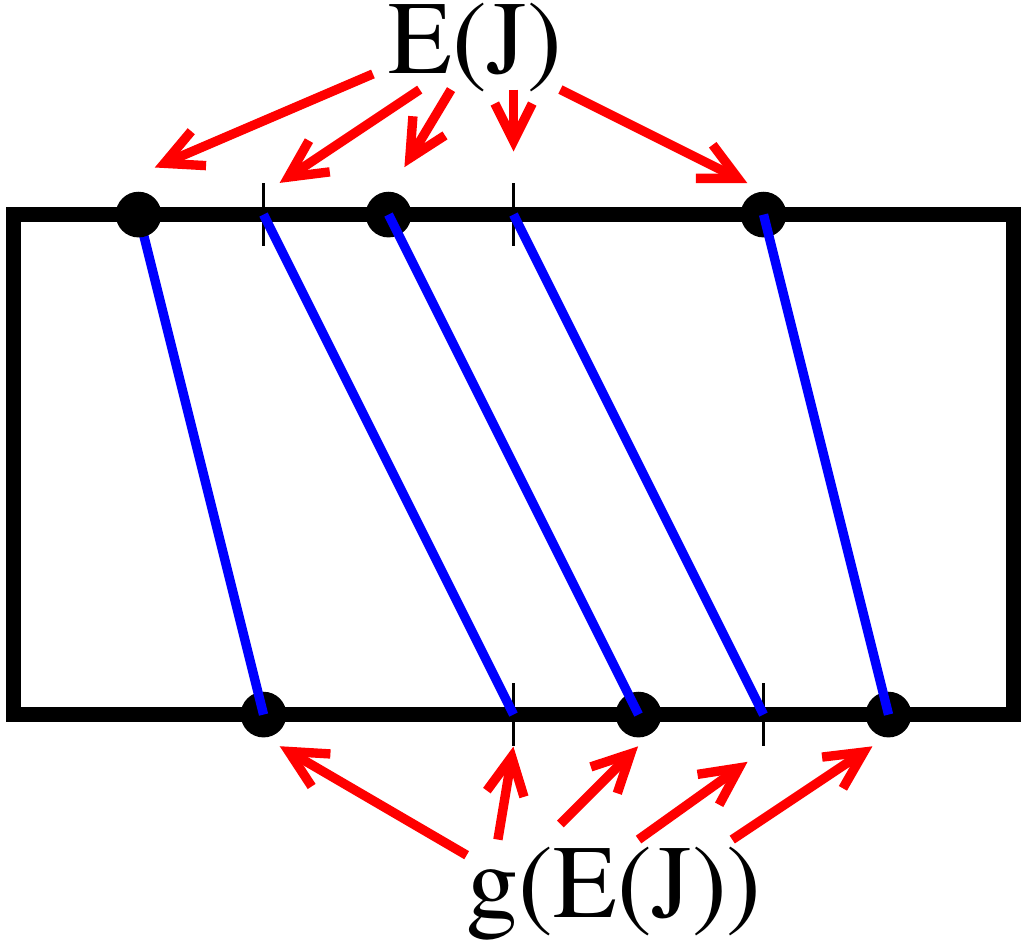} {1.5in}

\end{definition}

Given  $g\in F$ it follows from \cite{CFP} that there is a $\mathcal J$ which is good for it. Also since $g$ does nothing but scale and translate on
each $J\in \mathcal J$ it follows that $g(E(\mathcal J))=E(g(\mathcal J))$ and $g(M(\mathcal J))=M(g(\mathcal J))$.

\begin{lemma} \label{main} If  $\mathcal J$ is good for $g$ and $\mathcal J\lesssim \mathcal K$ then $\mathcal K$ is good for $g$ and
 $ g_{\mathcal K}\circ T_{\mathcal K}^{\mathcal J} $ is isotopic as a labelled tangle to $T_{g(\mathcal K)}^{g(\mathcal J)}\circ g_{\mathcal J} $.
\end{lemma}
\begin{proof}
We focus on a fixed interval $J\in\mathcal J$ and build the isotopy up interval by interval. The rectangle in $Cat(\mathcal P)$ with $J$ at the top and the
various $K\in \mathcal K$ whose union is $J$ at the bottom contains a trifurcating tree whose leaves are alternately midpoints and endpoints
of those $K$'s. Slide that tree along the two connected straight line segments connecting the midpoint of $J$ to the midpoint of $g(J)$. 
The tree is now contained as smooth curves in the trapezoid in the bottom rectangle with $J$ at the top and $g(J)$ at the bottom. Repeat
the procedure for each $J\in \mathcal J$ to obtain a lower  rectangle with a forest isotopic to that of $T_{\mathcal J}^{\mathcal K}$
and an upper rectangle containing only vertical straight line segments connecting the points of $E(\mathcal J)$ at the top to points at 
the bottom of the top rectangle. See the second picture in the figure below. Now stretch and squeeze the bottom of the top
rectangle so that the intervals of $\mathcal J$ become those of $g(\mathcal K)$. Extending that isotopy to the whole picture, 
The upper rectangle becomes isotopic to $T_{g(\mathcal K)}^{g(\mathcal J)}$.

 \hpic{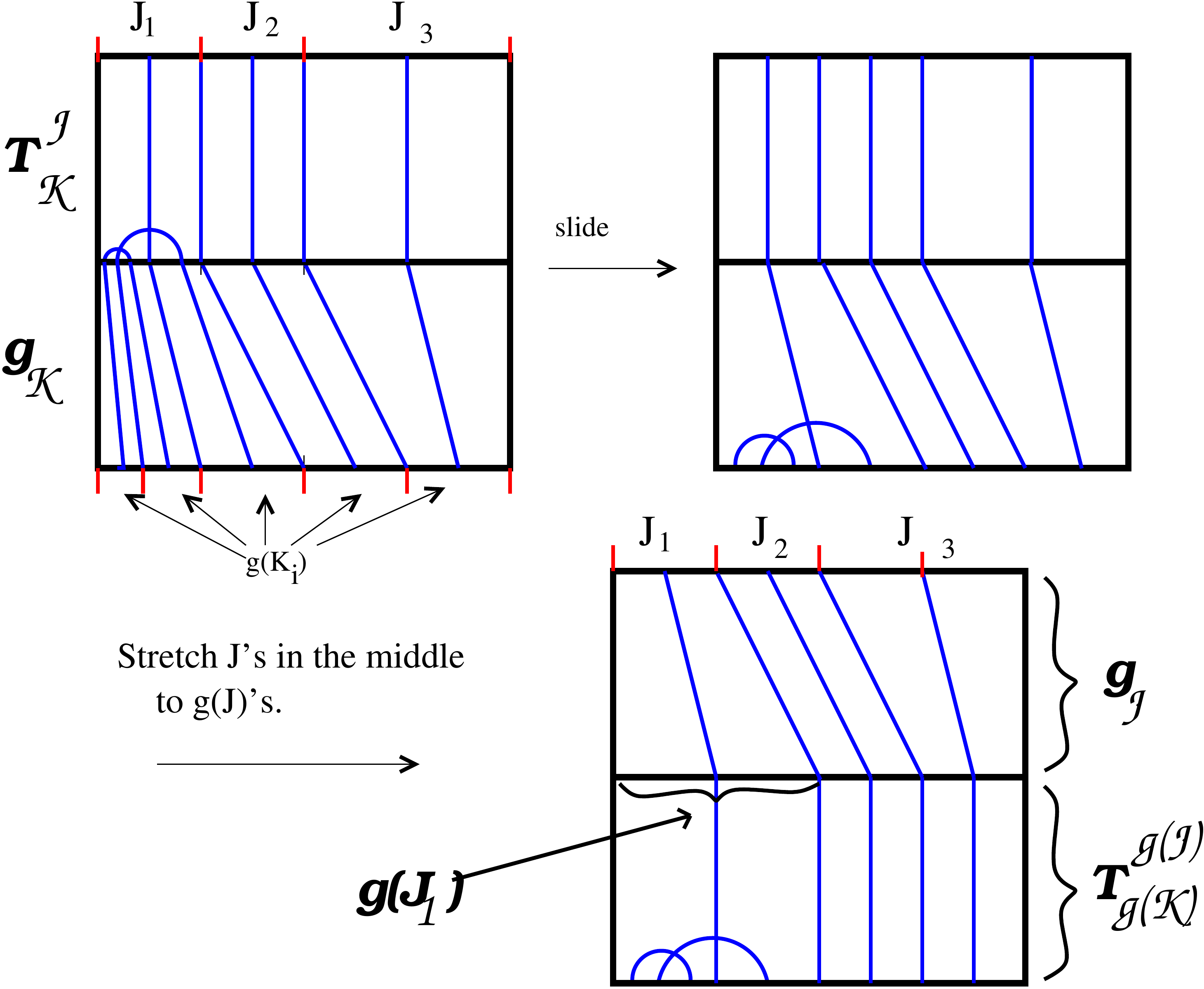} {3.5in}  \\
\end{proof}

Now suppose $g\in F$ and $v\in H(\mathcal I)$ for some $\mathcal I\in \ds$. Choose some $\mathcal J\gtrsim \mathcal I$ which is good for $g$.
Set $\rho_{\mathcal J}(v)=g_{\mathcal J}\circ T_{\mathcal J}^{\mathcal I}(v)\in H(g(\mathcal J))$ where we recall that $Cat(\mathcal P)$  acts
on the $V_n$.

\begin{proposition}\label{independence}The element $\rho_{\mathcal J}(v)$ is independent, in the dyadic limit, of $\mathcal J$ chosen as above.
\end{proposition} 
\begin{proof}
If $\mathcal J '$ is another choice with $\mathcal J '\gtrsim \mathcal I$ and which is good for $g$ then there is a $\mathcal K\in \ds$ with
$\mathcal K\gtrsim \mathcal J$ and $\mathcal K\gtrsim \mathcal J '$. In the dyadic limit, $T_{g(\mathcal K)}^{g(\mathcal J)}(\rho_{\mathcal J}(v))=\rho_{\mathcal J}(v)$
and $T_{g(\mathcal K)}^{g(\mathcal J')}(\rho_{\mathcal J '}(v))=\rho_{\mathcal J '}(v)$. But by \ref{independence} and the directed system property we see that these are both equal to
$g_{\mathcal K}(T_{\mathcal K}^{\mathcal I}(v))$.
\end{proof}

\begin{corollary} Suppose $v\in H({\mathcal I})$ and $w\in H({\mathcal I'})$ are equal in the dyadic limit. Then $\rho_{\mathcal J}(v)=\rho_{\mathcal J'}(w)$
for any appropriate choices of $\mathcal J$  and $\mathcal J'$ in the dyadic limit.
\end{corollary}

\begin{proof}
By definition there is a $\mathcal K \in \ds$ with $T_{\mathcal K}^{\mathcal I}(v)=T_{\mathcal K}^{\mathcal I'}(w)$. Clearly we may also assume that
$\mathcal K$ is good for $g$ so by what we have just seen we may use $\mathcal K$ as both $\mathcal J$ and $\mathcal J'$ to calculate 
$\rho_{\mathcal J}(v)$ and $\rho_{\mathcal J'}(w)$.
\end{proof}
\begin{definition} For $g\in F$ and $v$ in the dyadic limit space $\mathcal V=\lim_{\mathcal I \in \ds} H_(\mathcal I)$  we define $\pi(g)(v)=\pi_{V,R}(g)(v)$ by
$\pi(g)(v)= \rho_{\mathcal J}(w)$ for some choice of $w\in H({\mathcal I})$ for some $\mathcal I\in \ds$ representing $w$ in the dyadic limit.

\end{definition}

\begin{proposition} The map $g\rightarrow \pi(g)$ defined above is a group representation which preserves the inner product
if $\mathcal P$ is a subfactor planar algebra and hence extends to a unitary representation of $F$ on $\mathcal H$.
\end{proposition}
\begin{proof} Unitarity follows immediately from \ref{normalized}. To show $\pi(g)\pi(h)=\pi(gh)$ observe first that for any  $\mathcal J$ which is
good for $h$ and for which $h(\mathcal J)$ is good for $g$, then $\mathcal J$ is good for $gh$ and
 it is clear that the composition  tangle $g_{h(\mathcal I)}\circ h_{\mathcal I}$ is isotopic
to $(g\circ h)_{\mathcal I}$. But to define $\pi(gh)(v)$ take $v\in H(\mathcal I)$ and choose $\mathcal J\gtrsim \mathcal I$ so that it is good for $h$ and $h(\mathcal J)$ is 
good for $g$. Then
 $$\pi(g)\pi(h)(v)=g_{h(\mathcal J)}(h_{\mathcal J}\circ T_{\mathcal J}^{\mathcal I}(v))=(g\circ h)_{\mathcal J}\circ T_{\mathcal J}^{\mathcal I}(v)=\pi(gh)(v)$$
\end{proof}

\subsubsection{Projective unitary representations of Thompson's group $T$ from affine $\mathcal P$-modules.}

 Suppose we are also given an affine
representation $V=\{V_{n,\pm}\}$ of $P$ with lowest weight vector $\psi\in V_{\ell,\pm}$.Let us use $P_n$ and $V_n$  to denote $P_{n,+}$ and $V_{n,+}$ respectively.

\vskip 20pt
 \begin{definition} Let $\widetilde T$ be the group of homeomorphisms $g$ of $\mathbb R$ with $g(x+1)=g(x)+1$ which are
 piecewise affine, differentiable on intervals with dyadic rational endpoints  and with slopes in $2^{\mathbb Z}$.
 \end{definition}

The map $g(x)\mapsto g(x) \mod 1$ defines a group homomorphism from $\widetilde T$ to Thompson's group $T$ as defined in \cite{CFP}. The kernel
is the group $\mathbb Z$ acting by translations by integers on $\mathbb R$ (which become rotations by multiples of $2\pi$ on
the circle) and the extension $\mathbb Z\rightarrow \widetilde T\rightarrow T$ is central
and not split. Another way to think of $\widetilde T$ is as periodic piecewise linear foliations of the strip $\mathbb R\times [0,1]$ with (straight line)
leaves connecting $(x,0)$ to $(g(x),1)$. By periodicity such a foliation passes to a foliation of the annulus which is the quotient of the
strip by $\mathbb Z$.

Given a normalized element $R\in P_2$ we will define a Hilbert space $\mathcal H=H_{V,R}$ which will carry a  unitary representation $\rho_{V,R}$ of  $\ct$ for which $\mathbb Z$ acts by scalars.
We proceed just as for $F$ but with the necessary modifications. Fortunately the directed set remains the same, only the standard dyadic partitions are
now thought of as partitions of $S^1$, or periodic dyadic partitions of $\mathbb R$. 

Given $\mathcal I\lesssim\mathcal J$ in $\ds$ we define an affine labelled tangle in almost the same way as just before definition \ref{blowup}.
Fix $0<r<R \in \mathbb R$ and circles $r\{z:|z|=1\}$ and $R\{z:|z|=1\}$ as the boundaries of the affine tangle. 
Suppose $I=[a,b]\in \mathcal I$ and that $I$ is the union of standard dyadic intervals $J\in S_I$ for $J\in \mathcal J$. Then join
$re^{2\pi i a}$ to  $Re^{2\pi i a}$ and $re^{2\pi i b}$ to $Re^{2\pi i b}$ by radial  straight line strings.
The region in between these two strings and the inner and outer circles is then isotopic to a rectangle which we fill
with exactly the same pattern as in the definition of $T_{\mathcal J}^{\mathcal I}$ in definition \ref{blowup}.
Applying this procedure to all the intervals of $\mathcal I$ we obtain an annular tangle which we again call $T_{\mathcal J}^{\mathcal I}$.
We illustrate below with the same $\mathcal I\lesssim\mathcal J$ as in \ref{blowup}.\\

\hpic {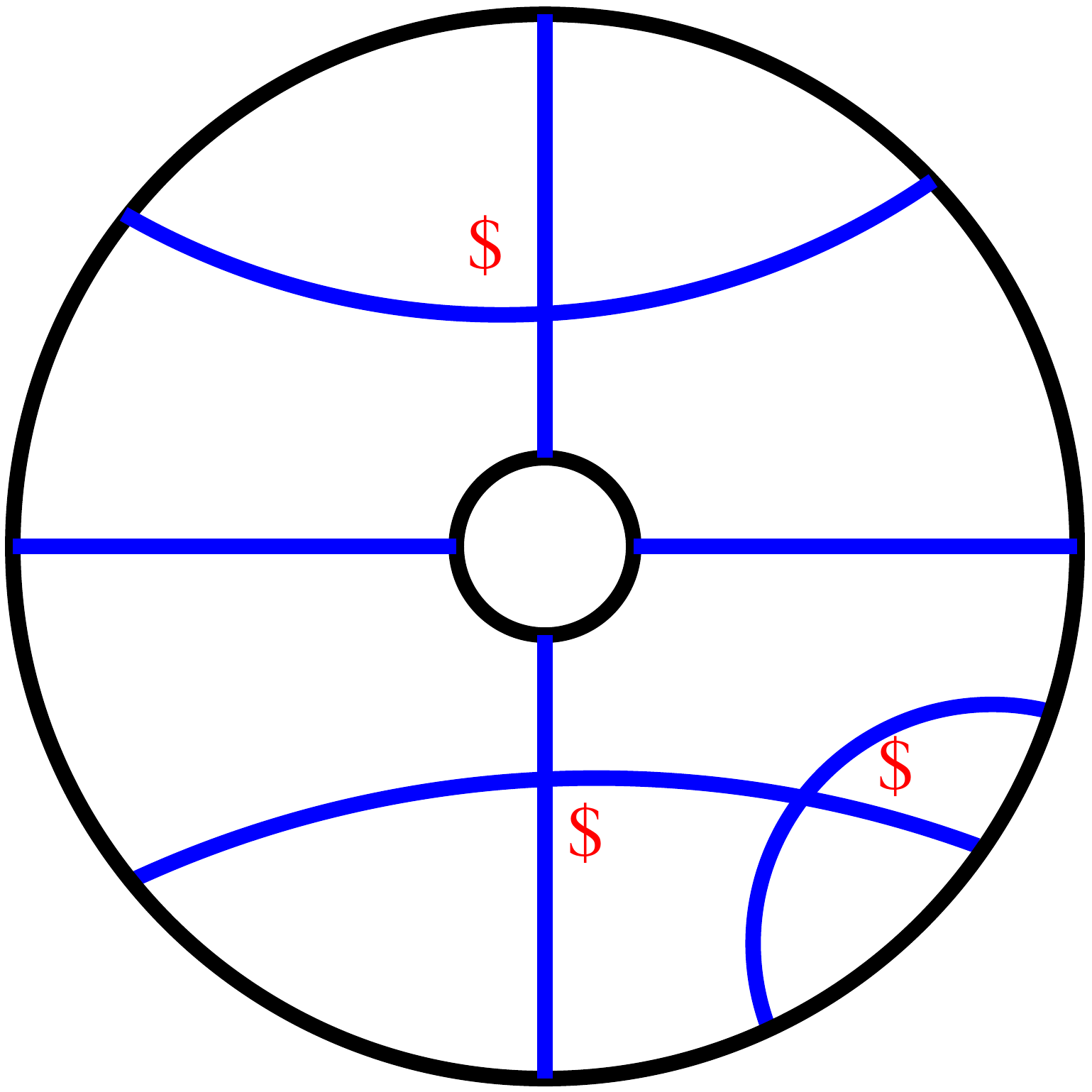} {2in}

Here the crossings between the strings need to be blown up to discs containing $R$, as we have done before.

Note that the affine tangle $T_{\mathcal J}^{\mathcal I}$ will always have a straight line string joining $r$ to $R$.

\begin{definition} \label{annularembeddings}Given $\mathcal P$ and an affine representation $V=(V_n)$ of it, and an element $R\in P_2$ as above we define the directed system of vector spaces 
$\mathcal I\mapsto  K(\mathcal I)$ for $\mathcal I\in \ds$ by $K( \mathcal I)=V_{|E(\mathcal I)|+1}$ and for 
$\mathcal I \lesssim \mathcal J$ in $\ds$ we use the tangle $T_{\mathcal J}^{\mathcal I}$ to define a map from 
$K( \mathcal I)$ to $K( \mathcal J)$ as required by the definition of a directed system.
 \end{definition}

(The reason for the $+1$ in $E(\mathcal I)+1$ is that the points $0$ and $1$ in $[0,1]$ now play a role but constitute exactly
one more point on $S^1$.)

Note further that an affine representation determines a rectangular one by restriction and that our directed system is the \emph{same} 
as that obtained by restricting $V_n$ to the rectangular category. In particular the trivial representation of the affine TL category
has 1-dimensional $V_2$ spanned by the image of the empty diagram (``vacuum vector'') in $V_0$ under a tangle with a single string joining
the boundary points. When restricting from affine to rectangular, this is the vector $\xi$ which we have used above in definition
\ref{xi}. We will persist in calling it $\xi$ in the affine context.

\begin{definition}\label{hilbertaffine} The direct limit vector space $\mathcal V=\lim_{\mathcal I \in \ds} K_{\mathcal I}$ is called the \emph{dyadic limit}
of $V$. If $\mathcal P$ is a subfactor planar  and the affine representations are Hilbert, $\mathcal V$ becomes a pre-Hilbert space since all the $T_{\mathcal J}^{\mathcal I}$
define isometries. Its Hilbert space completion is called the \emph{dyadic Hilbert space} $\mathcal K_{R,V}$ of $V$.
\end{definition}

Construction of the unitary representations  $\rho_{V,R}$ of  $\ct$ proceeds as for $F$, with the necessary modifications.
 
 As before we say that an element $\mathcal I \in \ds$ is good for $h\in \ct$ if 
$\{h(I) \mod 1:I\in \mathcal I\}$ forms a standard dyadic partition of $[0,1]$ and we write $h(\mathcal I)\in \ds$ for this partition.
To see that such an $\mathcal I$ exists,
 suppose that $h$  projects to an element $g\in T$. 
By \cite{CFP} $g$ is uniquely represented by a certain pair of bifurcating
trees representing standard dyadic partitions $\mathcal I$ and $\mathcal J$ of $[0,1]$, the difference from $F$ being that the  interval beginning with $0$ 
in $\mathcal I$  may map
to an interval in $\mathcal J$ which does not begin with $0$. Since $h$ is a lifting of $g$, $\{h(I) \mod 1:I\in \mathcal I\}$
is the same as $\mathcal J$.

If $\mathcal I$ is good for $h$ we define the affine tangle $h_{\mathcal I}$ to be the tangle  obtained as follows:\\
First connect $(0,0)$ to $(h(0),1)$ and all the points $(p,0)$ to $(h(p),1)$ for $p\in E(\mathcal I)$  by straight lines in $\mathbb R\times [0,1]$.  
Repeat this for all horizontal integer translates to obtain a periodic pattern of straight line connections. Then take the quotient
under translations by $\mathbb Z$ to obtain the affine tangle $h_{\mathcal I}$. We illustrate for a lifting $\hat C$ to $\ct$ of the well known
element $C$ (see \cite{CFP}) of $T$. $C$, as a map on $[0,1]$ identified with the circle, is defined by its standard dyadic partition $\{I,J,K\}$
with $I=[0,1/2], J=[1/2,3/4],K=[3/4,1]$ with $C(I)=K,C(J)=I$ and $C(K)=J$. The picture below gives the periodic pattern for the
lifting of $C$ with $\hat{C}(0)=-1/4$, followed by the affine tangle it defines.

\hpic {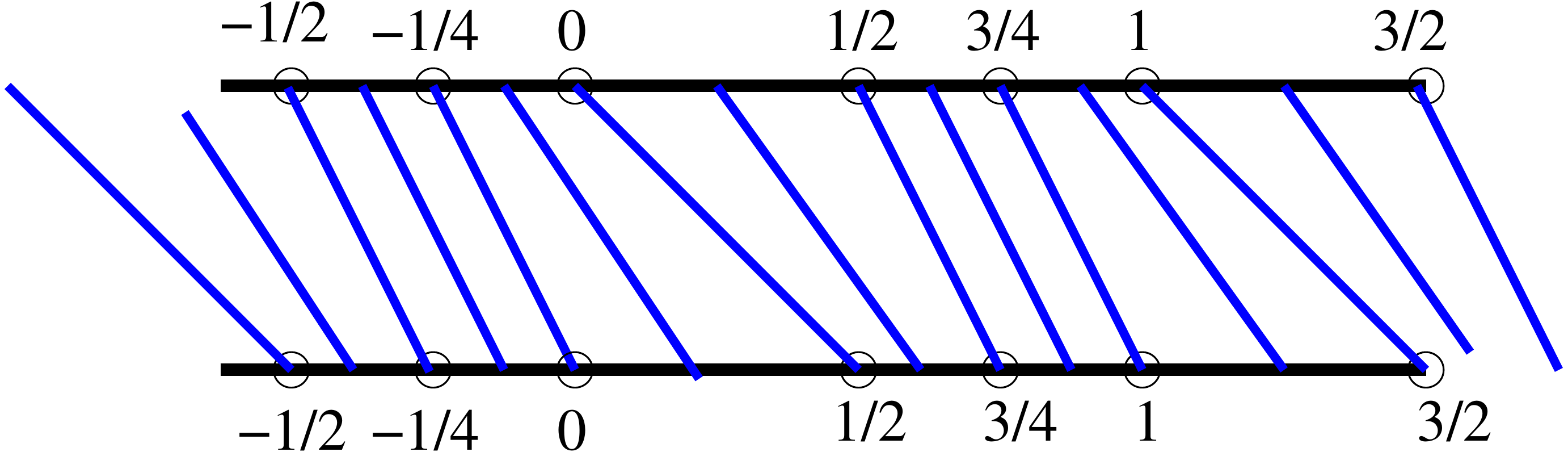} {0.8in}  \hpic {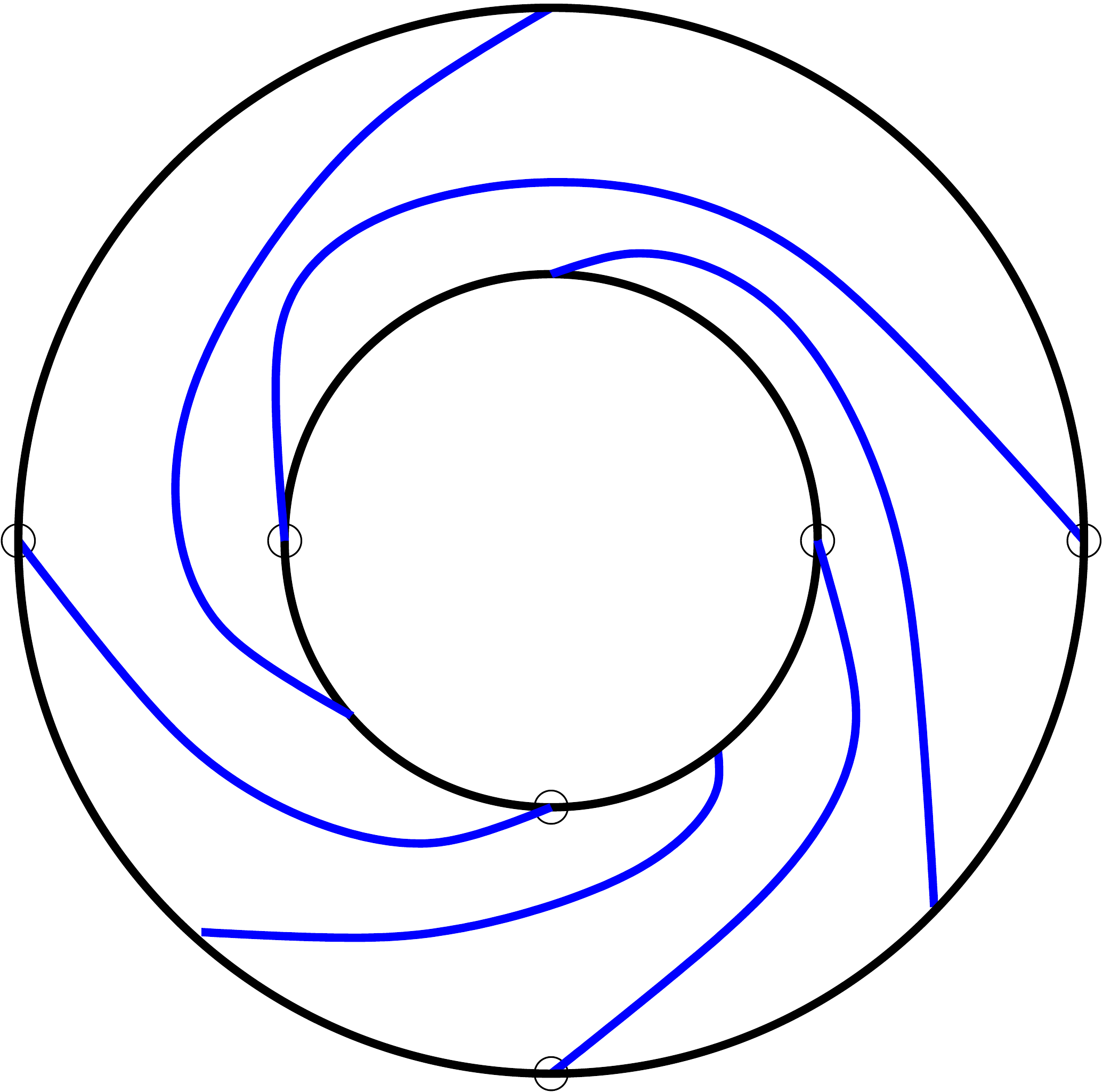} {1.5in}

 where we have indicated the end points of the standard dyadic intervals with small circles but no such circles for the midpoints.
 
 We assert that lemma \ref{main}  is true in this context exactly as stated,  noting simply that tangle now means affine tangle.
 The proof is also the same as lemma \ref{main} if one replaces the pictures by periodic pictures in the strip $\mathbb R \times [0,1]$.
 The sliding of the patterns of $R$'s from bottom to top and required rearrangement of points can be done in a $\mathbb Z$
 periodic way so that it passes to the annulus.
 
 Lemma \ref{main} is the only ingredient for the subsequent results that gave the existence of representations of $F$.
 
 Thus we deduce the following:
 
 \begin{theorem} Suppose we are given an affine representation $V=(V_n)$ of $\mathcal P$ and a normalized element $R\in P_2$. Then
 for $v\in K (\mathcal I)$ and $h\in \ct$ choose $\mathcal J\in \ds$ with $\mathcal I\lesssim \mathcal J$  which is good for $h$ and
 define $$\rho_h(v) = h_{\mathcal J}\circ T_{\mathcal J}^{\mathcal I}(v) \in  K(h(\mathcal J)).$$
 Then in the dyadic limit $\rho_h(v)$ is independent of any choices made and $h\mapsto \rho_h$ defines a representation of
 $\ct$ on the direct limit, which in the subfactor case extends to a unitary representation of the Hilbert space $\mathcal K$.
 \end{theorem}

 \begin{remark} \rm{It can happen that for an affine representation the rotation of $2\pi$ acts as the identity. These representations are
 called annular and the ensuing representations on the dyadic limit pass to the quotient $T$ of $\ct$. This is in particular the case of
 the trivial representation.}
 \end{remark}

The central extension for $\ct$ is canonically split when restricted to the geometric inclusions of $F$ inside $T$. Thus we can talk of the restriction of $\rho_{V,R}$ to $F$ as a representation of $F$. We leave it to the reader to check that, in the case of the embedding of $F$ inside $T$ as homeomorphisms of the
circle fixing the point $1\in \mathbb C$, the restriction of $\rho_{V,R}$ is the representation of $F$ we would get by restricting $V$ to a rectangular 
representation of $\mathcal P$.

\section{Calculation of coefficients}
\subsection{Representation of elements of $F$ as pairs of rooted planar trees.}\label{reps}
As in \cite{CFP}, any element of $F$ is given by a  pair of bifurcating trees $T_+$ and $T_-$ as below. Our convention will be that
each standard dyadic interval represented by  a leaf of the top tree $T_+$ is sent by the Thompson group element 
to the interval represented by the leaf on the tree $T_-$ to which it is connected.

\begin{definition}\label{thompsonfromtrees}
Given $T_+$ and $T_-$ as above the element of $F$ will be called $g(T_+,T_-)$.
\end{definition}

The element $g$ defines $T_+$ and $T_-$ provided there are no cancelling ``carets''-see \cite{CFP}.

It will be convenient to arrange the two bifurcating trees in $\mathbb R^2$ so that their leaves are the points
$(1/2,0),(3/2,0),(5/2,0),\cdots ((2N-1)/2,0)$,  with all of the edges being straight line
segments sloping either up from left to right or down from left to right. $T_+$ is in the upper half plane
and $T_-$ is in the lower half plane.
Then each region between the edges of each tree contains exactly one point in the set $\{(1,0),(2,0),...,(N,0)\}$.
Let us form a new planar  graph $\Gamma$  given from the two trees.
The  vertices of $\Gamma$ are $\{(0,0),(1,0),(2,0),...,(N,0)\}$ and the edges are given by curves passing once transversally through
certain edges of the top and bottom trees. From the top tree use all the  edges sloping up from left to right (which we 
call WN edges) and from the bottom tree use all the 
edges sloping down from left to right (which we cal WS edges).
The figure below illustrates the formation of the graph $\Gamma$ for a pair of bifurcating trees with 5 leaves. We have numbered the vertices of $\Gamma$ with
their $x$ coordinates, and we will henceforth use that numbering to label those vertices. \\

\qquad \qquad \vpic{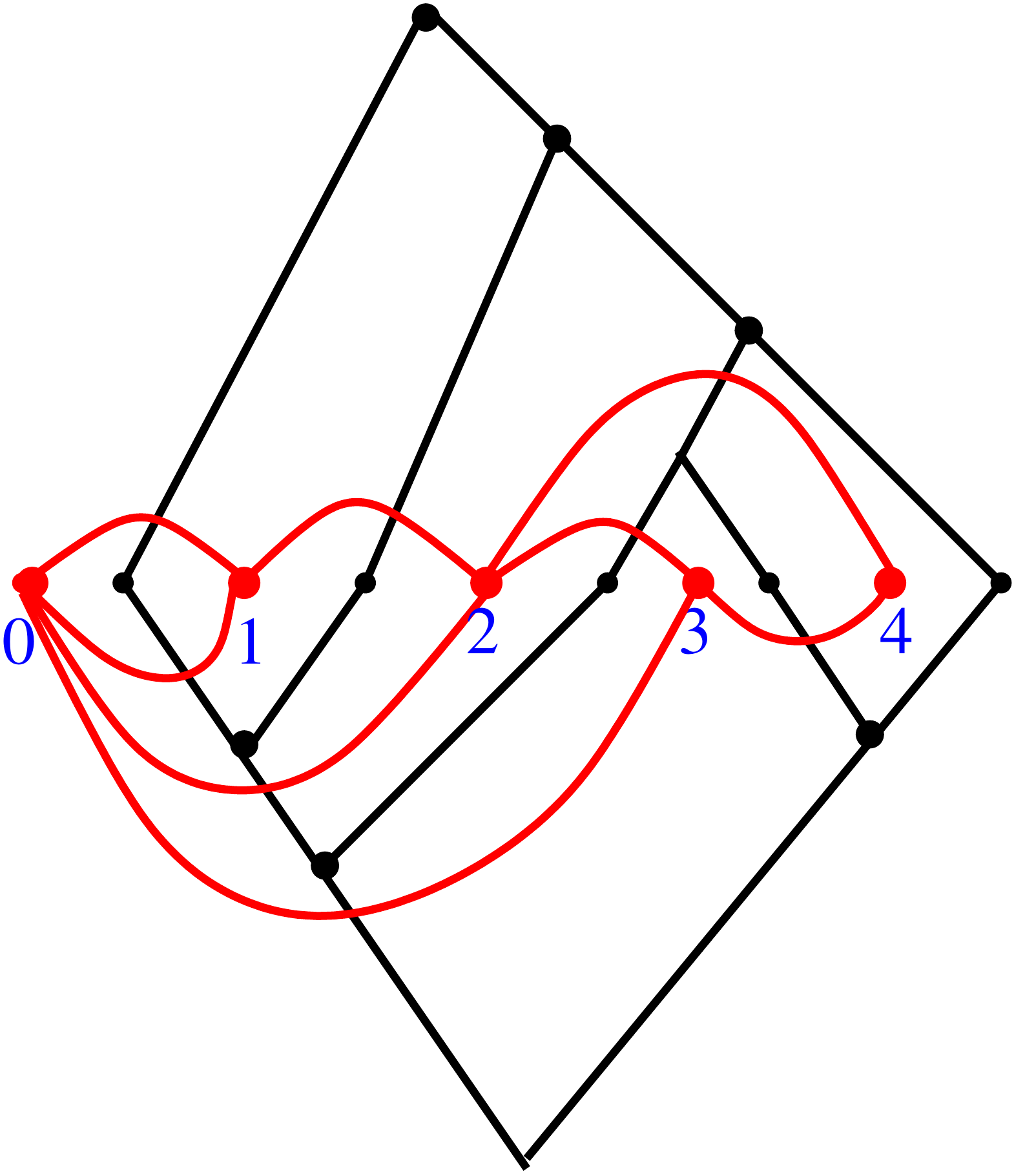} {2.5in}

To be quite clear the above element of $F$ is linear on each of the following five standard intervals, which
it maps to the next five in the given order:

$\{[0,\frac{1}{2}],[\frac{1}{2},\frac{3}{4}],[\frac{3}{4},\frac{13}{16}],[\frac{13}{16},\frac{7}{8}],[\frac{7}{8},1]\}\rightarrow
\{[0,\frac{1}{8}],[\frac{1}{8},\frac{1}{4}],[\frac{1}{4},\frac{1}{2}],[\frac{1}{2},\frac{3}{4}],[\frac{3}{4},1]\}$
\begin{definition} \label{treetotree1}Given $T_+$ and $T_-$ as above , the planar graph $\Gamma$ defined above is called the planar graph of $T_+,T_-$, written $\Gamma(T_+,T_-)$ or $\Gamma(g)$ if there are no cancelling carets so that the data of the two trees is the same as the data $g\in F$.
\end{definition}

Observe that the procedure for constructing $\Gamma$ actually constructs a rooted tree $\Gamma_\pm(T_\pm )$ with vertices $\{(0,0),(1,0),(2,0),...,(N,0)\}$ from a single
bifurcating tree $T_\pm$ either in the upper $(+)$  or lower $(-)$ half plane with leaves \\$(1/2,0),(3/2,0),(5/2,0),\cdots ((2N-1)/2,0)$.

Note that the graph $\Gamma (T_+,T_-)$ is also a pair of rooted planar trees, one in the lower half plane and one in the upper half plane having the
same root and the same leaves. But they
are not bifurcating in general, the valence of each vertex being unconstrained.

 Cancelling of carets between $T_+$ and $T_-$  corresponds to removal of  a 
two-valent vertex connected only to its neighbour,
and the edges connected to it. 

\begin{proposition}\label{conditions}
The graph $\Gamma$ formed above from a pair of bifurcating trees   consists of two trees, $\Gamma_+$  in the upper half plane 
and $\Gamma_-$  in the lower half plane,  having the following properties:\\
(0)The  vertices are $0,1,2,\cdots, N$.\\
(\romannumeral 1) Each vertex other than $0$  is connected to exactly one vertex to its left.\\
(\romannumeral 2) Each edge can be parametrized by a smooth curve $(x(t),y(t))$ for $0\leq t \leq 1$ with $x'(t)>0$ and 
either $y(t)>0$ for $0<t<1$ or $y(t)<0$ for $0<t<1$.
\end{proposition}
\begin{proof} This is obvious from the construction of $\Gamma$.
\end{proof}
 
 Graphs  of the from $\Gamma_\pm$ are obviously oriented so we may talk of the source and target of an edge.
 We will show below how to reconstruct the pair of bifurcating trees from a pair of rooted planar trees with vertices  satisfying the conditions
 of \ref{conditions}.
 
  This shows that $\Gamma(g)$ is an equally faithful way of representing elements of the Thompson group $F$.
 
 \begin{lemma}\label{catalan}
 Let $\Psi$ be a rooted tree in the  upper or lower half plane satisfying the conditions of proposition \ref{conditions}. Then there is a bifurcating
 tree  $T_\pm$ such that $\Psi=\Gamma_\pm(T_\pm)$.
 \end{lemma}
\begin{proof} Wolog we may assume everything is in the lower half plane.

We will work by induction on the number of leaves. So suppose we are given a $\Psi_-$ satisfying the conditions
of \ref{conditions}  with $N+1$ leaves.  Call a vertex of $\Psi_-$ 
\emph{terminal} if it is not the source of an edge. 


If $j$ is a terminal vertex then it is the target of a unique edge. The source of that
edge is $k$ for $k<j$. If $k=j-1$ we will call $j$ \emph{minimal terminal}. If $j$ fails to be minimal terminal then $j-1$ could, by
planarity, only be connected to the right to $j$, so $j-1$ is terminal. Continuing in this way we obtain a minimal terminal vertex $m$.
There are then two possibilities for $m-1$\\
\vskip 5pt
Case(1). Valence of $m-1$ is $2$. Then if $m$ is deleted $m-1$ becomes terminal and in a neighborhood of $m$ and $m-1$ $\Psi$ is as below:\\

\qquad \qquad \vpic{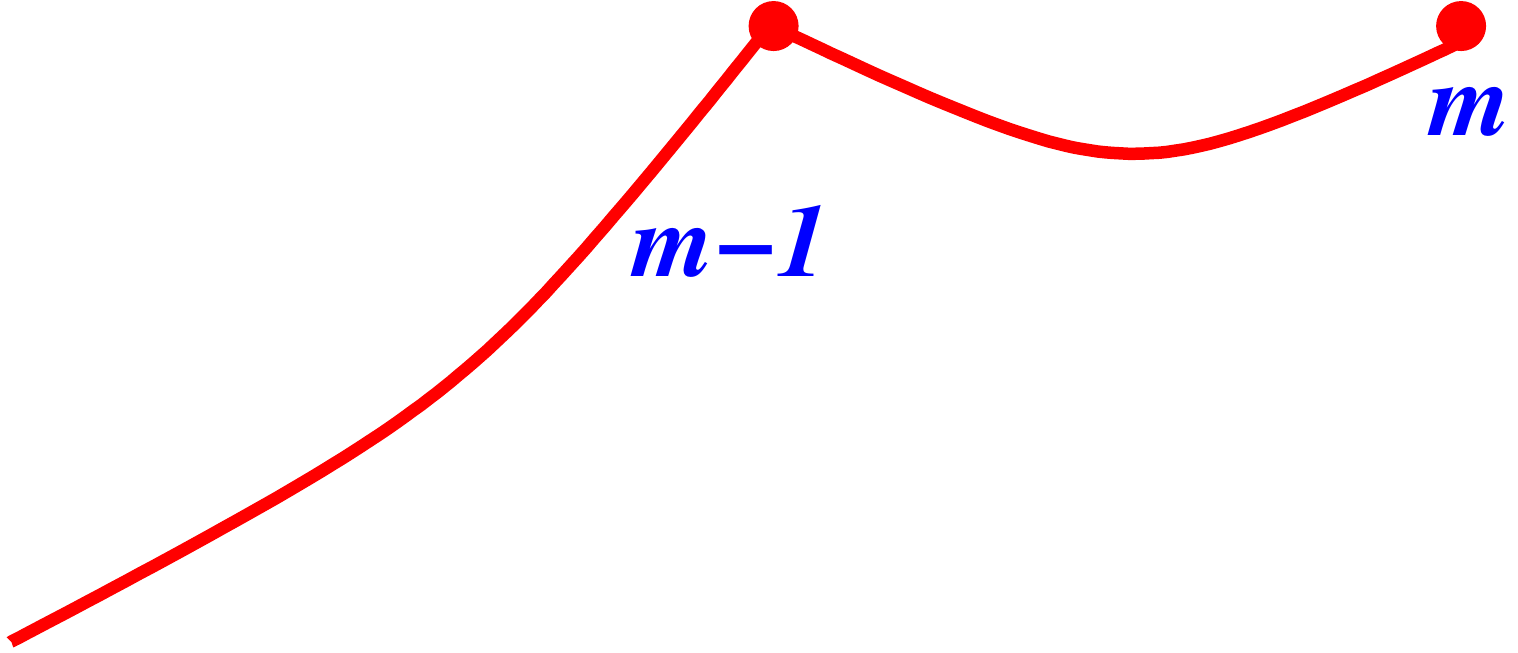} {2.5in}
 
 Removing $m$ and its edge the resulting  graph $\Psi'$  still satisfies the conditions of \ref{conditions} so there is by induction a $T'$
  with $\Psi'= \Gamma_-(T')$. Observe that the terminal vertex $m-1$ is necessarily in a caret of $T'$.   We may thus add WS edge to $T$ to re-insert
  the vertex $m$  and obtain the desired $T_-$:\\
  
  \qquad \qquad \vpic{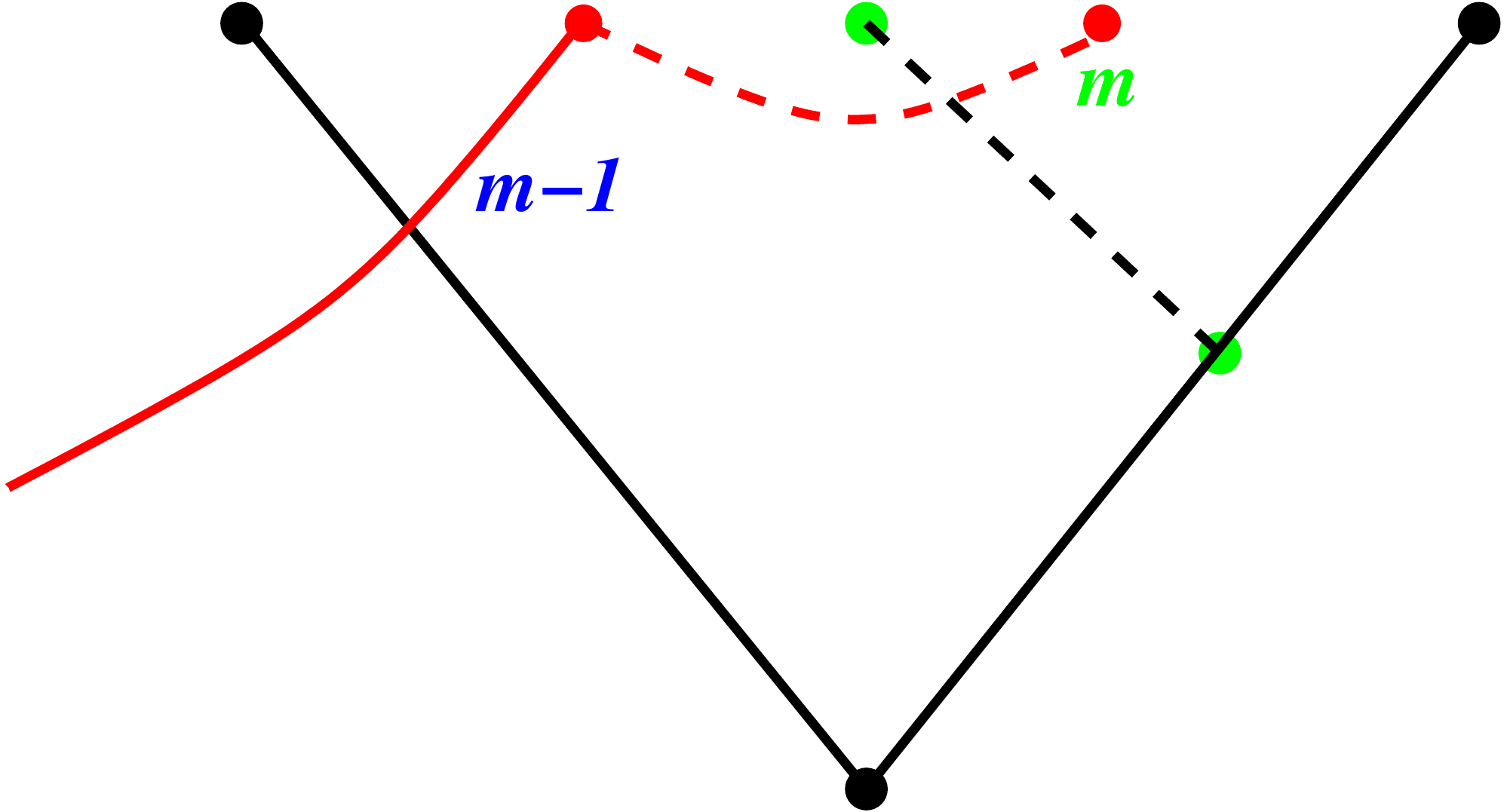} {2.5in}

  where the solid edges are those of $T'$ and the dashed edges are the ones added to obtain $T_-$ and $\Psi$.
  
  Case(2) Valence of $m-1$ is $>2$. In this case there is an edge with source $m-1$ connecting it to a vertex $k$ with $k>m$. By planarity
  there must be such an edge connecting $m-1$ to $m+1$. The situation near $m$ is thus:\\
  
  \qquad \qquad \vpic{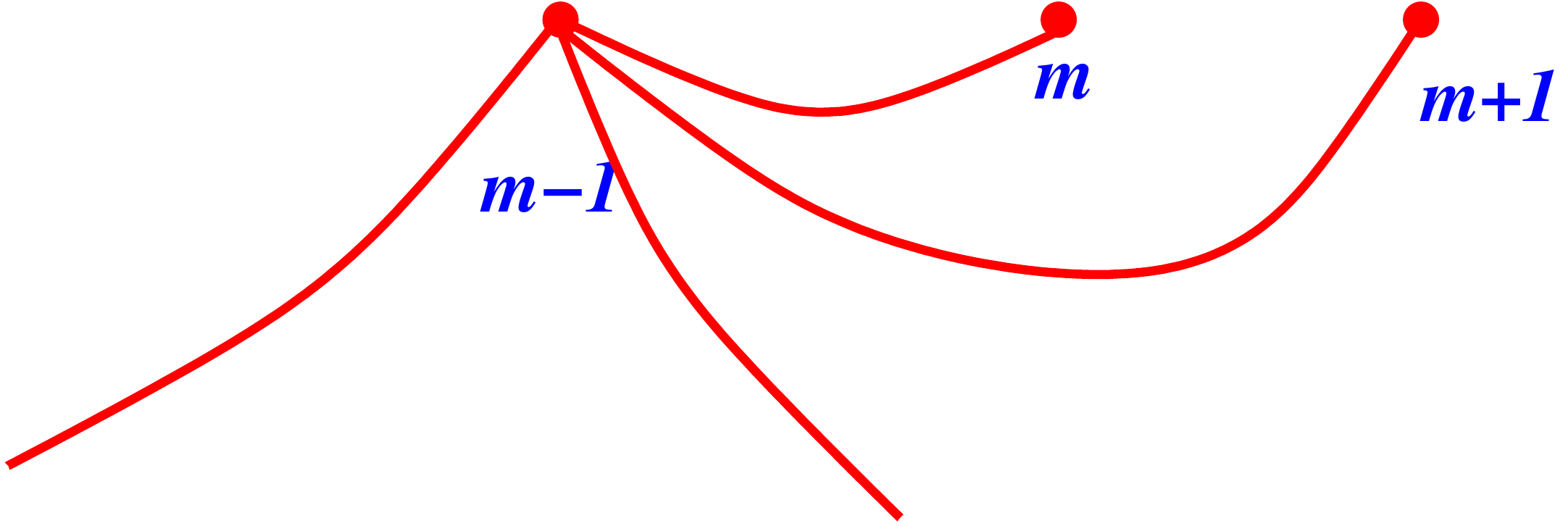} {2.5in}
  
  Removing $m$ and its edge the resulting  graph $\Psi'$  still satisfies the conditions of \ref{conditions} so there is by induction a $T'$
  with $\Psi'= \Gamma_-(T')$. There has to be a WS edge in $T$ between $m-1$ and $m+1$ so we may add a WN edge to $T'$ as 
  in the figure below re-insert $m$ and obtain the desired $T_-$:\\
  
  \qquad \qquad \vpic{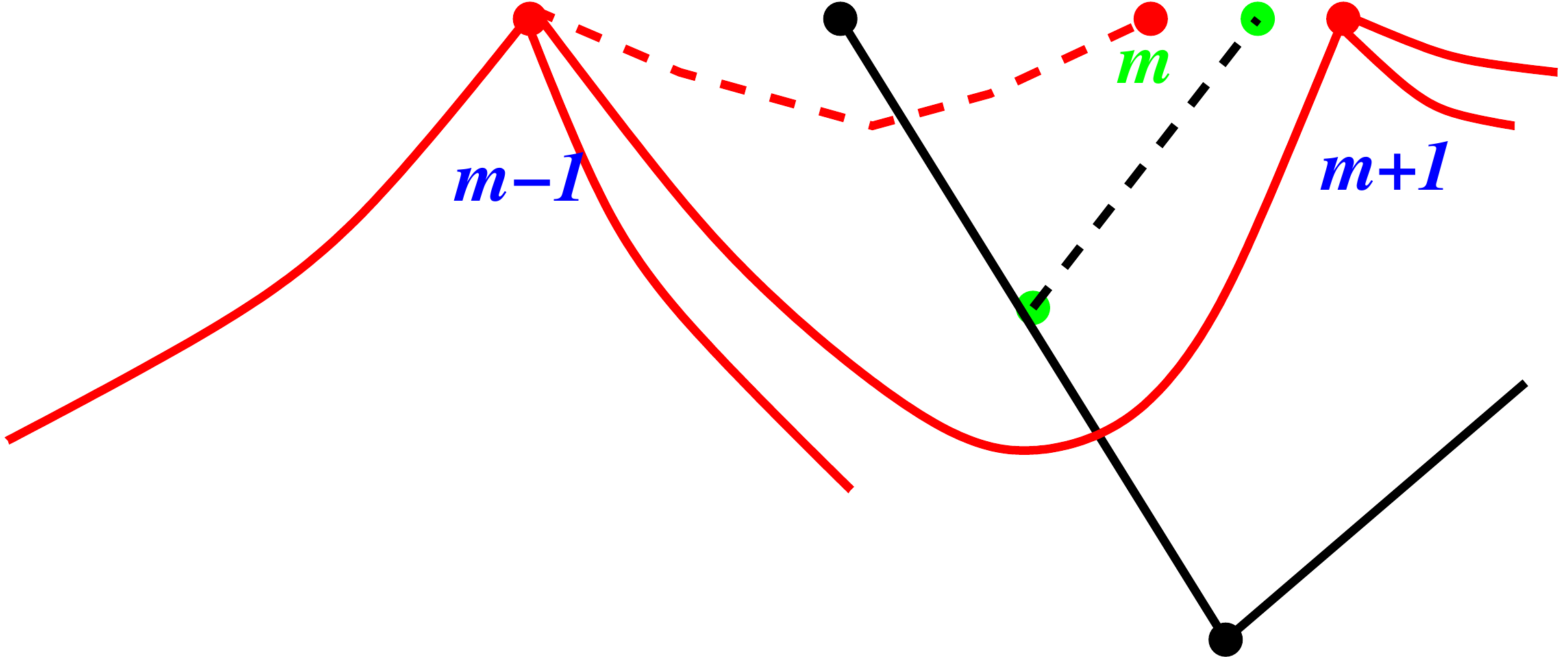} {2.5in}

  where the solid edges are those of $T'$ and the dashed edges are the ones added to obtain $T_-$ and $\Psi$.
\end{proof}

We have essentially given a bijection between two sets counted by the Catalan numbers which thus almost certainly exists in 
the literature. The result is so important to our examples that we have supplied a detailed description of what is an algorithm to
obtain a Thompson group element from a pair of planar graphs with the same set of leaves and the same root.

\subsection{Representation of elements of $T$ as pairs of rooted planar trees.}\label{trees}

As shown in \cite{CFP}, an element $g$ of Thompson's group $T$ can be uniquely represented by a pair $T_\pm$ of bifurcating trees whose leaves
index the intervals in two elements of $\ds$. The only difference is that for $T$, $g$ may send the leftmost interval specified by $T_+$ to
the interval of any leaf  $\ell_0$ of $T_-$, the other intervals lining up in cyclic order. Thus the pictures corresponding to the one we have
drawn for $F$ should be drawn on a circle. This is most readily achieved by identifying the first leaf of $T_+$ with $\ell_0$ and placing
$T_-$ in the lower half plane with as many leaves as possible identified with those of $T_+$ in cyclic order. Then join the unattached
leaves with curves in the only planar way possible. We illustrate below with the same $T_\pm$ as in  the figure before definition \ref{treetotree1},
but with the third leaf from the left of $T_-$ being connected to the first of $T_+$:\\

\hpic {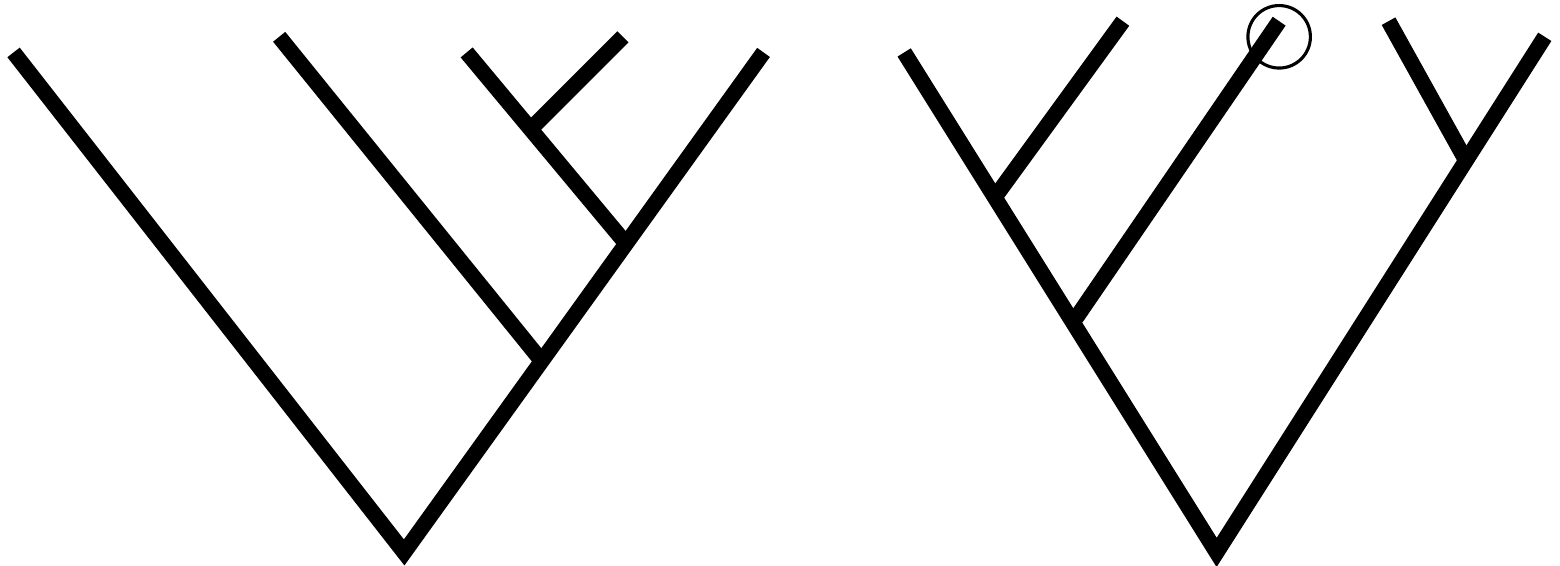} {0.6in} $\rightarrow$ \hpic {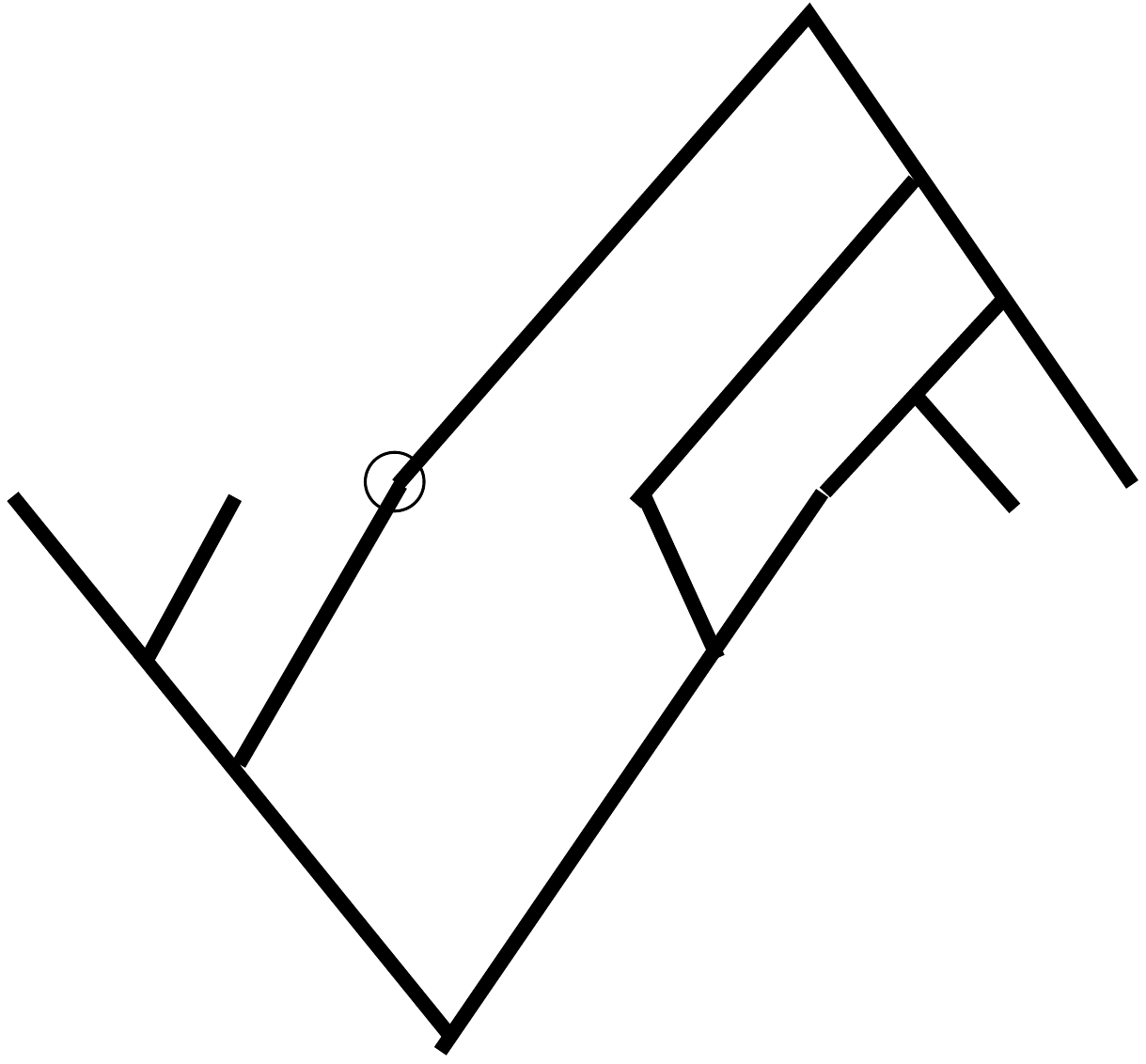} {1in} $\rightarrow$ \hpic {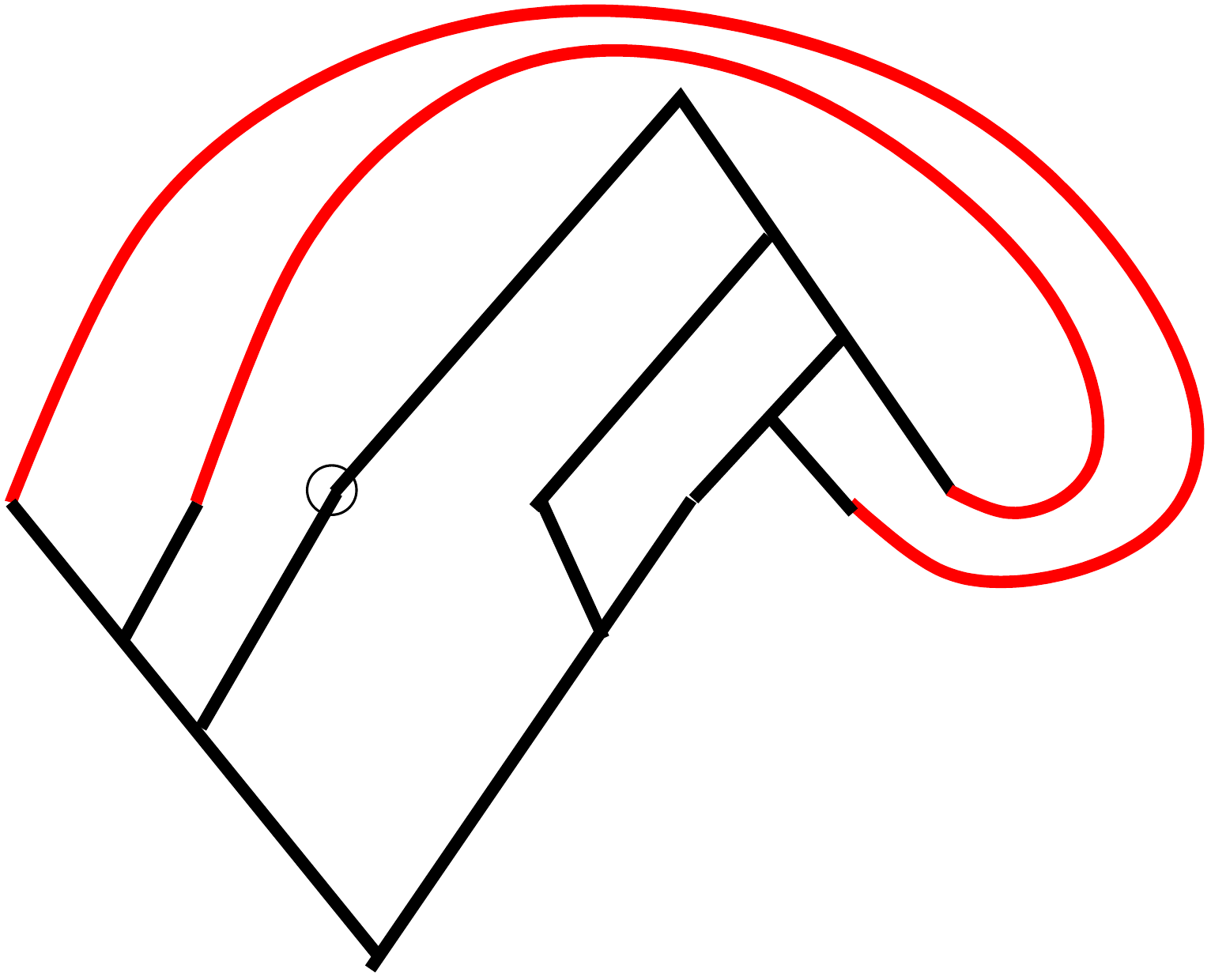} {1in}

\notetoself{turn this upside down to make it compatible with annular pictures}
Thus, exactly as for $F$, we may also replace the pair $T_\pm$ by a pair $\Gamma_\pm$ of planar rooted trees drawn in the plane
with  the same set of vertices on the $x$ axis and with the root of $\Gamma_-$ being the same as a marked vertex of $\Gamma_+$. In our example
above we get:\\

\hpic {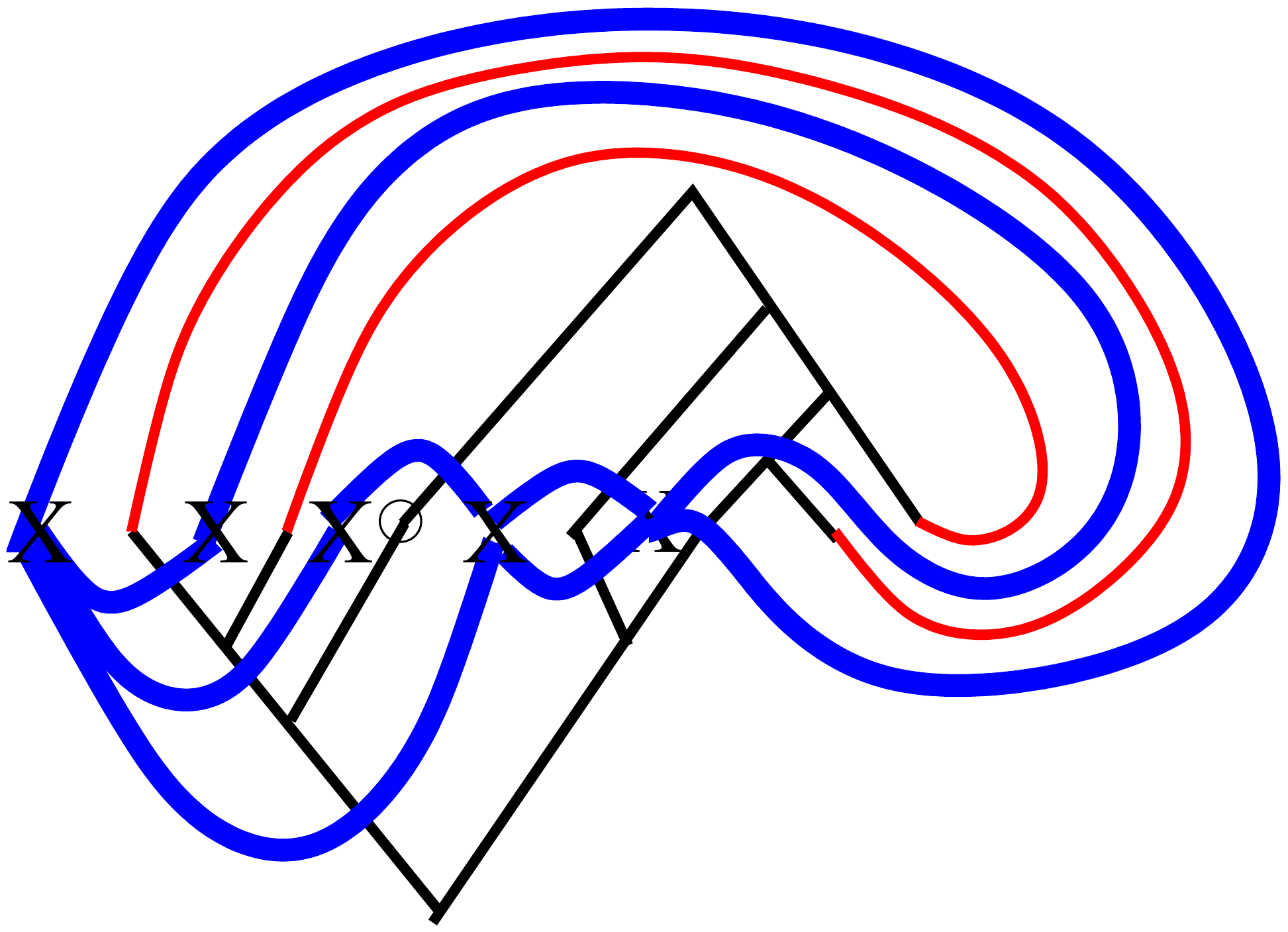} {2in}  

and after removing $T_\pm$ and cleaning up we get:\\

\hpic {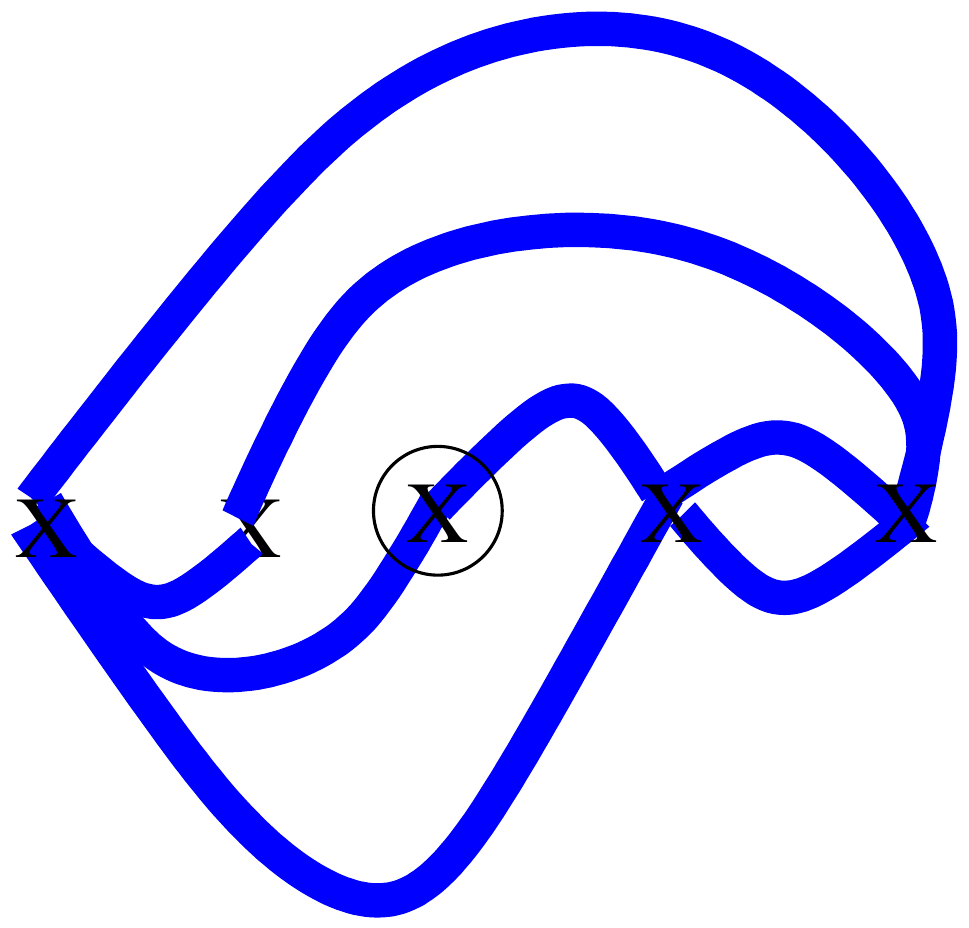} {2in}  \\

where the circled vertex is the root of $\Gamma_-$ and the vertex of $\Gamma_+$ with which it is identified.

The pair $\Gamma_\pm$ should really be thought of as on the inside and outside of the circle obtained by wrapping the
$x$ axis around on itself from below. The positive $x$ direction then becomes the clockwise direction on the circle.

As for $F$, $T_\pm$ can be reconstructed from $\Gamma_\pm$ so we have a faithful way of representing elements of $T$.

\subsection{The coefficients  $\langle \pi_{\Xi, R}(g)\xi,\xi\rangle$ for $F$}\label{coeffs}

We want to explain the relevance of the graph $\Gamma(g)$ for $g\in F$. First we explain a well known correspondence
between planar graphs and four-valent planar graphs.
 Suppose we have a planar graph $\Gamma$ with vertices $V$ and edges $E$. Then we may form the \emph{medial} graph whose
 vertices are midpoints of the edges $E$ and whose edges are obtained by connecting the vertices to adjacent edges around all 
 the polygonal faces of $\Gamma$ thus:\
 
  \qquad  \hpic{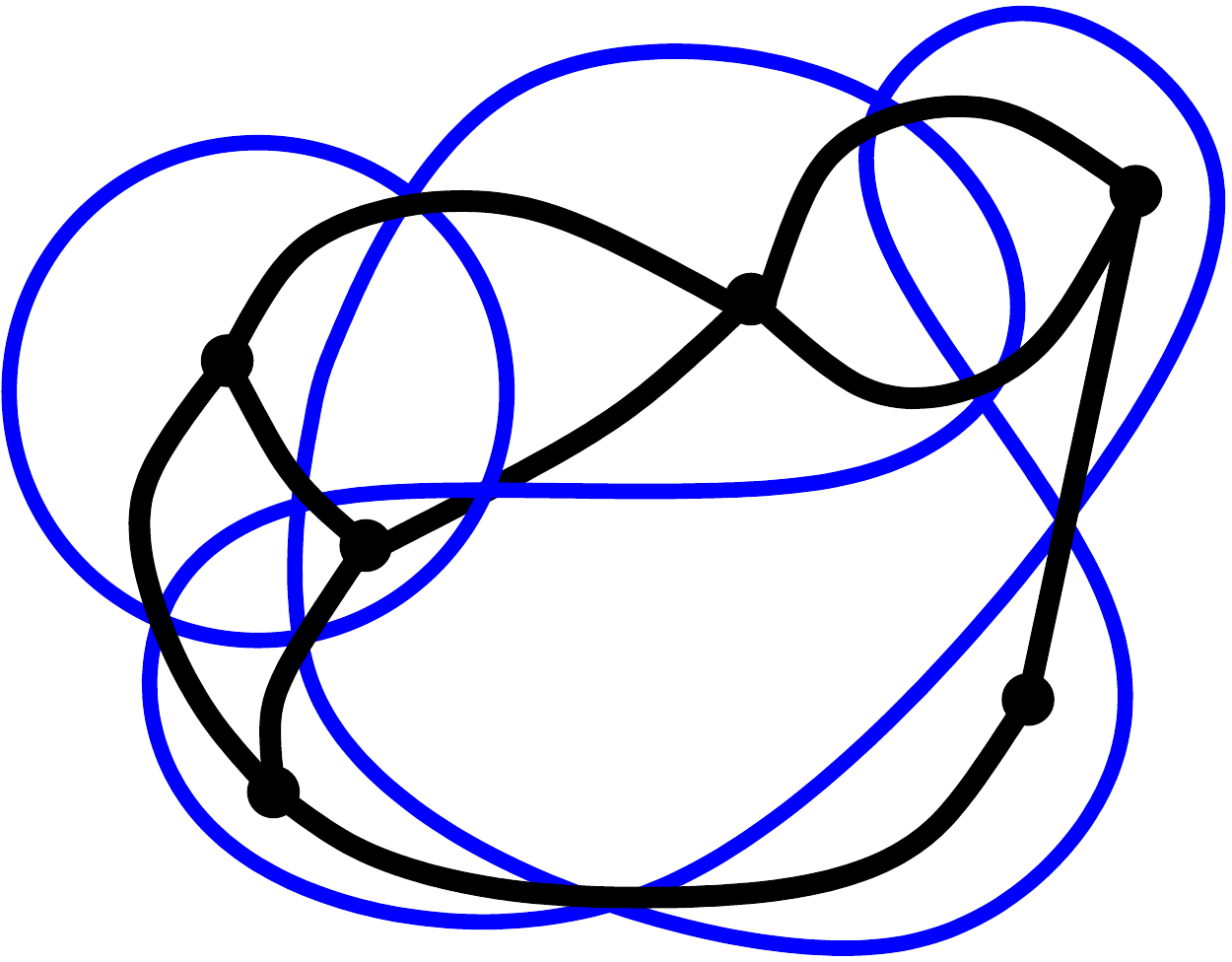} {2.5in}

One may shade the faces of the medial graph so that the unbounded face is unshaded. Then the vertices of $\Gamma$ are in the shaded
faces. This procedure gives a bijection between planar graphs and four-valent planar graphs (which are actually the same as generic projections
of smooth links in $\mathbb R^3$ onto $\mathbb R^2$). If $\Phi$ is the four-valent planar graph we call $\Gamma(\Phi)$ the 
graph $\Gamma$ and conversely given $\Gamma$ we call $\Phi(\Gamma)$ the medial graph.

The vertices of $\Gamma$ are the big black dots and the medial graph is in blue. We now want to blow up the crossings of the medial
graph and insert $R$  as we did in \ref{blowup}. For this we need to know where to put the $\$$ signs if $R$ is not invariant under the rotation
by $\pi$. To achieve this we need $\Gamma$ to be edge oriented. If it is we make the convention that the $\$$ sign always goes to the left. 
We also need to be able to insert $R^*$'s instead of $R$'s at the vertices of $\Phi(\Gamma)$ for which which we need a signing of the edges.
We illustrate below, for the previous example, the formation of the labelled 0-tangle from the graph above to which we have given an edge
orientation and an edge signing:\

 \hpic{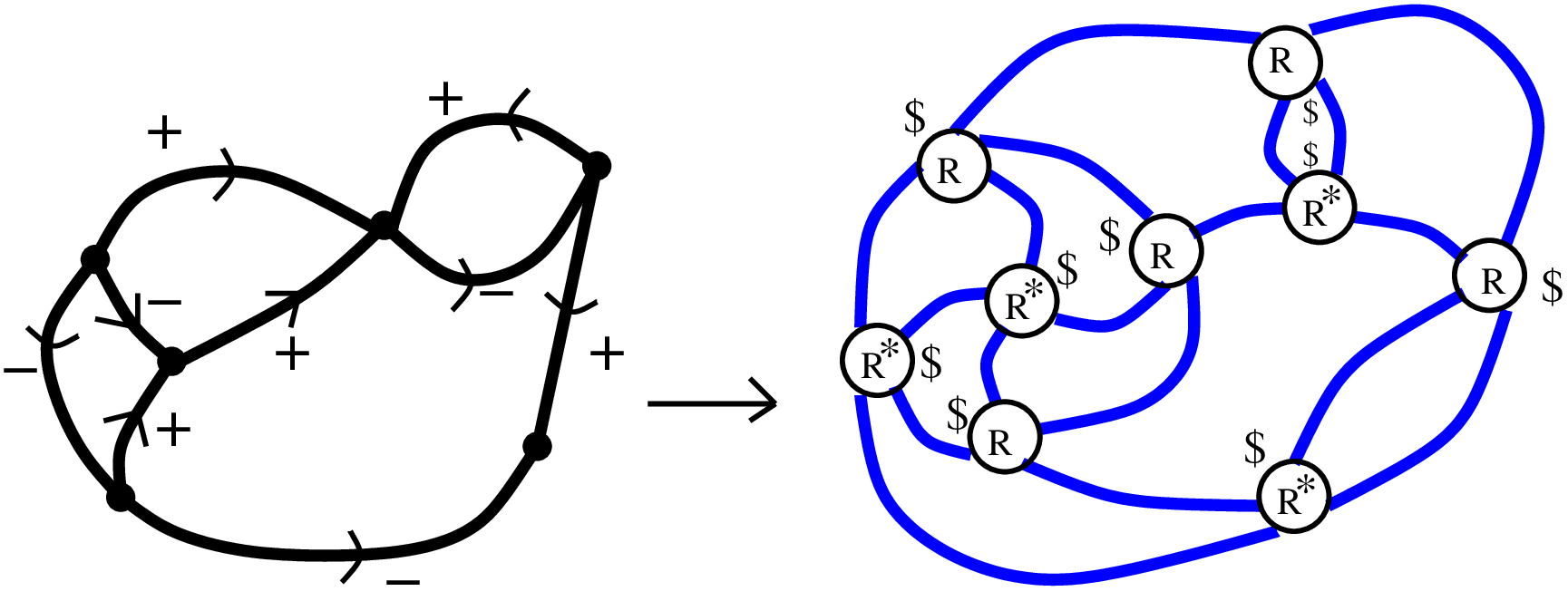} {2in}

\begin{definition}If $R\in P_{2,+}$ is given we call the labelled tangle above  the \emph{$R$-tangle} $T_R(\Gamma)$ of the edge oriented and edge signed
planar graph $\Gamma$.
\end{definition}
We use the same notation if $R$ is rotationally invariant and $\Gamma$ is not edge-oriented and/or if $\Gamma$ is not edge-signed
and $R$ is self-adjoint .

Recall from definition \ref{xi} that $\Xi$ is the rectangular representation of the planar algebra generated by
$\xi$ in the regular representation of $\mathcal P$ (\ref{regular}).


\begin{theorem} \label{calculation} If $R\in P_{2,+}$ and $T_+$ and $T_-$ define an element $g\in F$,
$$\langle \pi_{\Xi,R}(g)\xi,\xi\rangle=Z(T_R(\Gamma(T_+,T_-)))$$
where the edges of $\Gamma(T_+,T_-)$ are oriented from left to right and given $+$ signs on the top and $-$ signs on the bottom. 

\end{theorem}
\begin{proof}
Given $T_+,T_-$, by definition the partition $\mathcal J$ defined by $T_-$ is good for $g(T_+,T_-)$. So if $\mathcal I$ is the
partition with just one interval, the labelled tangle
$T_{\mathcal J}^{\mathcal I}$ may be composed with $g_{\mathcal J}$ to give  the tangle $g_{\mathcal J}\circ T _{\mathcal J}^{\mathcal I}$
which gives $\pi(g)(\xi)$ when applied to $\xi$. 
 It consists of a rectangular tangle 
labelled with $R$'s with one boundary point at the bottom and strings  with crossings blown up to include $R$'s. With the strings
connecting the boundary point at the bottom and the mid points and end points of the intervals of $g(\mathcal J)$ at the top. We illustrate below
with the element $g$ which we gave explicitly in section \ref{reps}:\\

\qquad  \hpic{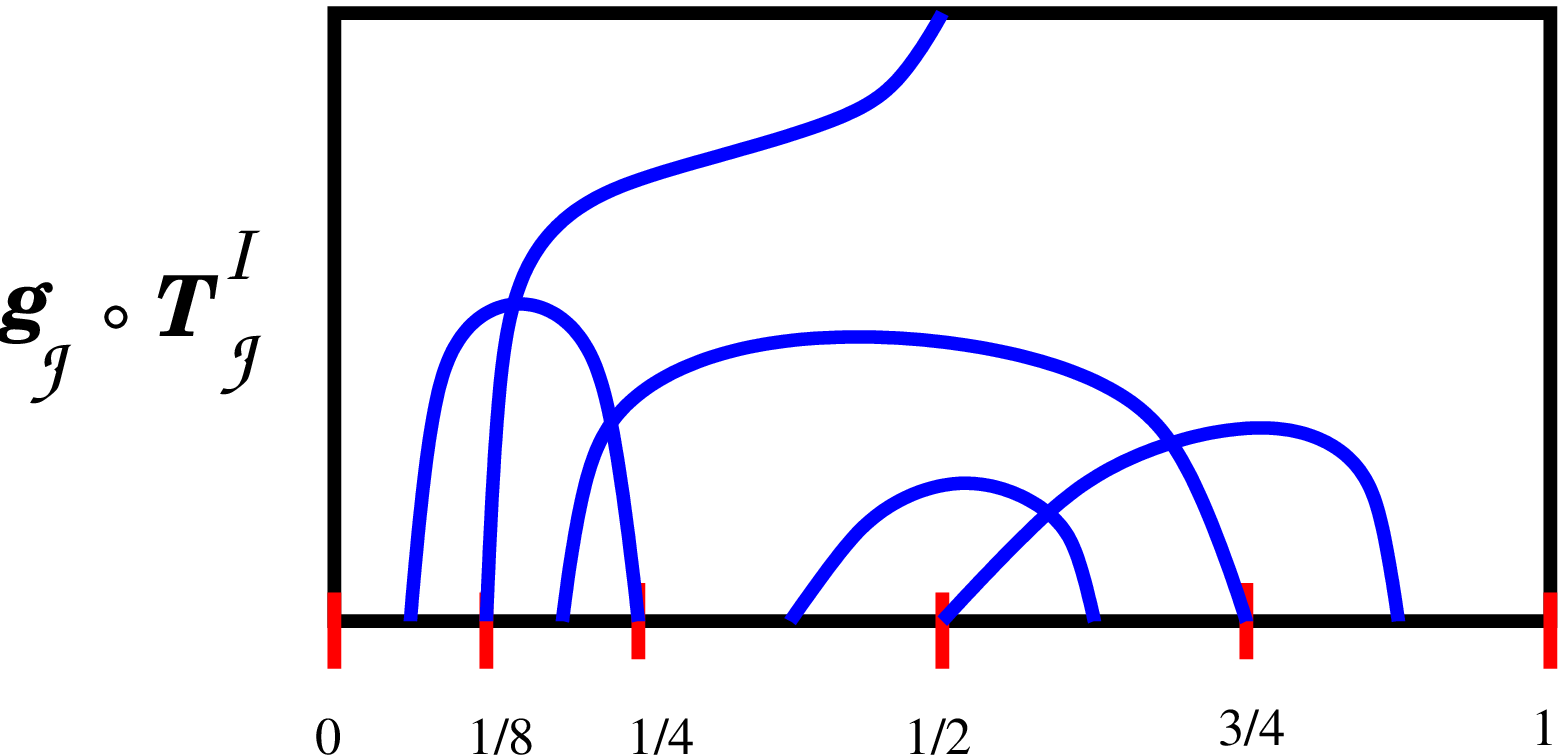} {2in}

Now, by the way $g$ is defined from $T_-$ and $T_+$, the intervals of $g(\mathcal J)$ are the intervals defined by $T_-$ so that
$T_{\mathcal J}^{g(\mathcal I)}(\xi)$ is $\xi$ viewed as an element of $H_{g(\mathcal J)}$. So to obtain $\langle \pi_{\Xi,R}(g)\xi,\xi\rangle$
we need simply attach  $(T_{\mathcal J}^{\mathcal I})^*$ to the bottom of $g_{\mathcal J}\circ T _{\mathcal J}^{\mathcal I}$
and join the top boundary point to the bottom, to the left of the rectangles so that the unbounded region is unshaded. 
We illustrate below, continuing with the example from section \ref{reps}: \\

\qquad \hpic{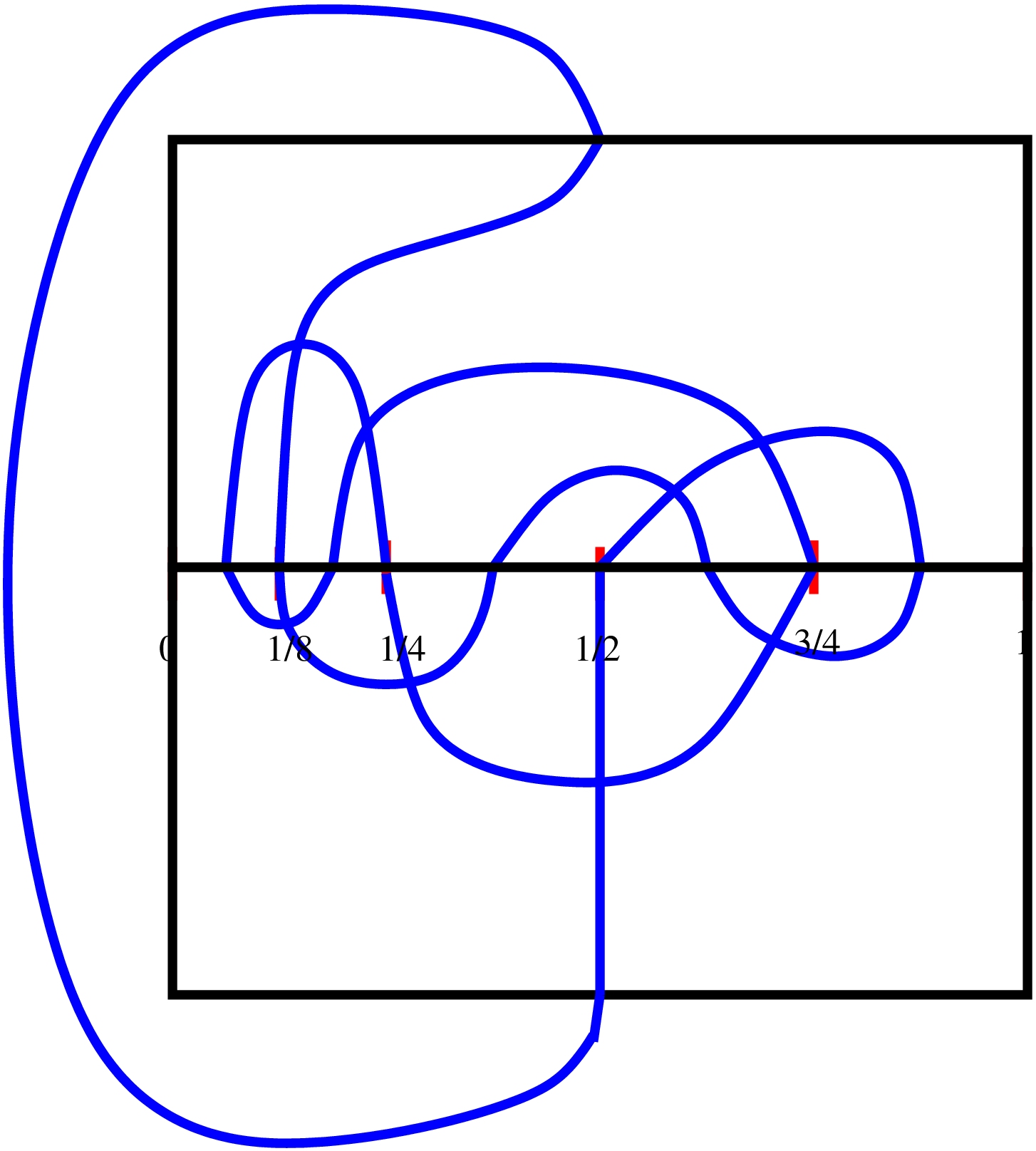} {2in}

Cleaning up and adding the \$ data and * data, we obtain:

\qquad \hpic{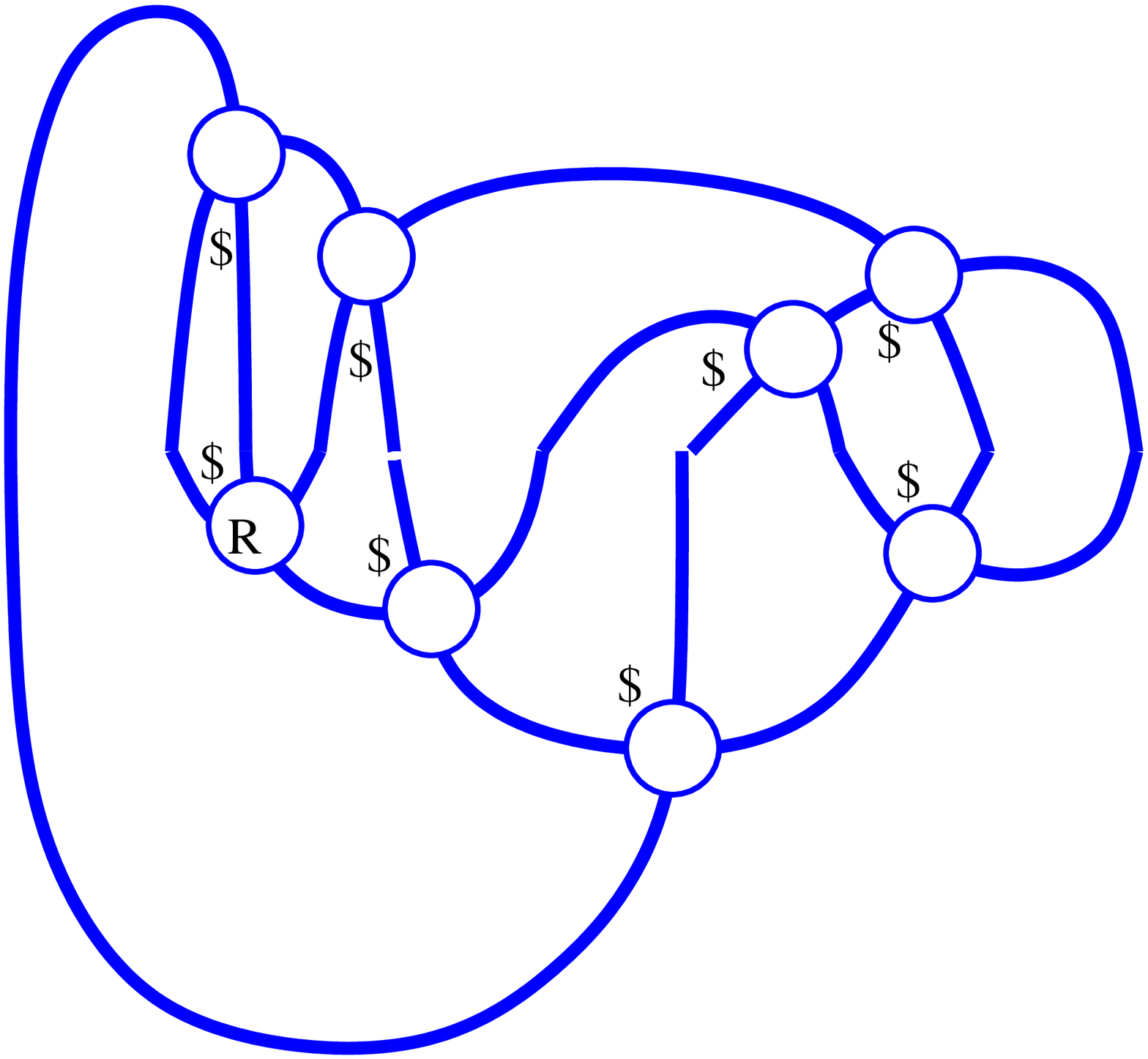} {2in}

Finally we extract the edge-signed edge-oriented graph of shaded regions (unbounded region unshaded) which we 
recognise as  $\Gamma(T_+,T_-)$:

\qquad  \hpic{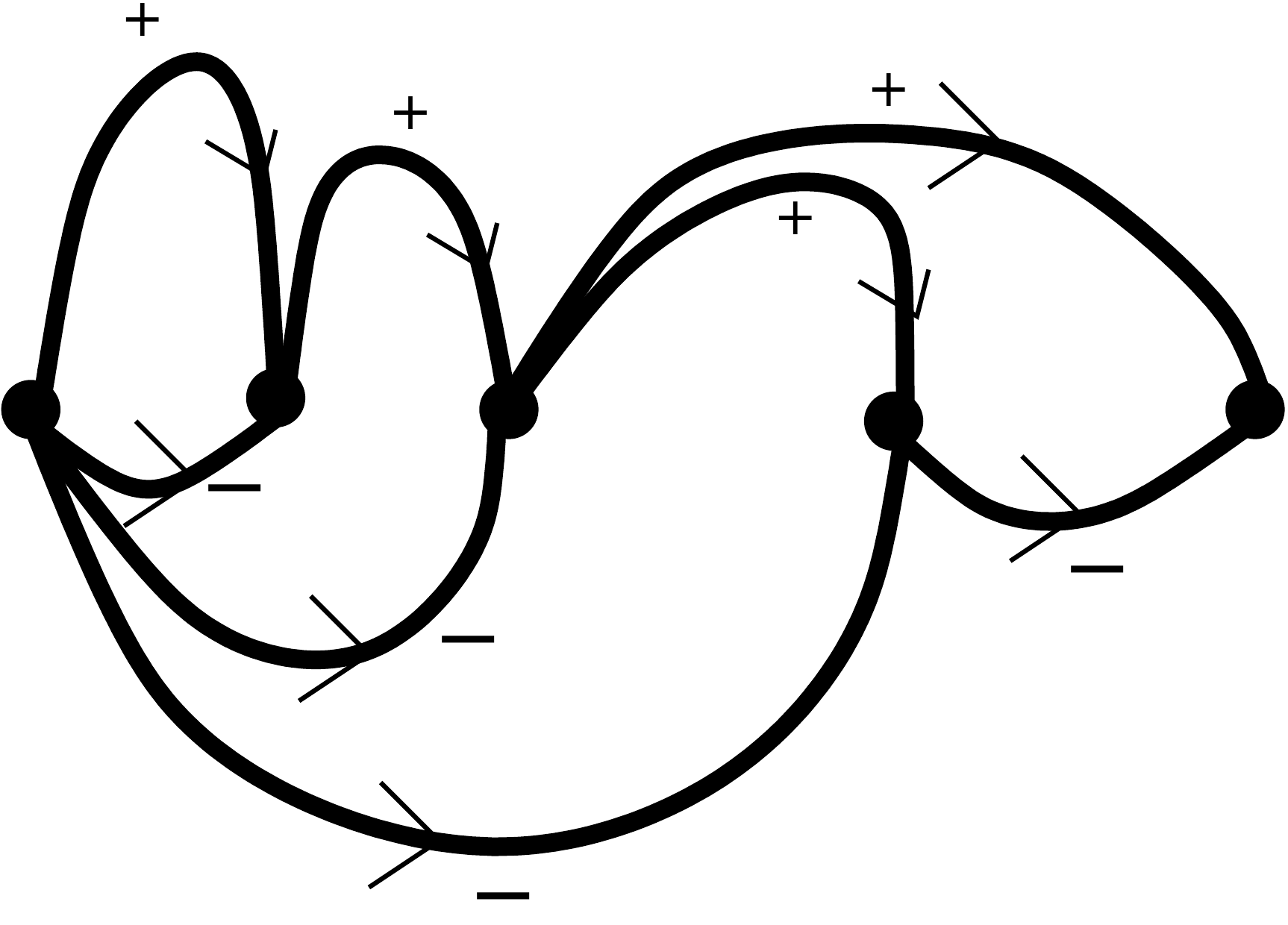} {2in}



\end{proof}
\subsection{The coefficients  $\langle \pi_{\Xi, R}(g)\xi,\xi\rangle$ for $T$}\label{tcoeffs}

All the remarks above pertaining to $\langle \pi_{\Xi, R}(g)\xi,\xi\rangle$ for $T$ apply equally
to $T$ where we agree to  call $\Xi$ the trivial affine representation of $\mathcal P$ with vector $\xi$ as
above. In particular we constructed in subsection \ref{reps} a planar graph consisting of $\Gamma_\pm$ and
if we put negative signs on the edges of $\Gamma_-$ and positive ones on the edges of $\Gamma_+$, and orient
the edges in the clockwise direction we obtain an edge oriented, edge-signed planar graph $\Gamma(T_+,T_-)$  from 
a pair of bifurcating trees (with a marked leaf for $T_+$) defining the element $g\in F$. 

The calculation of theorem \ref{calculation} is true for this $\Gamma(T_+,T_-)$ for $g\in T$ and the edge orientations being clockwise. We illustrate
below for the element $g$ we have been using. Note that
$g$ is defined by the following two elements $\mathcal I$ and $g(\mathcal I)=\mathcal J$ which we give below, with the image of
each standerd dyadic interval under $g$ placed directly below that interval:

$\begin{aligned}  \mathcal I = \{&[0,\frac{1}{2}],[\frac{1}{2},\frac{3}{4}],[\frac{3}{4},\frac{13}{16}],&[\frac{13}{16},\frac{7}{8}],[\frac{7}{8},1]\}\\
                       \mathcal J= \{&[\frac{1}{4},\frac{1}{2}],[\frac{1}{2},\frac{3}{4}],[\frac{3}{4},1],&[0,\frac{1}{8}],[\frac{1}{8},\frac{1}{4}]\} 
                       \end{aligned}$
                       
 We present a sequence of 5 pictures taking us from the definition of $\langle \pi_{\Xi, R}(g)\xi,\xi\rangle$
to the edge-oriented edge signed graph above.\\
Picture 1:the element $\xi$ viewed in $K(\mathcal I)$,  \hpic {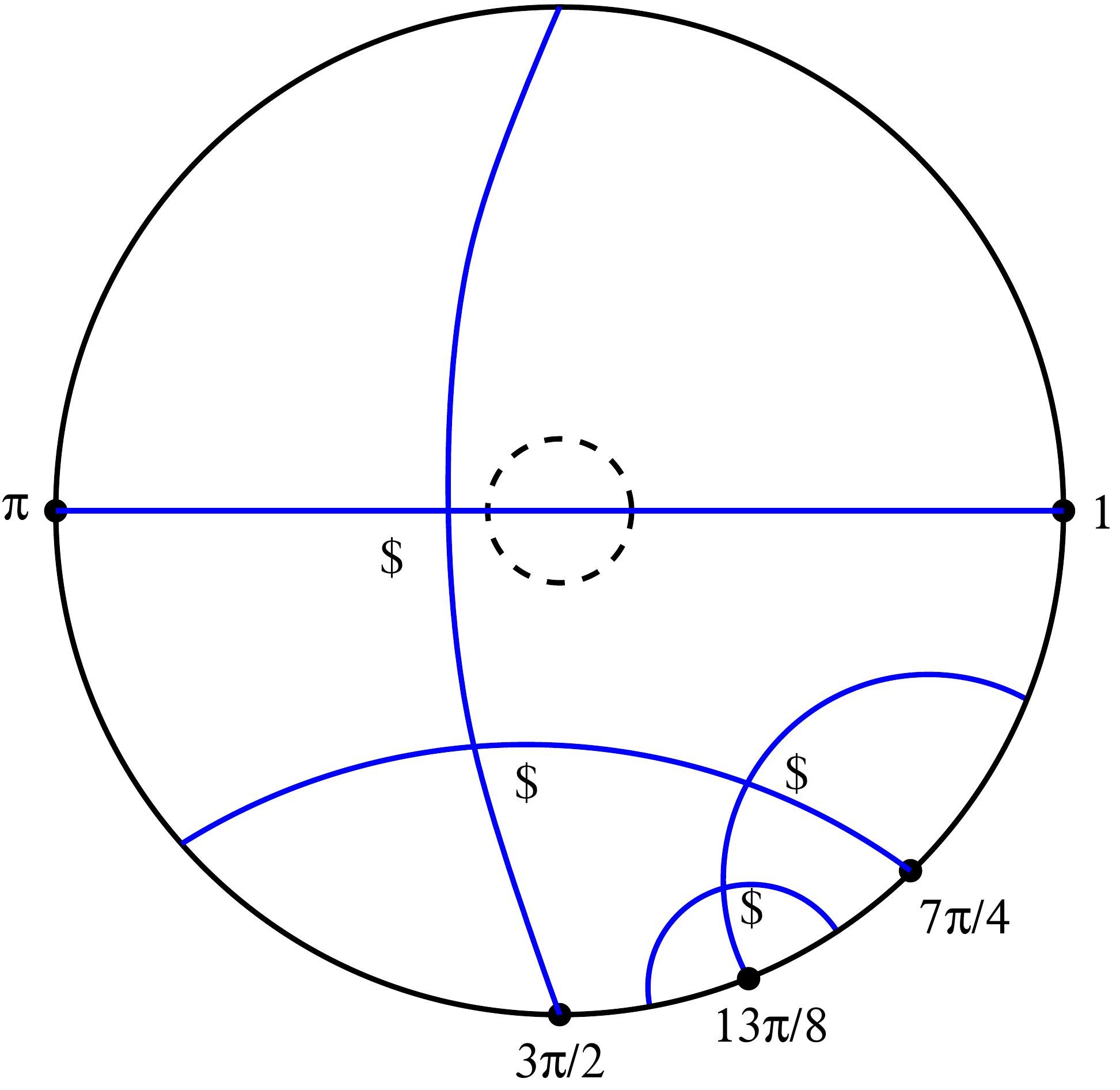} {2in}  \\
Picture 2:  $\pi_{\Xi, R}(g)\xi$:   \hpic {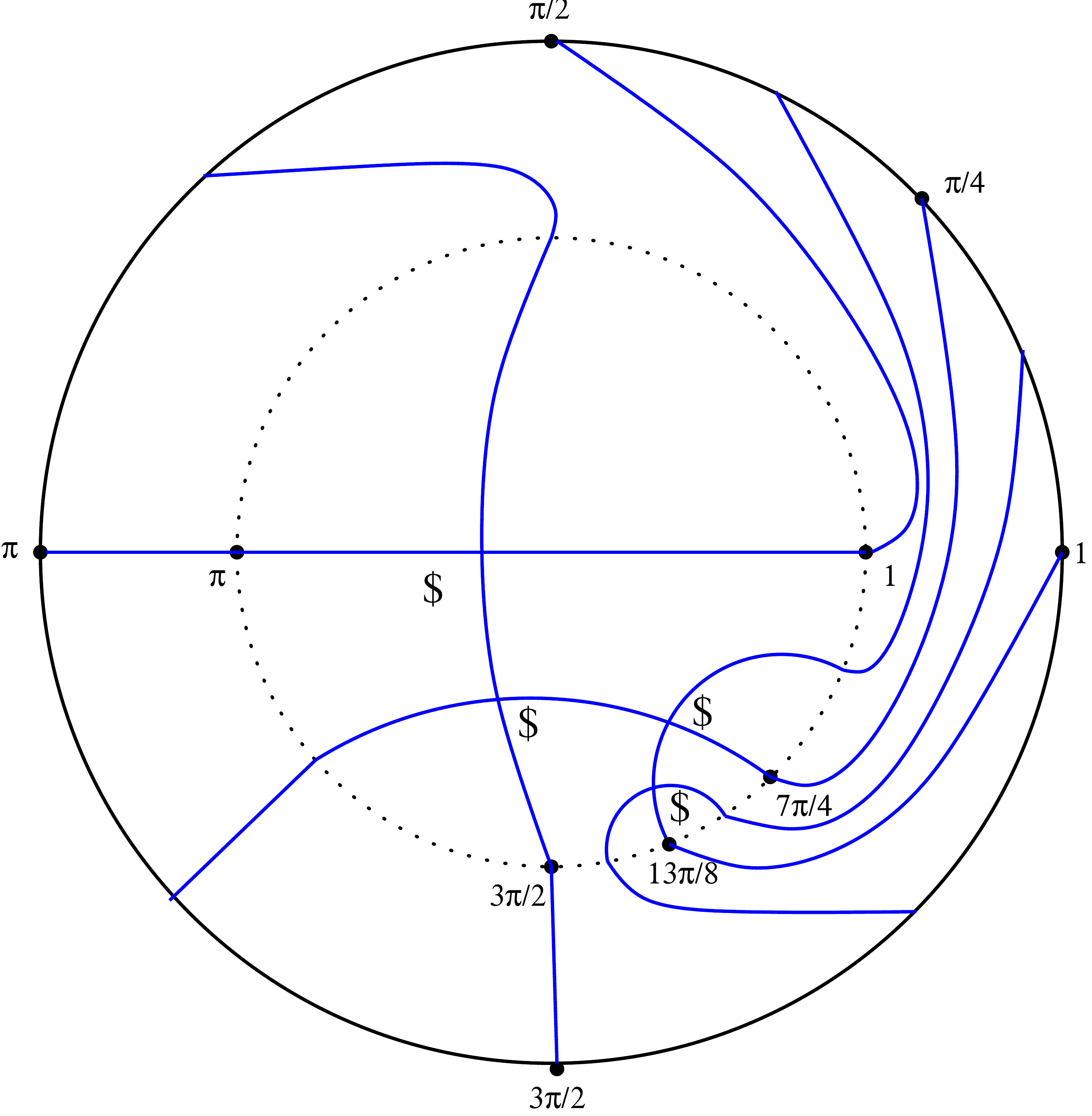} {2.6in}  \\
Picture 3:  $\xi$ viewed in $K(\mathcal J)$: \hpic {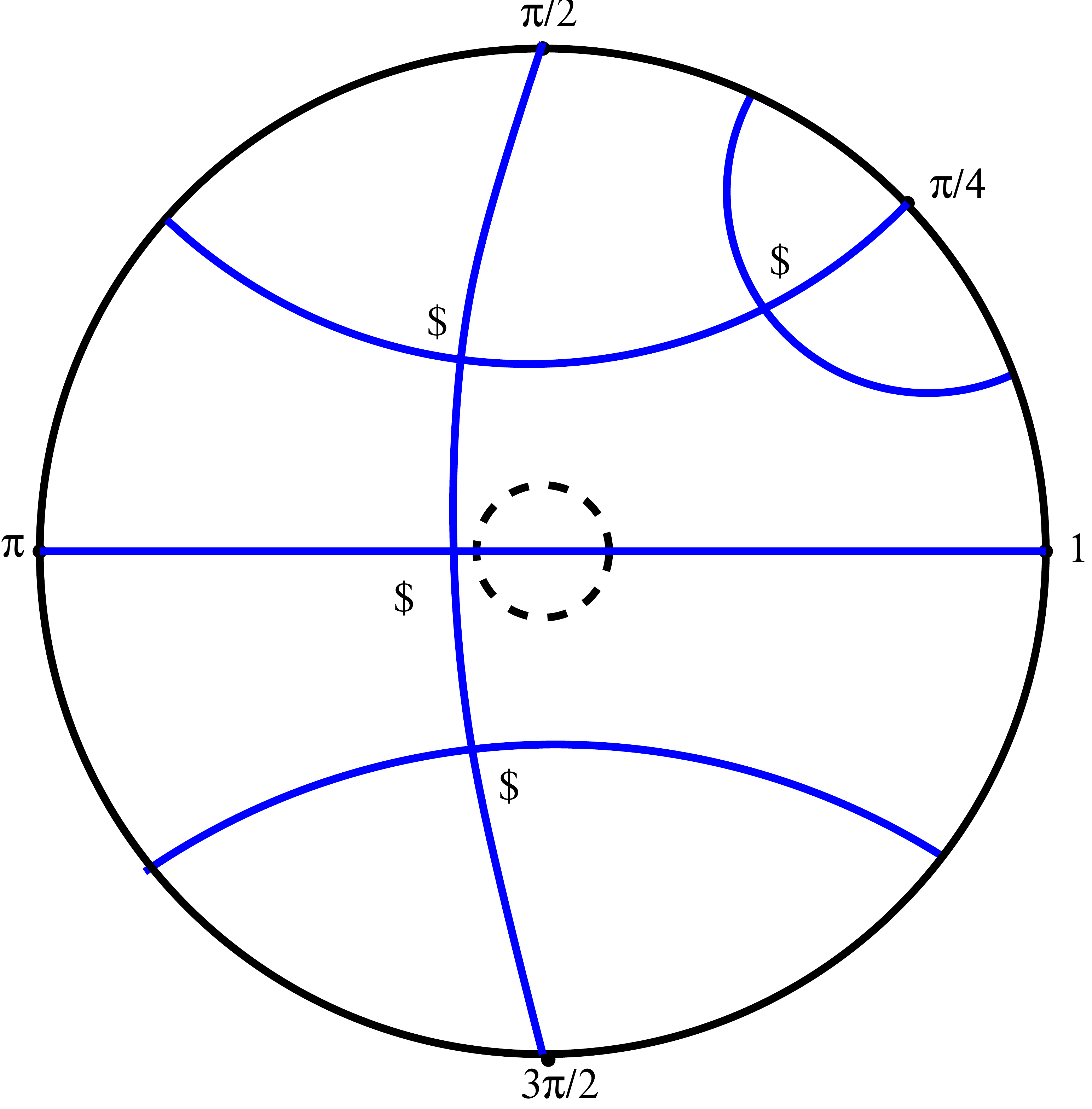} {2.6in} \\
Picture 4:  $\langle \pi_{\Xi, R}(g)\xi,\xi\rangle$:\\  \hpic {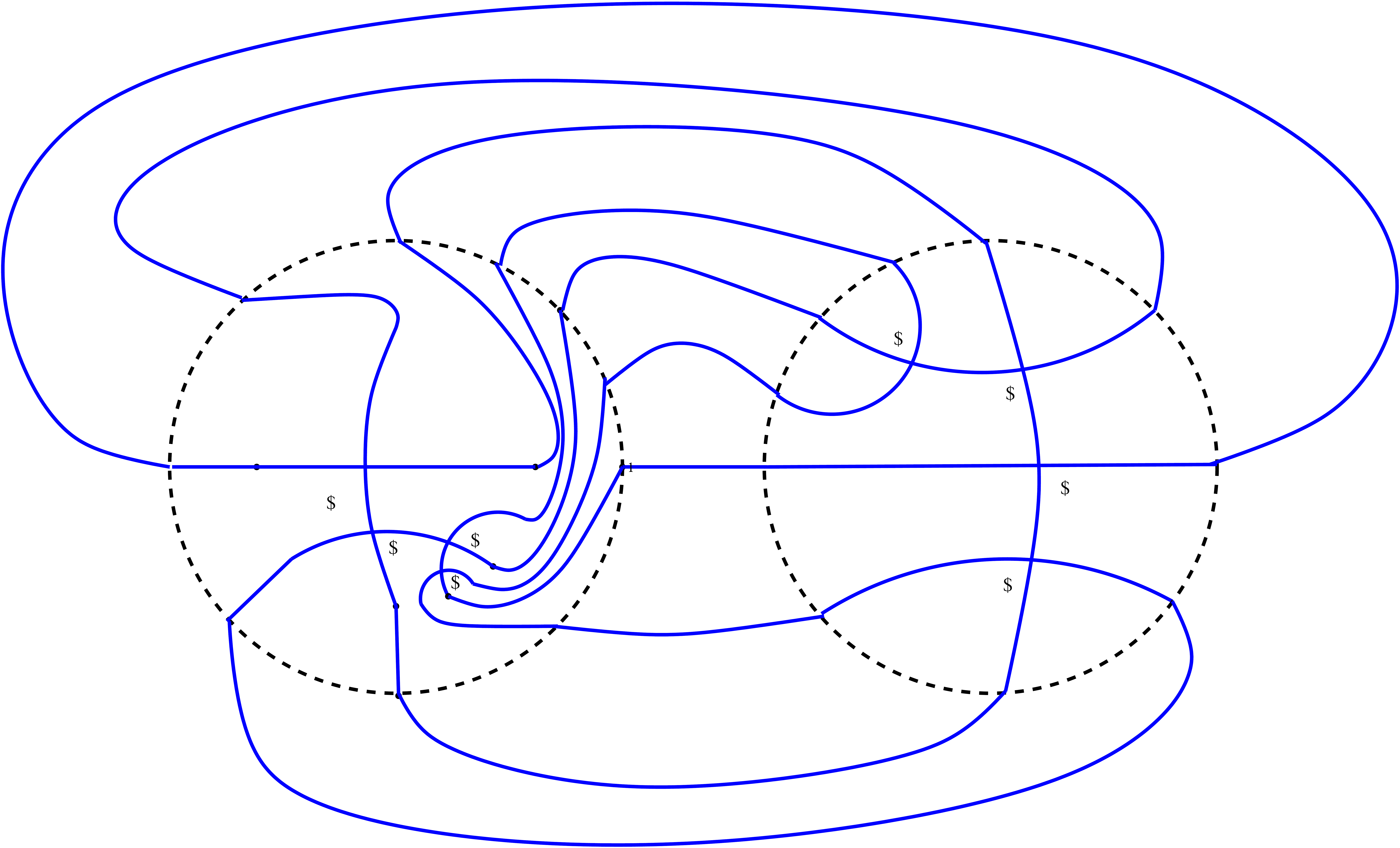} {2.6in}  \\

Where we note that the blown up crossings in the right part of the picture contain
$R^*$ and not $R$.\\
Picture 5:  The graph of shaded regions.:\\ \hpic {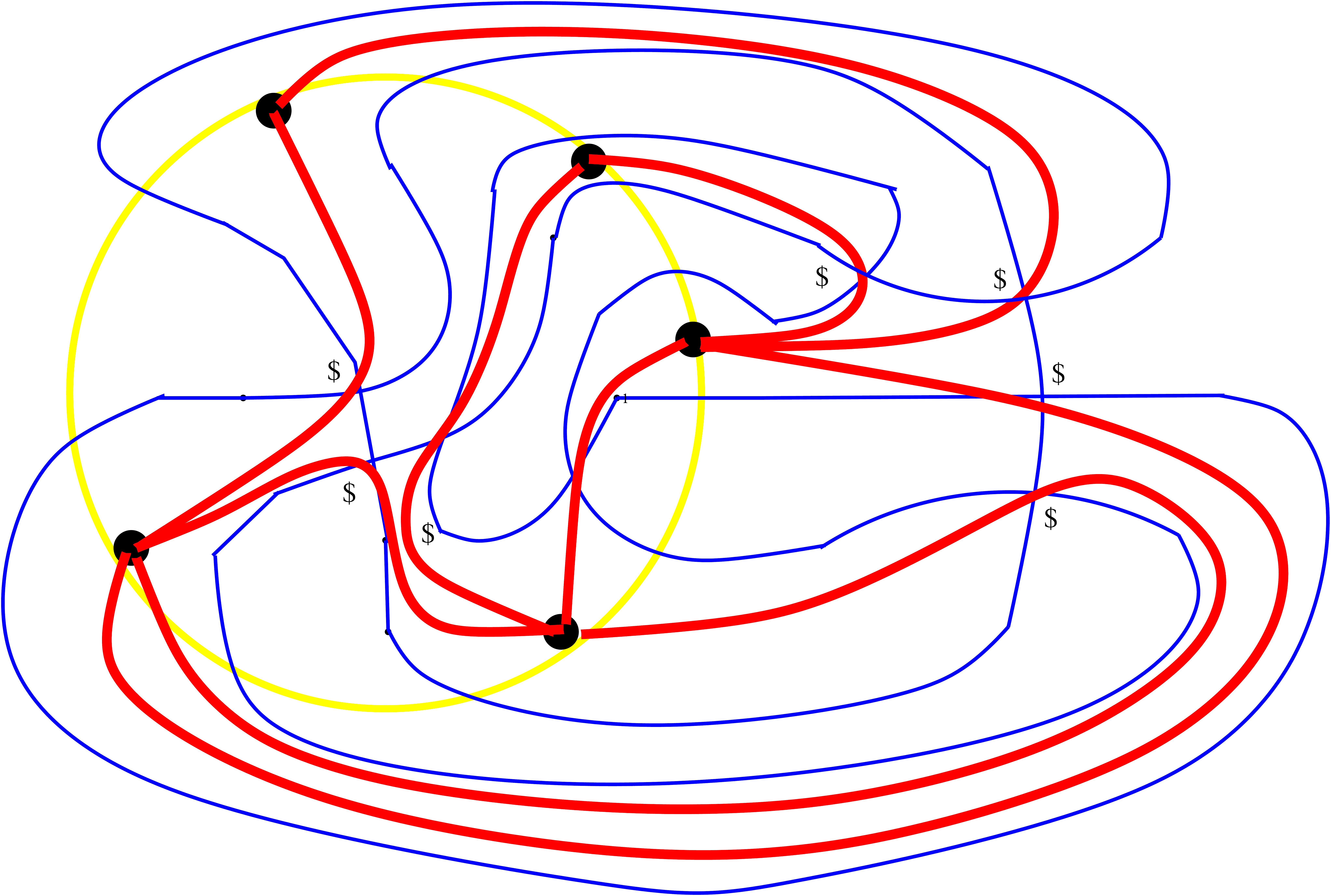} {2.6in} 

where we have drawn the picture so that the trees inside and outside the circle are clearly visible as the trees at the 
top and bottom of the picture at the end of section \ref{trees}

\section{Special choices of $R$}
\subsection{Dumb choices.}
The simplest possible choices of $R$ are  the tangles below which give equivalent results up to symmetry:\\

\hspace{0.5in} \hpic {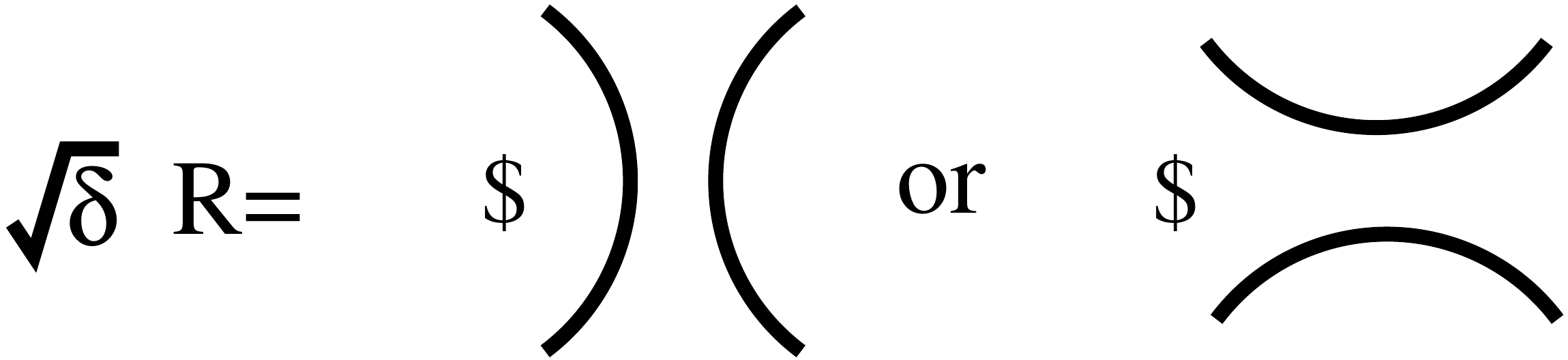} {0.8in} \\

Viewed in $H(\mathcal I)$, $\xi$ is, for the first choice of $R$, just a TL tangle where the boundary point at the bottom
is connected to the extreme right point (for $\mathcal I$) at the top, and all other boundary points at the top are connected
to their nearest neighbors (left or right depending on parity), as below:

\hspace{1in} \hpic {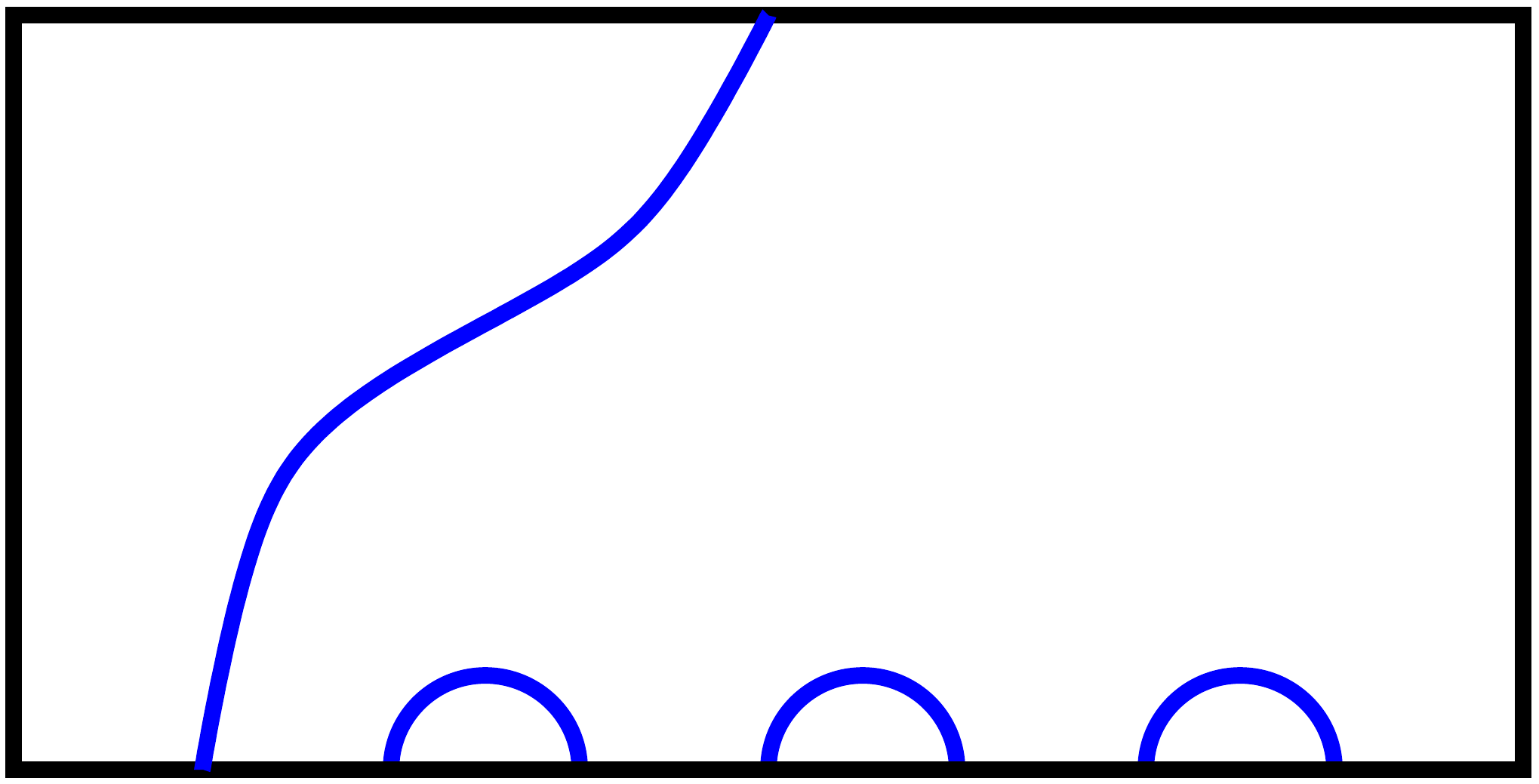} {1.5in}

 This property is invariant under $F$ so we see that 
the linear span of $g\xi$ is one dimensional and defines the trivial representation of $F$.

Other orbits of the action of $F$ on tangles are easy to analyze also. For instance if we take the 
rectangular representation of  $TL$ definite by a minimal projection with 6 boundary points altogether
we see that the image $TL$ element under a $g\in F$ is entirely determined by the image of the point $\frac{1}{2}$
under $g$, the other two boundary points being automatically connected to the extreme left point at the bottom
and the point to the right (at whatever scale) of $g(\frac{1}{2})$. 
These vectors, for $g,h\in F$, are orthogonal iff $g(\frac{1}{2})\neq h(\frac{1}{2})$ so we see that the representation
is just that on $\ell^2(\mbox{dyadic rationals})$. All actions on finite subsets of the dyadics can be obtained from 
minimal projections in higher $TL_n$'s. In particular all these representations are irreducible.

The situation for $T$ is similar. The vacuum vector is fixed by $T$ so defines the trivial representation.
The lowest weight 2 representation (spanned under $T$ by a one dimensional $V_2$ space, with $V_0=0$)
is the quasiregular representation on $\ell^2(T/F)$, and similarly for higher lowest weights. For $\ct$, one obtains
the same representations except that any character of the the center of $\ct$ may be induced (remember that
the central extension $\ct$ is split when restricted to $F$).

Note that choosing $R$ to be a normalized linear combination of the two choices above and taking the limit 
as one coefficient tends to zero we obtain the trivial representation as the limit along a curve of our representations
of $F$ and $T$.

\subsection {Chromatic choices.}\label{chromatic}

It is convenient in this section to use the "lopsided" version of the shaded Temperley Lieb planar algebra $\mathcal {TL}$ where closed strings
count $1$ if they are unshaded inside and $Q$ if they are shaded inside. Then a $\mathcal {TL}$  element in $TL_{0,+}$ is just a collection
of closed strings with the unbounded region unshaded and it is equal in $TL_{0,+}$ to $Q^{\#(\mbox{ shaded connected components})}$.
(The formulae for going between the lopsided and spherically invariant versions are detailed in \cite{jo2} in the discussion of spin models.)
We do need to be careful of how tangles with boundary points are closed, to the left or right.

It is easy to see that if we choose $\hat{R}$ to be:\\
\hspace{2in} \vpic {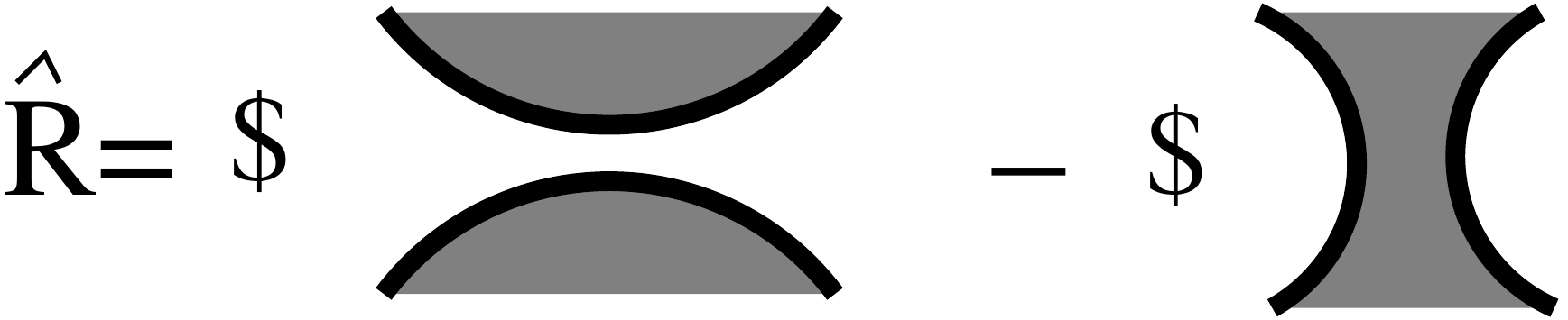} {2.5in} \\
then $\hat R$ is rotationally invariant and, given a planar graph $\Gamma$, the $R$-tangle $T_R(\Gamma) \in TL_{0,+}$ is
exactly the chromatic polynomial of $\Gamma$, in the variable $Q$.
(The chromatic polynomial $Chr_\Gamma(Q)$ of any graph $\Gamma$ is the unique polynomial in $Q$ whose value when $Q$ is
an integer is the number of ways of colouring $\Gamma$ with $Q$ colours.)

All this is very well known.

In order to use $\hat R$ to define a representation of $F$ we need to normalize it so that \ref{normalized} is satisfied. 
Thus we set $R=\frac{1}{\sqrt {Q-1}}\hat R$ and use that choice for the rest of this subsection. In particular for every rectangular
representation $V$ of $\mathcal TL$  we have the
unitary representations $\pi(g)=\pi_{V, R}(g)$ of Thompson's groups $F$ and $T$ (not just $\ct$ because
the trivial representation has rotational eigenvalue equal to $1$). Recall the representation $\Xi$ from definition
\ref{xi} with generating vector $\xi$. Note that in this lopsided version 
$\langle \xi, \xi \rangle =Q$.

\begin{proposition} If $T_+$ and $T_-$ are a pair of bifurcating trees defining an element of $F$ as explained in \ref{reps} then
$$\langle \pi(g)\xi,\xi\rangle= (Q-1)^{-n} Chr_{\Gamma(T_+,T_-)}(Q)$$
\end{proposition}
\begin{proof} This is visible for integer $Q$ if one uses spin models. Otherwise simply observe that $T_R(\Gamma) \in TL_{0,+}$
satisfies the same recursive relations as the Chromatic polynomial.
\end{proof}

\begin{remark}
The general linear combination of the two TL tangles in the choice of $R$ will yield the Tutte polynomial of a planar graph.
\end{remark}

The same results and notation apply to Thompson's  group $T$ and the graphs of section \ref{tcoeffs}.

The case $Q=2$ is rather special. Here the Chromatic polynomial takes only two values according to whether $\Gamma$ is 
bipartite of not. And the normalization constant of $\hat R$ is $1$. 

\begin{definition} Let $Q=2$:
$$\overrightarrow F = \{g\in F|\langle \pi(g)\xi,\xi\rangle=2\}$$
$$\overrightarrow T = \{g\in T|\langle \pi(g)\xi,\xi\rangle=2\}$$
\end{definition}
\begin{proposition} Both $\overrightarrow F$ and $\overrightarrow T$ are subgroups and the unitary representations $\pi_\Xi$
of $F$ and $T$ are the permutation representations on $\ell^2(F/\overrightarrow F)$ and $\ell^2(T/\overrightarrow T)$
respectively.
\end{proposition}

\begin{proof} Since $\langle \xi,\xi\rangle =2$ we have by the Cauchy Schwarz inequality that $g\xi=\xi$  iff $g$ is in the set specified.
Identification of the representations is obvious. \end{proof}

In general the representation $\pi_\Xi$ is unitary iff $Q\in \{4cos^2\pi/n:n=3,4,5,\cdots\}$ and as $Q\rightarrow \infty$  we have by the properties
of the Chromatic polynomial that  for every $g\in F$, $\langle \pi(g)\xi,\xi\rangle$ tends to $2$ for all $g\in G$. Thus we have:

\begin{proposition}\label{weak}
The trivial representations of $F$ and $T$ are in the weak closure of the $\pi$ defined above as $Q$ varies.
\end{proposition}

Note that both $\overrightarrow F$ and $\overrightarrow T$ have dual versions $\overleftarrow F$ and $\overleftarrow T$ obtained
by changing the shadings. It is not hard to find elements in $\overrightarrow F \cap \overleftarrow F$.

We can easily enumerate the elements of $\overrightarrow F$ (we leave the case of $\overrightarrow T$  to the reader).
For since the graph $\Gamma$ is bipartite, if we assign a $+$ sign to the leftmost vertex then all the other signs are
determined. And conversely if one fixes $n$ and gives a sequence of $n$ signs, any planar graph joining $n$ signed points
on a line satisfying \ref{conditions} will give an element of $\overrightarrow F$. Thus one may choose independently a
top and bottom tree $\Gamma_+$ and $\Gamma_-$ whose vertices are coloured by the sequence of signs, satisfying \ref{conditions}
and one can reconstruct an element of $\overrightarrow F$. 

Obviously the sequence of signs must begin with  $+-$ since the first two vertices are always connected in both $\Gamma_+$ and $\Gamma_-$.
It is easy to check that any sequence of signs beginning with $+-$ admits at least one pair $\Gamma_+, \Gamma_-$. Determination of exactly the number of Thompson group elements with a given sequence of
signs might not be easy but it can certainly be calculated. Here is a ``designer'' element of $\overrightarrow F$ starting from a sequence of signs:\\

\hspace{1in} \hpic {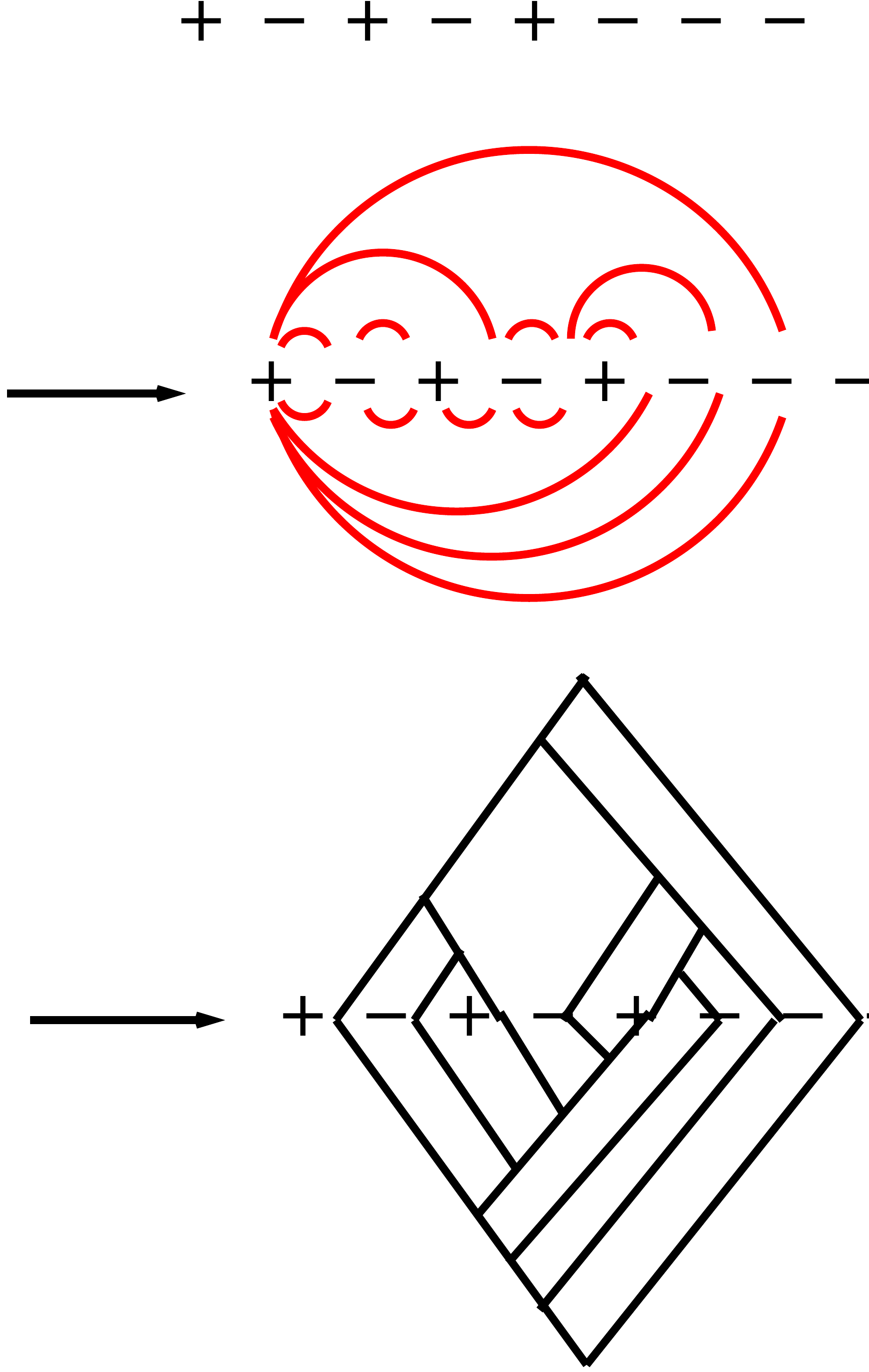} {2.5in} 
 
  It is obvious that $\overrightarrow F$
 is stable under the ``sum'' of elements in $F$ and we have the following observation of Sapir- note that the abelianisation of
 $F$ is $\mathbb Z^2$ with the abelianisation map being $g\mapsto (\log_2 g'(0), \log_2g'(1))$
 
 \begin{proposition} $\overrightarrow F $ is contained in the kernel of the homomorphism  $g\mapsto g'(1) \mod 2$.
 \end{proposition}
 \begin{proof} The log of $g'(1)$ for $g$ given by a pair $T_+,T_-$ of bifurcating trees is the number of boundary vertices at the of
 $T_+$ minus the same number for $T_-$. But each of these vertices defines an edge of $\Gamma$ so $\log_2g'(1)$ must
 be even.
 \end{proof}
 In fact $\overrightarrow F$ has a description in terms of the action of $F$ on subsets of $\overrightarrow F$. A sequence of signs as above 
 defines an element of $F/\overrightarrow F$ or a basis element of $\ell^2(F/\overrightarrow F)$. Given such a sequence of signs  and pair 
 of bifurcating trees we may assign to each left endpoint of the standard dyadic intervals of the top tree, the number $1$ for $+$ and $0$ for $-$.
 Under the embedding maps of the directed system this function does not change so in fact we get a function from the dyadic rationals in
 $[0,1]$ to $\{0,1\}$, i.e. a subset of the dyadics. The action of the Thompson groups on subsets of the dyadics is the same, by definition of
 $\pi(g)$ as the action we have thus defined on certain subsets of the dyadics. Moreover $\overrightarrow F$ is by definition the stabilizer
 of the sequence of signs defined by the identity element. This is easily seen to correspond to the subset of all dyadic rationals in $[0,1]$
 whose digit sum is odd when written as a binary expansion. (Thanks to Sapir for this  simplification of a clumsier earlier descripition).
 
 Golan and Sapir will supply more information on $\overrightarrow F$ in a forthcoming publication. In particular they will show that
 it is isomorphic to the Thompson group $F_3$ and that its commensurator is itself, thus showing that the representation of $F$ on
 $\ell^2(F/\overrightarrow F)$ is irreducible. (\cite{Sapir}.) 
 
\subsection{Knot-theoretic choices.}

We will work with the planar algebra $\mathcal C=(C_n)$ of linear combinations of Conway tangles (see \cite{conway},\cite{jo2}) where
we identify two tangles if they differ by a family of distant unlinked unknots. Thus a basis of $C_n$ is the set of 
isotopy classes (determined on diagrams by the Reidemeister moves) of Conway tangles with $2n$  boundary points
and no  unknots that can be isotoped to an arbitrarily small neighborhood of  a point on the boundary of the tangle.

$\mathcal C$ can be made into a *-planar algebra in more than one way but the * structure we will consider is the one
relevant to the Jones polynomial $V_L(t)$ when $t$ is a root of unity (or the Kauffman bracket when $A$ is a root of unity),
i.e. if $T$ is a Conway tangle as a system of curves with crossings in $\mathbb C$,
 $T^*$ is obtained from $T$ by complex conjugation in $\mathbb C$ thus:
 
 \hspace{0.5in} \hpic {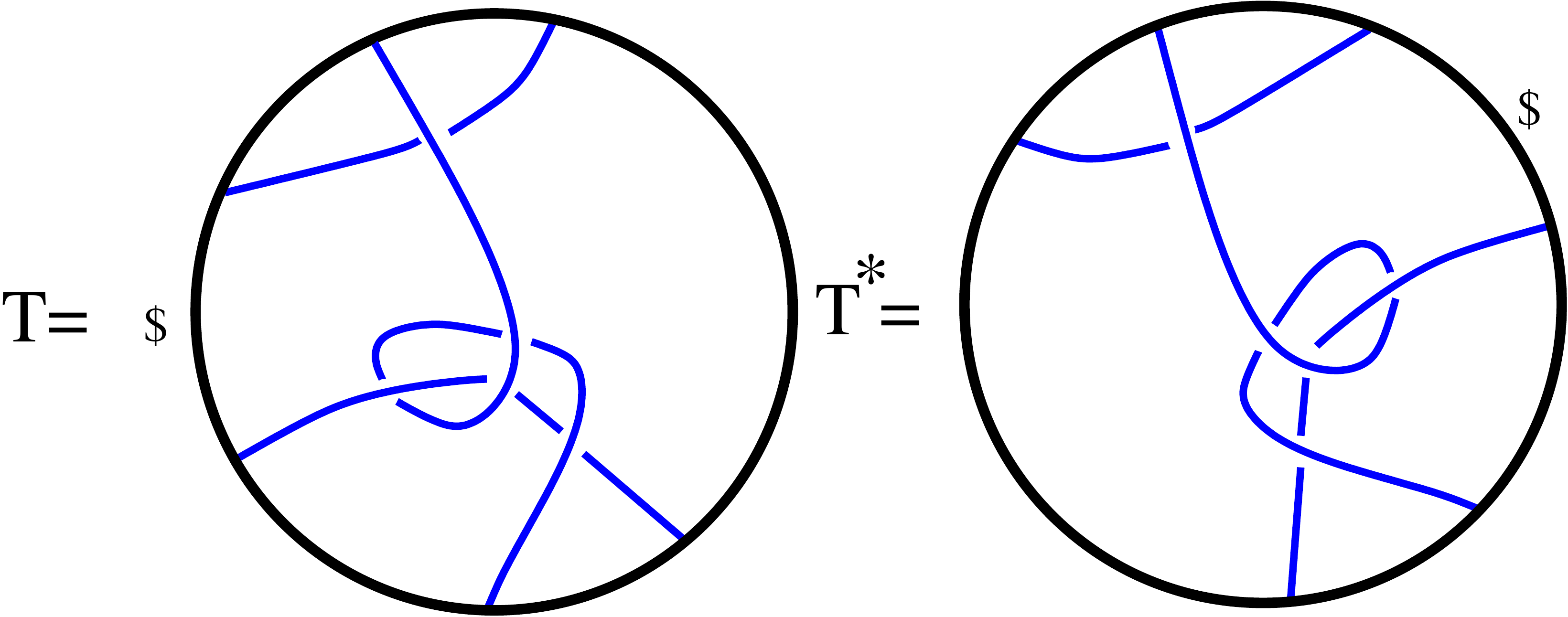} {1in}
 
and so if $R$ is as below, $R^*$ is as depicted:

\hspace{1in} \hpic {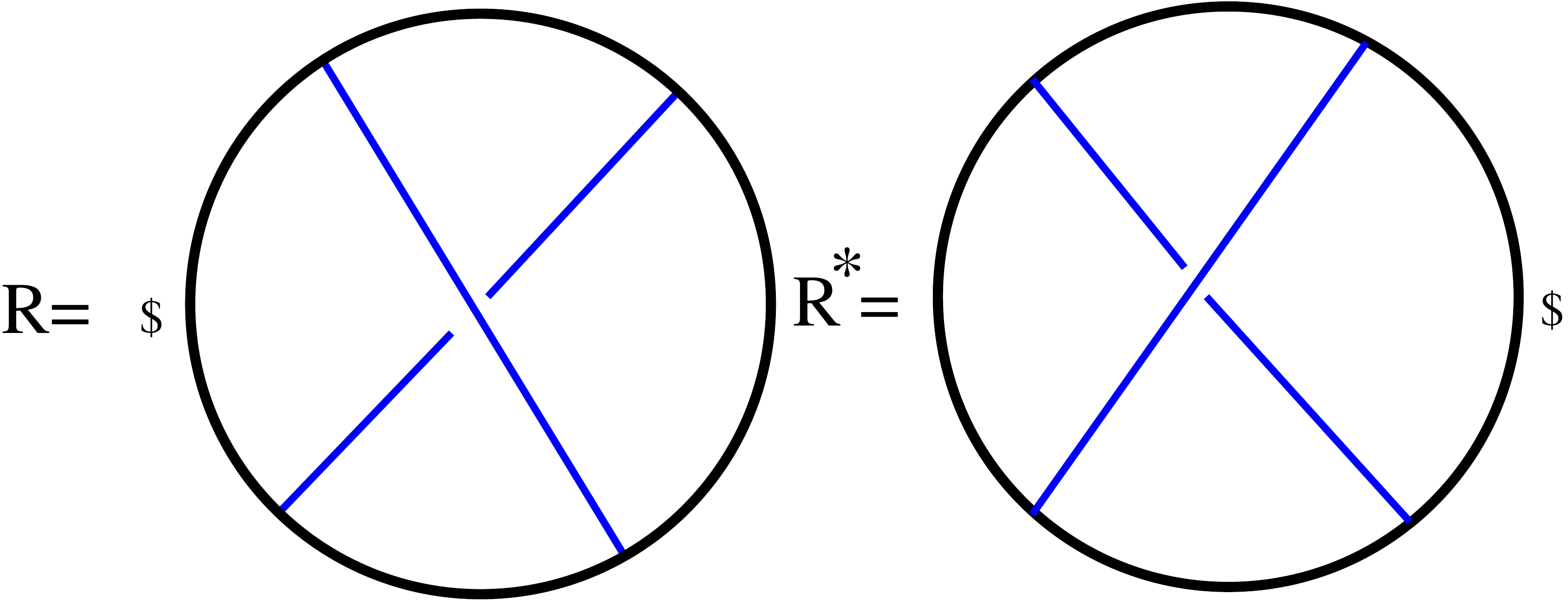} {1in}

 For the rest of this section we implicitly make this choice of $R$. (For $A\in \mathbb R$ for the Kauffman bracket the
 * is the identity on a crossing as above- this leads to a different theory for which all links arising as coefficients 
 are alternating-there is no hope of obtaining all alternating links as coefficients since alternating links have a 
 simple behaviour when more crossings are added and there are too many of them for Thompson group elements
 of a bounded length.)
The vector space $C_0$ is infinite dimensional, spanned by isotopy classes of links with no distant unknots. It is 
 in fact an algebra under the obvious tangle multiplication, and each $C_n$ has a
$C_0$-valued sesquilinear inner product $\langle S, T\rangle$ defined by the tangle as in section \ref{planar}.
 
 We consider the rectangular representation $\Xi$ of $\mathcal C$-see definition \ref{xi}.
  Note that because of our convention
 with distant unknots, $R$ automatically satisfies \ref{normalized} so that $\pi_\Xi$ preserves the inner
 product $\langle, \rangle$.

\subsubsection{All unoriented links arise as coefficients of representations.}
For simplicity, for the rest of this section, links will be considered the same if they differ by distant unknots.
\begin{theorem}\label{all}
Given any (unoriented) smooth link L in $\mathbb R^3$ there is an element $g\in F$ such that 
$\langle \pi_{\Xi} (g) \xi,\xi\rangle = L$.
\end{theorem}
\begin{proof}
The proof will proceed by a series of definitions and lesser results. 

It is worth pointing out before we begin that the procedure we give for producing a Thompson group element will not
produce any new distant unknots so we could renormalize our Conway skein planar algebra in such a way as to
obtain exactly the link we begin with, i.e. not up to a distant union of unknots. 

We saw in section \ref{coeffs} how to go between 4-valent planar graphs and planar graphs by shading
the regions of the four-valent graph. An unoriented  link diagram consists of an underlying 4-valent planar graph
together with crossing data. Moreover, given the shading, crossings have a sign according to the convention:

\hspace{1in} \hpic {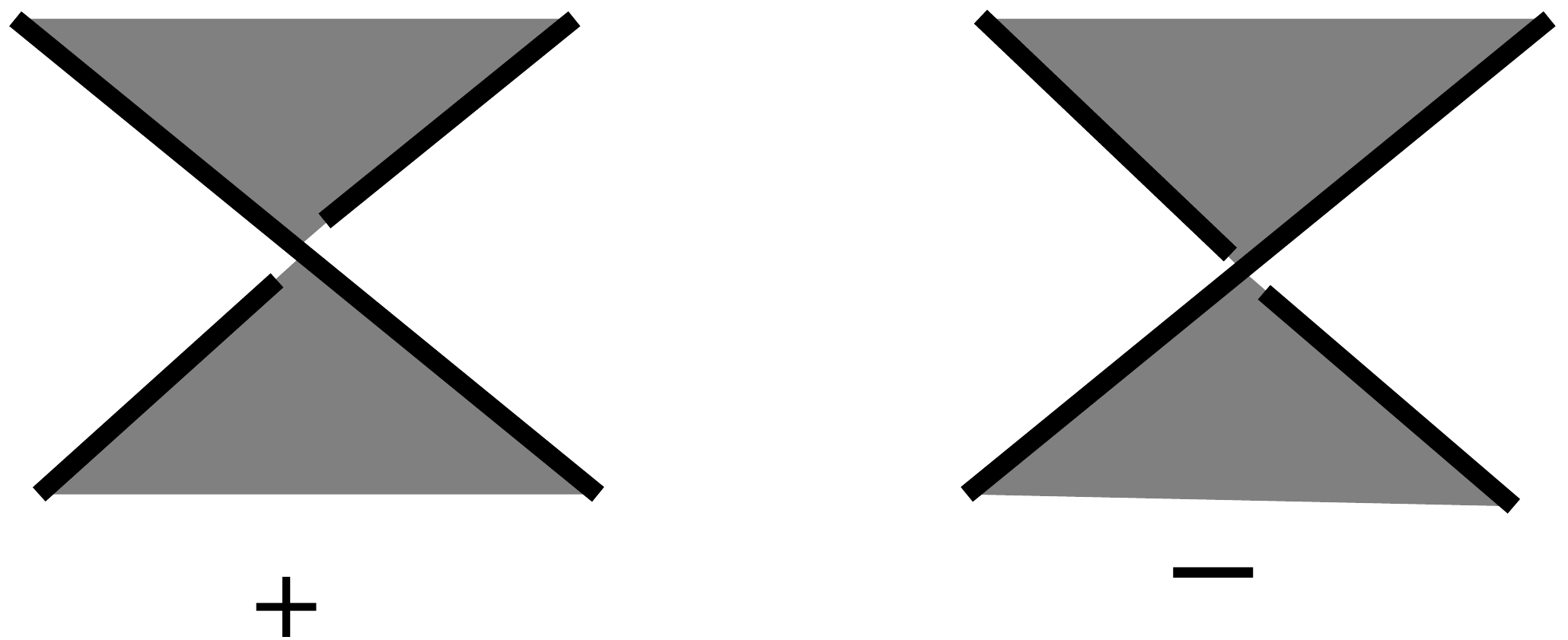} {1.5in}

Thus a link projection gives an edge-signed planar graph.

\begin{definition} The signed planar graph defined above is called the \emph{semidual} graph $\Gamma(L)$ of the link diagram $L$.
\end{definition}
Here is an illustration of the formation of $\Gamma(L)$:

\hspace{1in} \hpic {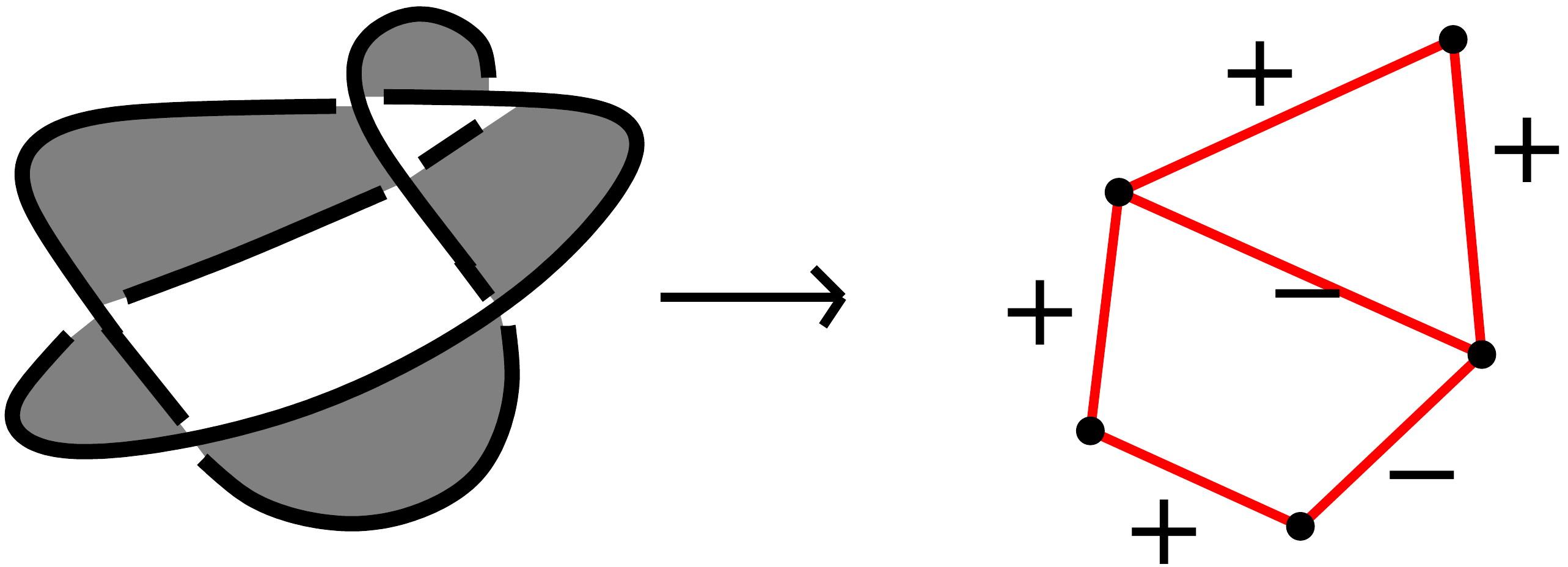} {1in}
\vskip 5pt
If  $\Gamma$ is a signed planar graph  the \emph {medial link diagram} $L(\Gamma)$ is what we called the $R$-tangle of $\Gamma$ (for
our choice of $R$)
formed by putting
crossings on the middles of all the edges of $\Gamma$. The  sign of a crossing is determined by the sign on the edge and the convention above, where the
shaded regions are defined by the vertices of the graph. Then the crossings are joined around the edges of each face.

Here is an illustration of the formation of the medial link diagram:\\

\hspace{1in} \hpic {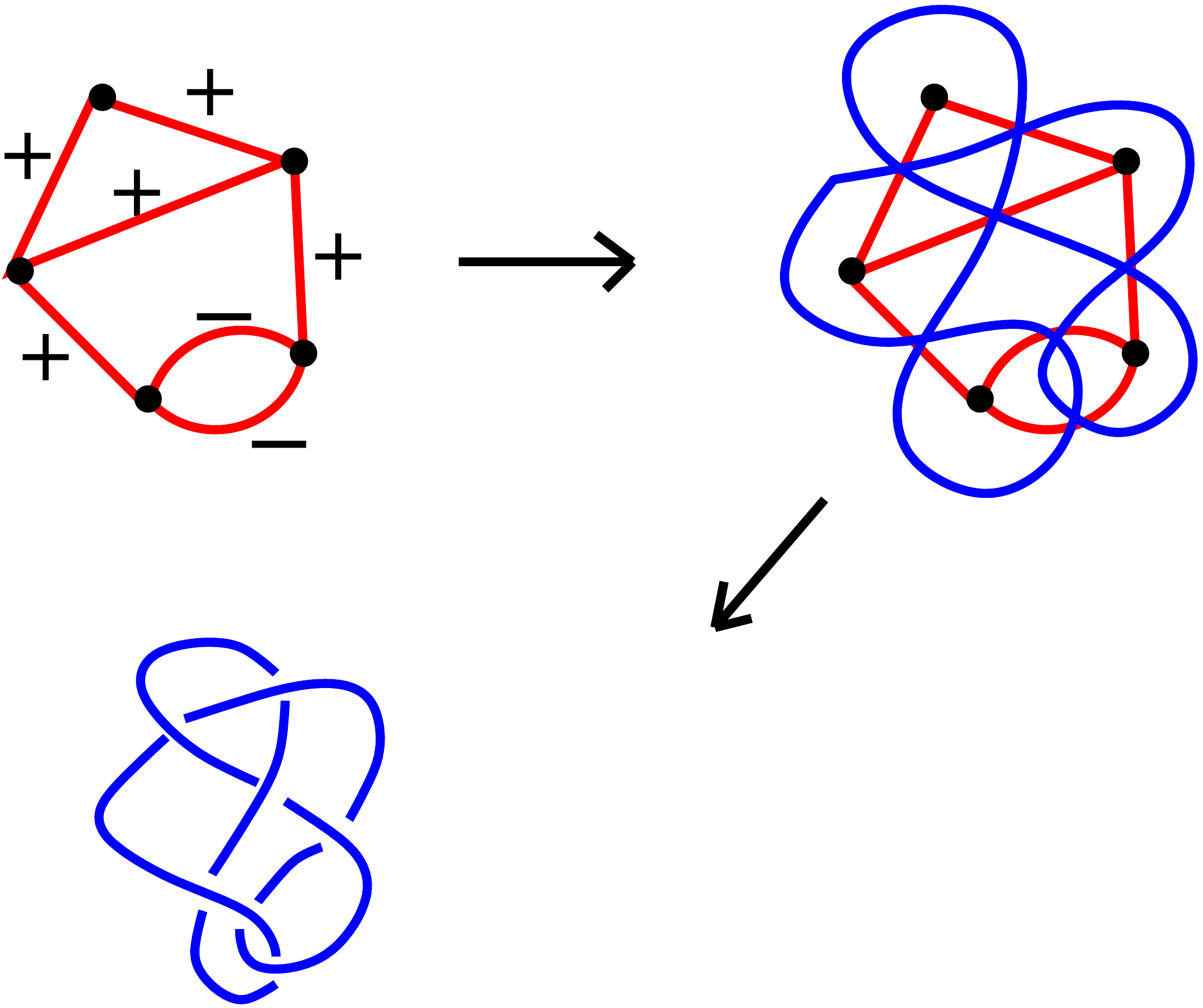} {2in}

Any link diagram is thus the medial link diagram of a signed planar graph.
\vskip 10pt
Now consider the following three local moves on signed planar graphs: 

Type I:   \hspace{0.4in} \hpic {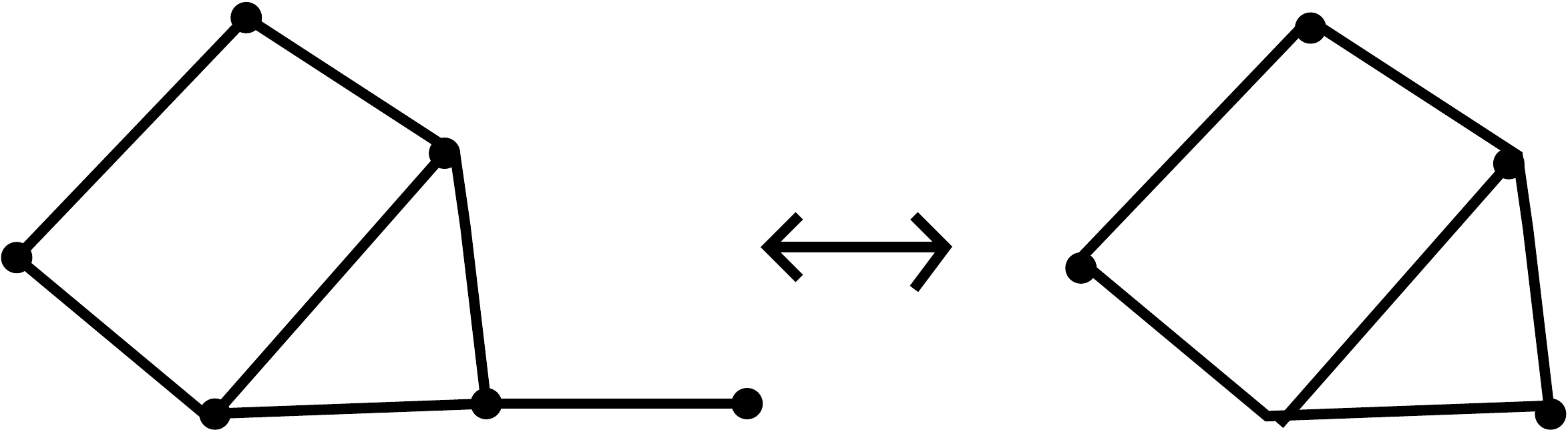} {1in}

Here a 1-valent vertex and its edge are eliminated. The signs on all the edges are arbitrary.\\

Type IIa:   \hspace{0.4in} \hpic {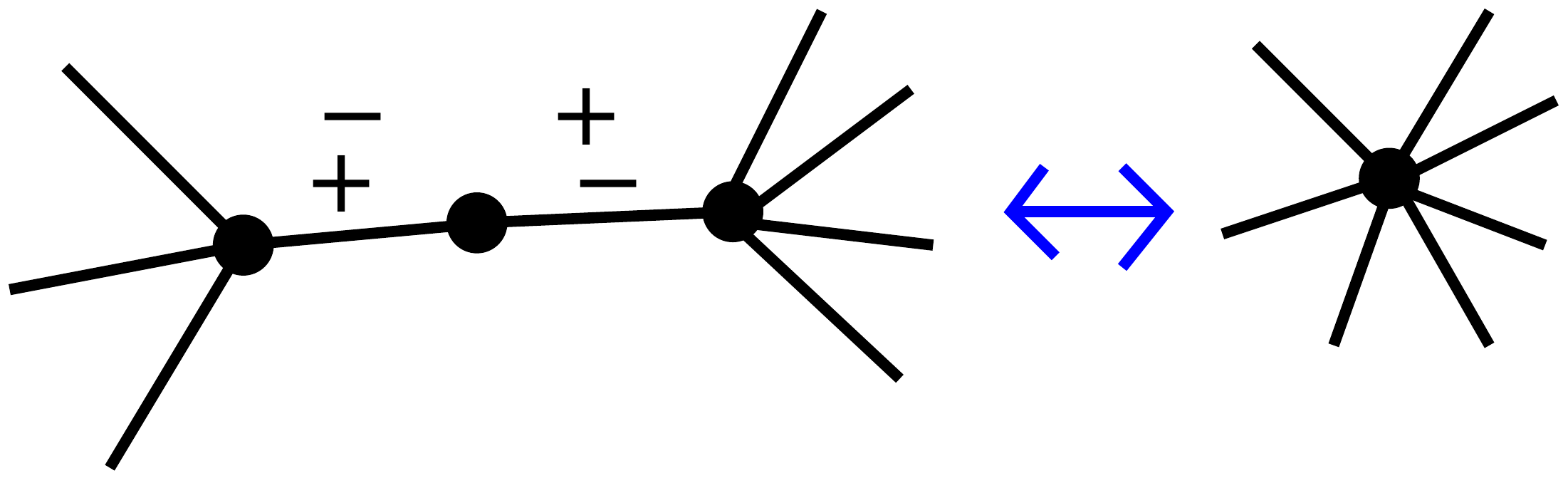} {1in} \\

Here a 2-valent vertex whose two edges have opposite signs is eliminated along with those edges, and the vertices
at the other ends of the removed edges are fused into one. All other signs are arbitrary.

Type IIb: \hspace{0.4in} \hpic {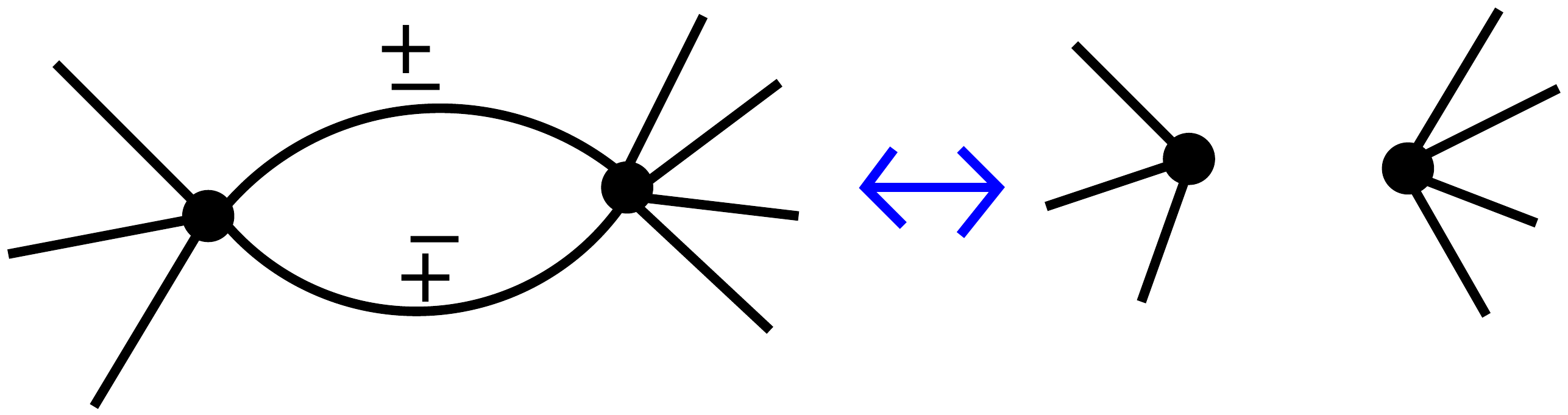} {1in} \\

Here two edges joining the same pair of vertices, with opposite signs, with no other edges or vertices in between are simply eliminated.
All other signs are arbitrary.

\begin{proposition} Two signed planar graphs $\Gamma_1$ and $\Gamma_2$  give the same link in $\mathbb R^3$ if they differ by planar isotopies 
and any  of  the above three moves.
\end{proposition}.
\begin{definition} We say two signed graphs are 2-equivalent if they are as in the proposition.
\end{definition}

Now suppose we are given two bifurcating trees $T_+$ and $T_-$ with $N+1$ leaves  defining an element $g$ of $F$  as
in \ref{thompsonfromtrees}. By our calculation of a coefficient of $\Xi$ in section \ref{coeffs}, we know that
$\langle \pi_{\Xi} (g) \xi,\xi\rangle$ is a link with a diagram of a special form, namely  the medial link diagram of
a signed planar graph satisfying the conditions of proposition \ref{conditions}, with edges in the upper half plane
having sign $+$ and edges in the lower half plane having sign $-$.

\begin{definition} A  planar graph $\Gamma$ will be called \emph{standard} if its vertices are the points $(0,0),(1,0),(2,0),(3,0),..., (N,0)$ and each edge $e$
can be parametrized by a function $(x_e(t),y_e(t))$ with $x_e'(t)>0 \quad \forall t \in [0,1]$ and either $y_e(t)>0 \quad \forall t \in (0,1)$ or 
$y_e(t)<0 \quad \forall t \in (0,1)$

\end{definition}

\begin{lemma} Any link $L$ is the medial link of a standard edge-signed planar graph.
\end{lemma}
\begin{proof}This is clear from what we have described.
\end{proof}
Now suppose we are given two bifurcating trees $T_+$ and $T_-$ with $N+1$ leaves  defining an element $g$ of $F$  as
in \ref{thompsonfromtrees}. By our coefficient calculation in section \ref{coeffs}, we know that
$\langle \pi_{\Xi} (g) \xi,\xi\rangle$ is a link with a diagram of a special form, namely  the medial link diagram of
a signed planar graph satisfying the conditions of proposition \ref{conditions}, with edges in the upper half plane
having sign $+$ and edges in the lower half plane having sign $-$. Moreover by lemma \ref{catalan}, any 
signed planar graph as we have just described comes from a pair $T_+, T_-$.
Thus theorem \ref{all} will be proved if we can show that any standard signed planar graph is  2-equivalent to a graph
of the form we have just described.

The edges of a standard planar graph acquire  orientations so that the source vertex of $e$ has smaller $x$ coordinate than the target vertex of $e$.
\begin{definition} If $\Gamma$ is a signed standard planar graph set \\
 $e^{up}= \{e\in e(\Gamma): y_e(t)>0 \mbox{  for  } 0<t<1\}$ and\\ $e^{down}=\{e\in e(\Gamma): y_e(t)<0 \mbox{ for  } 0<t<1\}$.
 
 Further define $e^{up}_\pm$ to be those edges in $e^{up}$ with sign $\pm 1$ respectively. Similary for $e^{down}_\pm$.
 
 Note that $$e(\Gamma)=e^{up}_+(\Gamma)\sqcup e^{up}_-(\Gamma)\sqcup e^{down}_+(\Gamma)\sqcup e^{down}_-(\Gamma).$$
 \end{definition}
 
 \begin{definition}
Given a vertex $v\in \Gamma$ $e_v^{in}$ and $e_v^{out}$  will be the set of all edges $e$ with target and source equal to $v$ respectively.
\end{definition}

\begin{definition} A standard signed planar graph will be called \emph{Thompson} if $e^{up}_-(\Gamma)=\emptyset=e^{down}_+(\Gamma)$ and it has the property that $|e_v^{in}\cap e^{up}_+|=|e_v^{in}\cap e^{down}_-|=1$ for all vertices 
other than $(1,0)$. 
\end{definition}

\begin{lemma}\label{treetotree} If $g\in F$, $\Gamma(\langle \pi_{\Xi} (g) \xi,\xi\rangle)$ is Thompson, and if $\Gamma$ is Thompson there is a $g\in F$ with $\Gamma(\langle \pi_{\Xi} (g) \xi,\xi\rangle)=\Gamma$
\end{lemma}
\begin{proof}The condition of being Thompson is just the translation of the fact that the pair of trees, one in the upper half plane, and one
in the lower, satisfy the conditions of proposition \ref{conditions}, with signs as specified.
\end{proof}

\begin{definition}
The badness $TB(\Gamma)$ of a standard signed planar graph $\Gamma $ is 
$$TB(\Gamma)=\sum_{v\in v(\Gamma)\setminus \{(0,0)\}} \mbox{\Big(}|1-|e_v^{in}\cap e^{up}||+ |1-|e_v^{in}\cap e^{down}|| \mbox {\Big )}+|e^{up}_-| +|e^{down}_+|$$
\end{definition}

\begin{proposition}
If $TB(\Gamma)=0$ then $\Gamma$ is Thompson.
\end{proposition}
\begin{proof} This is immediate from the definition of  $TB$.
\end{proof}

\begin{lemma} \label{bad}Given a signed standard planar graph $\Gamma$ with $TB(\Gamma)>0$, there is another one one $\Gamma'$ which is 2-equivalent to it and
with $TB(\Gamma')<TB(\Gamma)$
\end{lemma}
\begin{proof}
Case (1):Suppose there is a vertex $v$ different from $(0,0)$ with $e_v^{in}=\emptyset$. Then $v$ and the vertex $w$ to the left of it
are as below:\\
\vpic {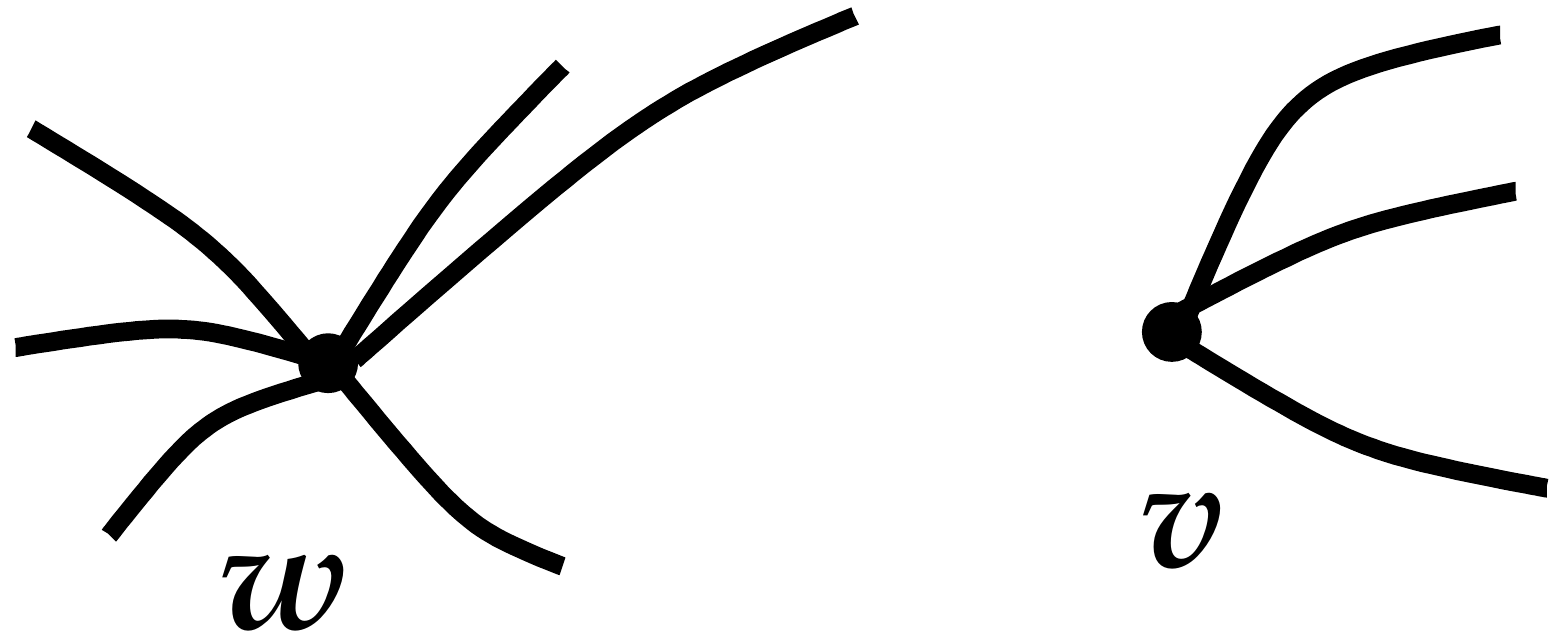} {2.5in}  \\
Now add two edges as below to obtain $\Gamma'$\\
\vpic {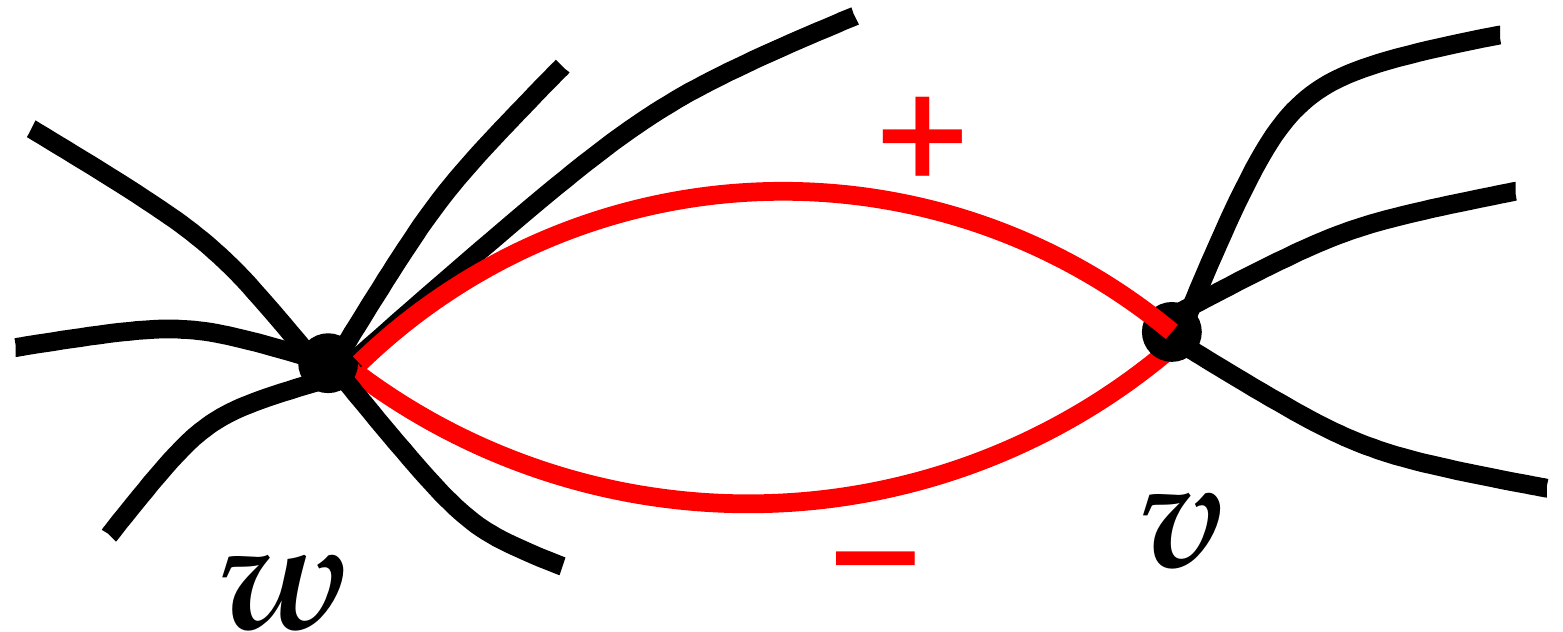} {2.5in}  \\
 $TB(\Gamma')= TB(\Gamma)-2$  and $\Gamma'$ is 2-equivalent to $\Gamma$.
 
 Case (2): Suppose $\Gamma$ has a vertex $v$ with $|e_v^{in}|=1$. Wolog we may assume that
 the incoming edge to $v$ is up. It may be $+$ or $-$. Then near $v$ the situation is as below:\\
 \vpic {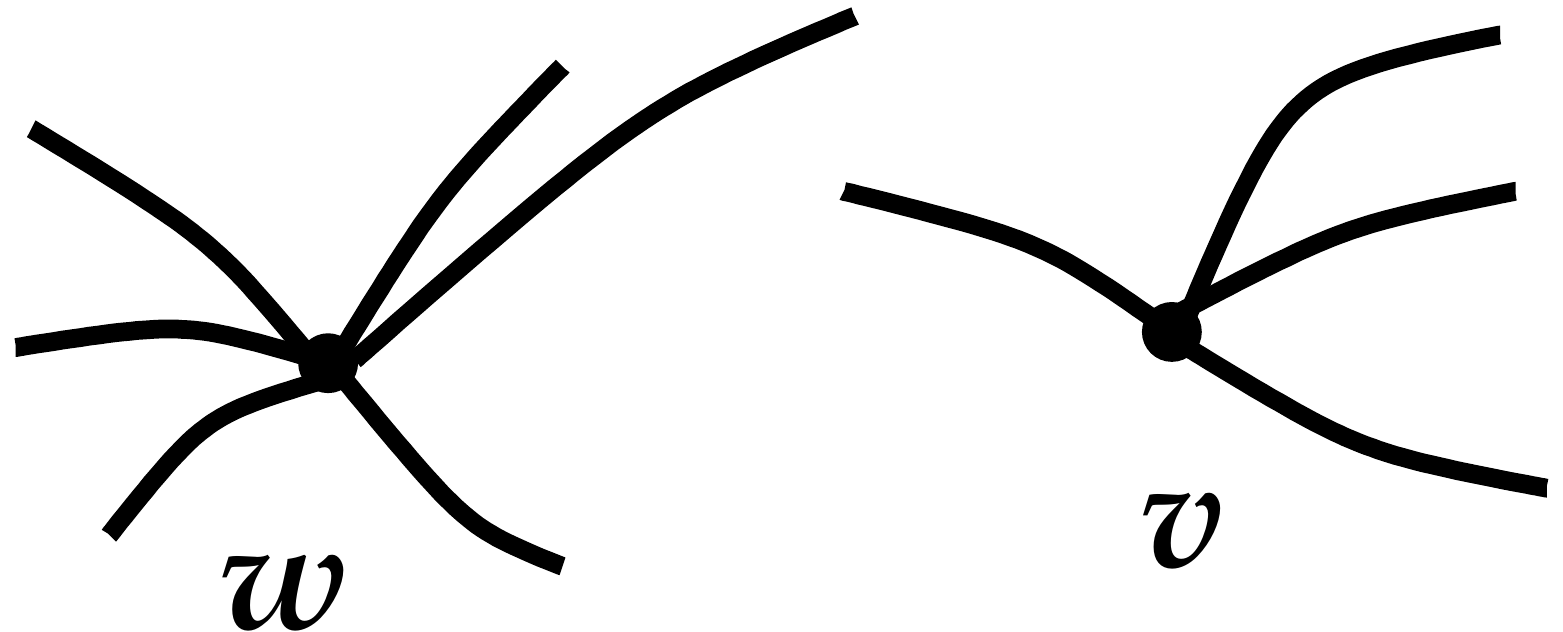} {2in} \\
 
 where we have labelled by $e$ and $f$ the edges whose positions we will change. In fact $f$ may not
 exist, which changes nothing for the argument. 
 
 Now add one vertex and three edges as below to obtain $\Gamma'$ which is clearly standard, 2-related
 to $\Gamma$ and $TB(\Gamma')=TB(\Gamma)-1$.\\
 
  \hpic {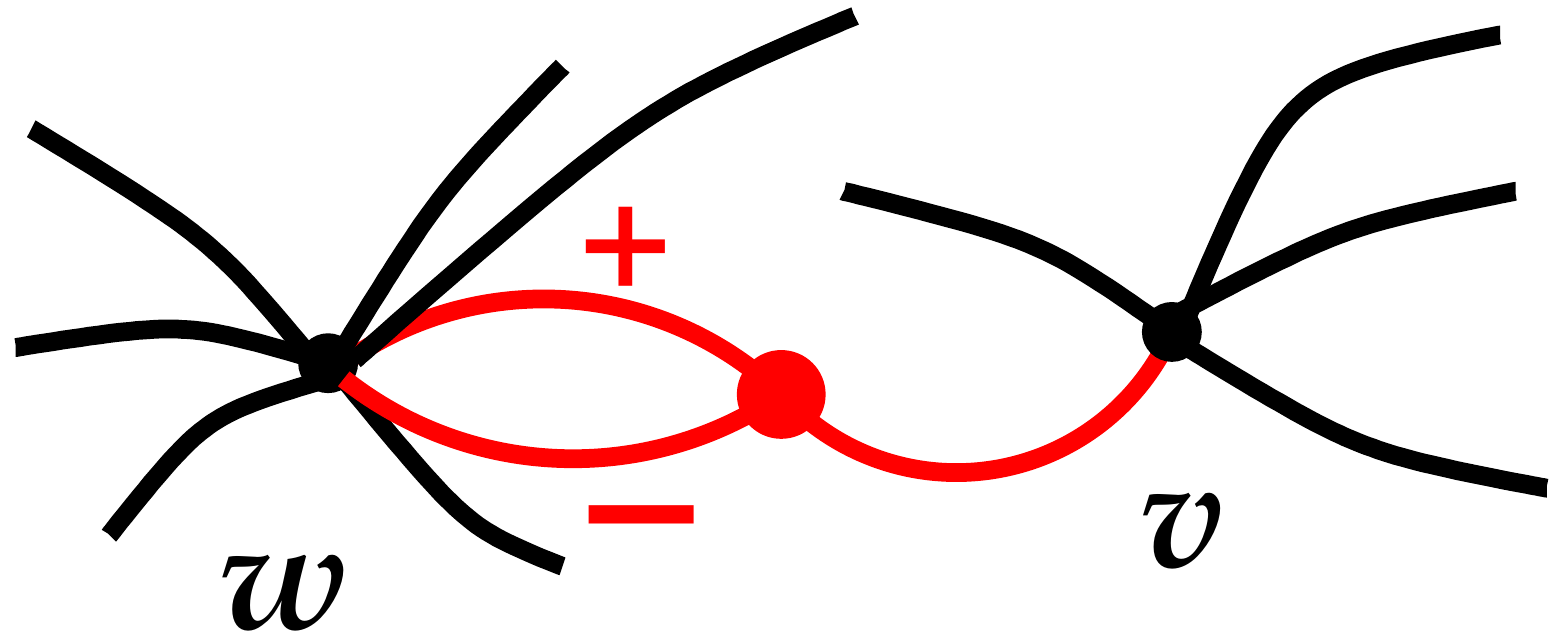} {2in} \\

 Case(3) Suppose $\Gamma$ has a vertex $v$ with $|e_v^{in}\cap e^{up}|>1$ or $|e_v^{in}\cap e^{down}|>1$. Wolog suppose it's $|e_v^{in}\cap e^{up}|>1$.
 Then near $v$, $\Gamma$ is as below:\\
 
 \hpic {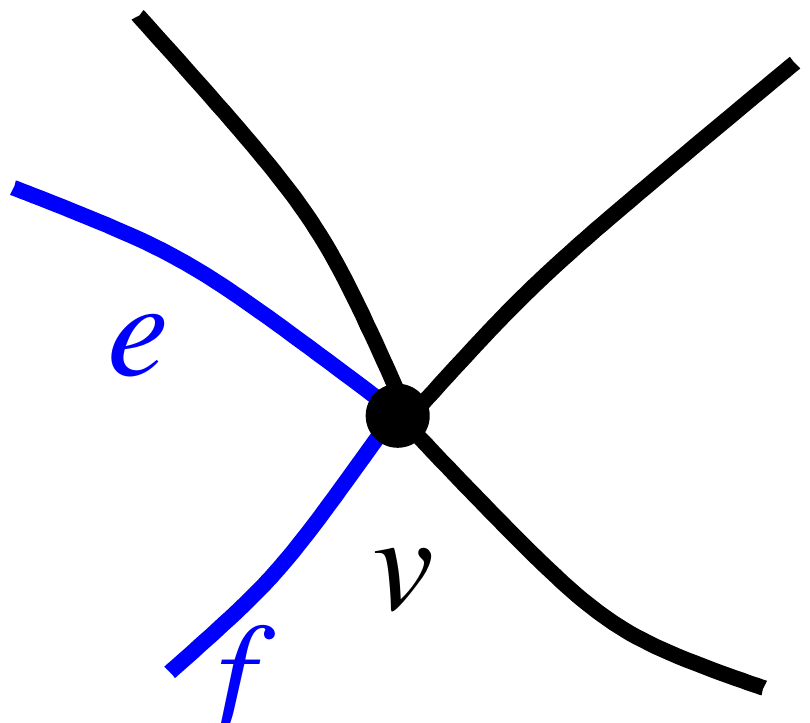} {1.5in} \\
 
 Now add $3$ vertices and $5$ edges as below to obtain $\Gamma'$:
 
 \hpic {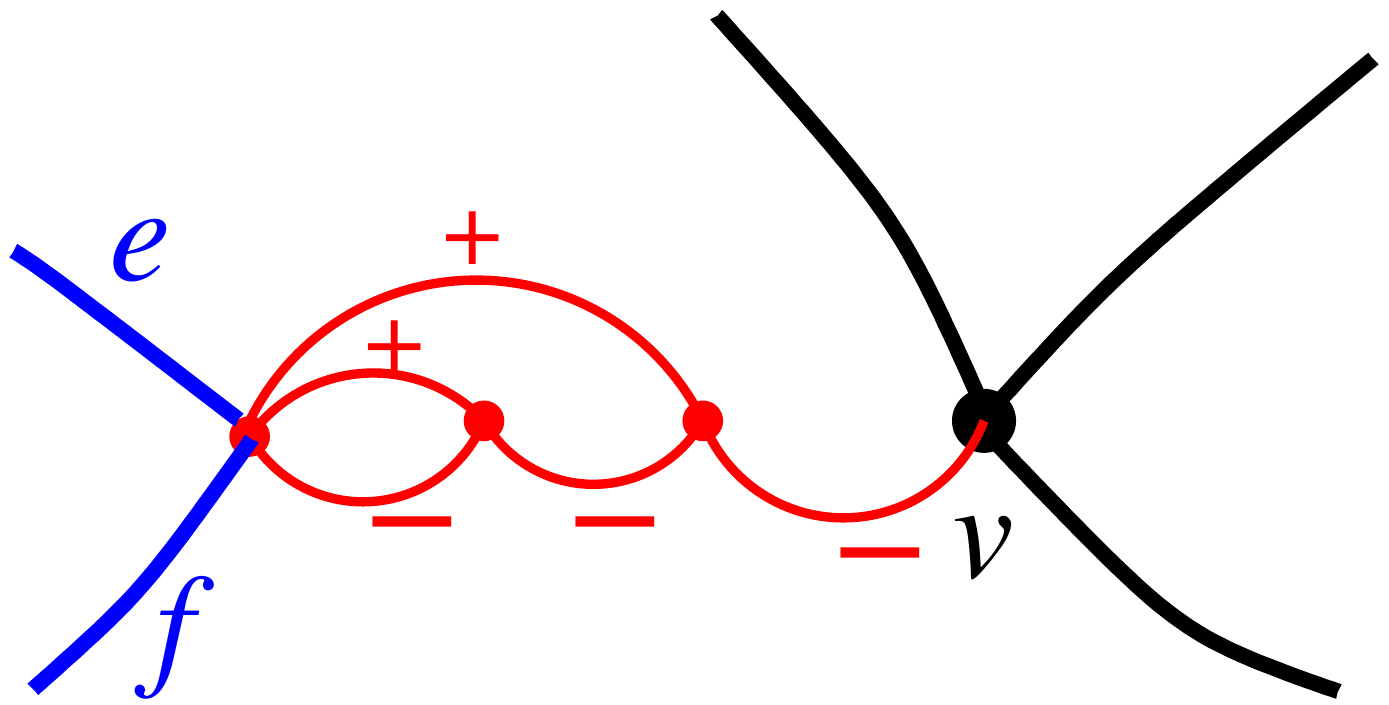} {1.5in} \\
 
 $\Gamma'$ is manifestly standard and a few applications of type I and type II moves show that
 $\Gamma'$ is 2-equivalent to $\Gamma$. And one of the offending top incoming edges at
 $v$ has been assigned to another vertex where it is the only top incoming edge. 
 All edge counts at other vertices are either as they were or do not change $TB$. Hence $TB(\Gamma')=TB(\Gamma)-1$.

 Case (4): All vertices except $(0,0)$ have 2 incoming edges, one up and one down. Then since $TB(\Gamma)>0$, there must be an edge in $e^{up}_- \cup e^{down}_+$. Wolog we can assume it
 is in $e^{up}_-$. If the target of that edge is $v$ then near $v$, $\Gamma$ looks like:\\
 
 \hpic {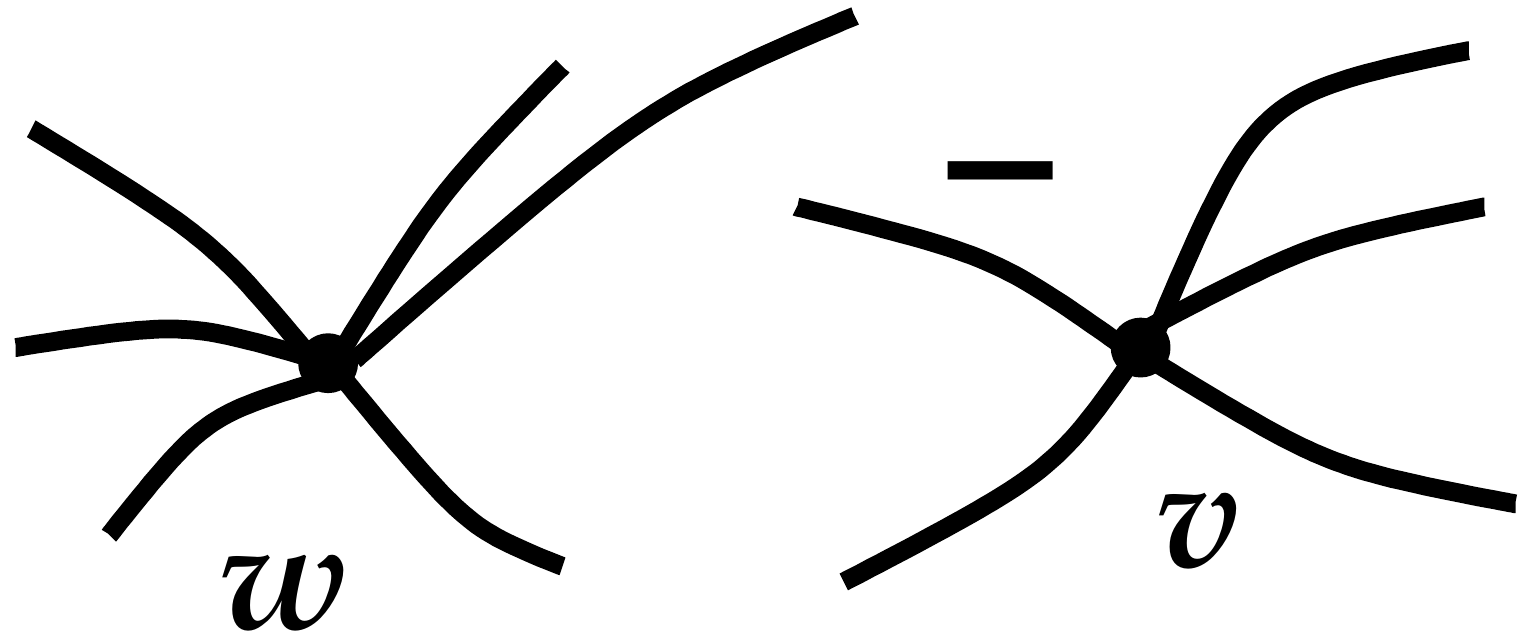} {2in} \\
 
 Now add eight new vertices and 16 new edges to produce $\Gamma'$ as below. It is manifestly
 still standard and a few applications of type I and type II moves show it is 2-equivalent to $\Gamma$.
 But also $TB(\Gamma')=TB(\Gamma)-1$ since $|e^{up}_-|$ has been reduced by one.\\
 
 \hpic {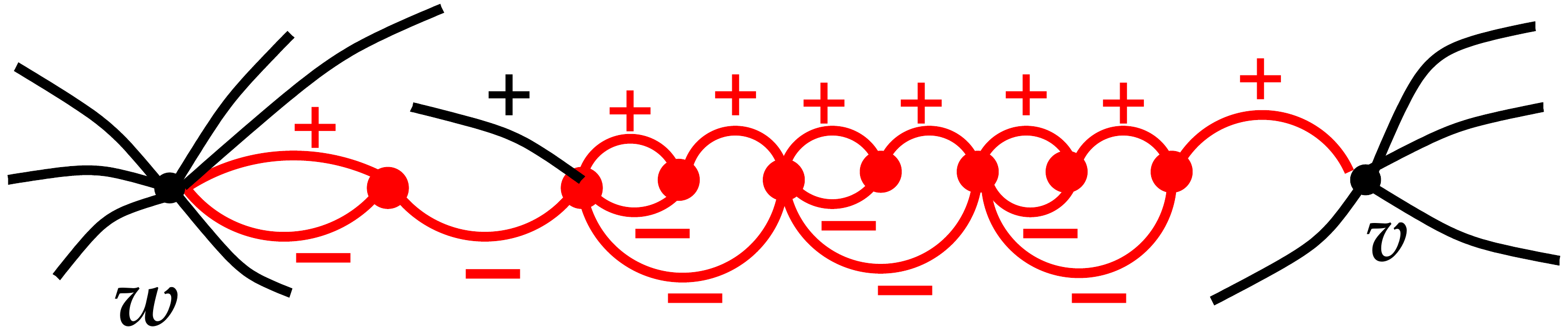} {1in} \\

\end{proof}

We now have all the ingredients for the proof of theorem \ref{all}. Given a link diagram of an unoriented link L, extract the semidual
signed planar graph.  If necessary change $L$ by type I and II moves so that it becomes isotopic to a standard signed planar
graph $\Gamma$. Then use \ref{bad} to reduce $TB$ to zero through a sequence of 2-equivalent standard signed planar graphs. The resulting
graph $\Gamma'$  is Thompson and it is the medial graph of a link diagram for $L$. On the other hand by \ref{treetotree} there is an element $g\in F$ such
that $\Gamma(\langle \pi_{\Xi} (g) \xi,\xi\rangle)=\Gamma'$.

This ends the proof of theorem \ref{all}

\begin{example} \rm{We illustrate our method by obtaining an element of $F$ with
$\langle \pi_{\Xi} (g) \xi,\xi\rangle $ equal to the Borromean rings.
}

\hpic {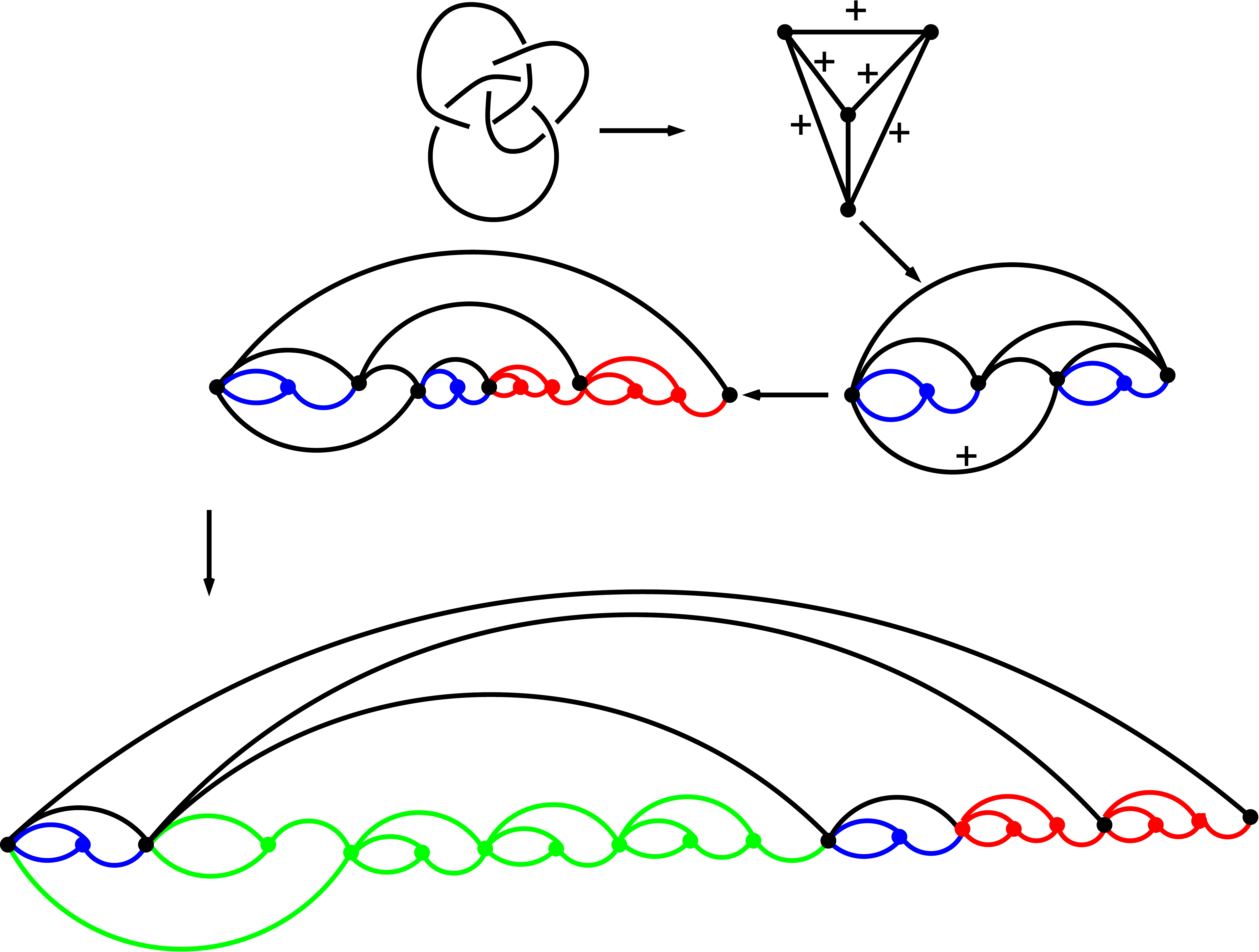} {3.7in} \\

\rm{Once the signed planar graph has become standard we have dropped all signs that are not in agreement with
the final situation with edges in the upper half plane being + and edges in the lower half plane being minus.
}
\end{example}
\end{proof}
\subsubsection{All oriented links arise as coefficients of representations.}

We will show that all oriented links can be obtained by using the group $\overrightarrow F$.

 A shading of a link diagram determines a surface in $\mathbb R^3$ whose boundary is the link-replace the shaded regions
 by smoothly embedded discs in the plane and use twisted rectangular strips to join these shaded regions where the crossings
 are. By construction the boundary of this surface is the link. This surface may or may not be orientable. If it is the link itself
 may be oriented by choosing an orientation of the surface and orienting the knot as the oriented boundary. 

In theorem \ref{calculation} we saw that, for any choice of $R$,  and any pair $T_+,T_-$ of trees defining $g\in F$, the coefficient $\langle \pi_\xi(g)\xi,\xi\rangle$ is the partition function of a 
labelled tangle which is a four-valent planar graph whose planar graph is (definition \ref{treetotree}) the graph $\Gamma(T_+,T_-)$. Thus if we choose
$R$ as in this section, the link $\langle \pi_\xi(g)\xi,\xi\rangle$ is the boundary of a surface obtained as above from $\Gamma(T_+,T_-)$. This surface is
clearly orientable iff $g\in \overrightarrow F$. So if we decree that the shaded region containing the point $(0,0)$ in the graph $\Gamma(T_+,T_-)$ is positively
oriented, the link $\langle \pi_\xi(g)\xi,\xi\rangle$ (up to distant unknots) obtains a definite orientation.

\begin{theorem}\label{alloriented}
Given any oriented smooth link L in $\mathbb R^3$ there is an element $g\in \overrightarrow F$ such that 
$\langle \pi_{\Xi} (g) \xi,\xi\rangle = L$.
\end{theorem}
\begin{proof} It is a well-known fact that any oriented link admits a diagram for which the surface obtained as above from a shading may be oriented
with the link as its oriented boundary (for instance it
follows from the treatment of the HOMFLYPT polynomial via projections with only triple points in \cite{jo2}). 
So we take the first step of the  proof of theorem \ref{all} and isotope such link diagram for $L$ so that  the shaded regions become what we
called a standard edge-signed graph. The standard graph then obtains a 2-colouring if  we
colour each vertex according to the orientation of the region it represents.  Now  we may foliow the proof of theorem \ref{all} just making sure that the element of $F$ that we end up
with is in $\overrightarrow F$. To do this we just need to make sure the graph is bipartite at every step of the procedure we use to
turn the graph into a graph which is Thompson. This procedure required frequent insertions of the two canceling edges thus:\\

\hpic {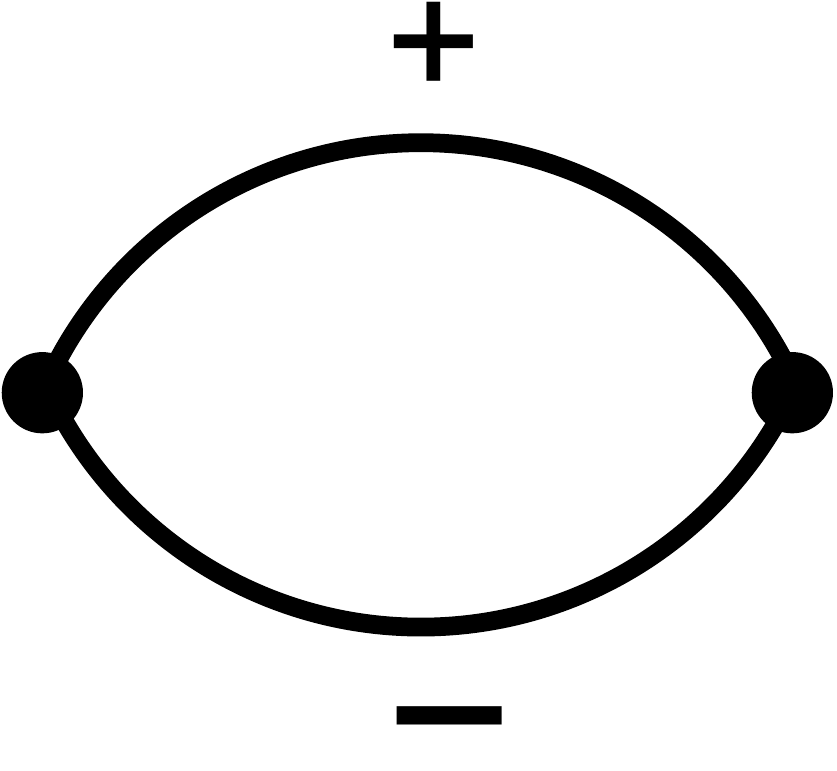} {2in} \\

These  insertions may or may not interfere with the 2-colouring. If they do, simply replace the above insertion by\\

\hpic {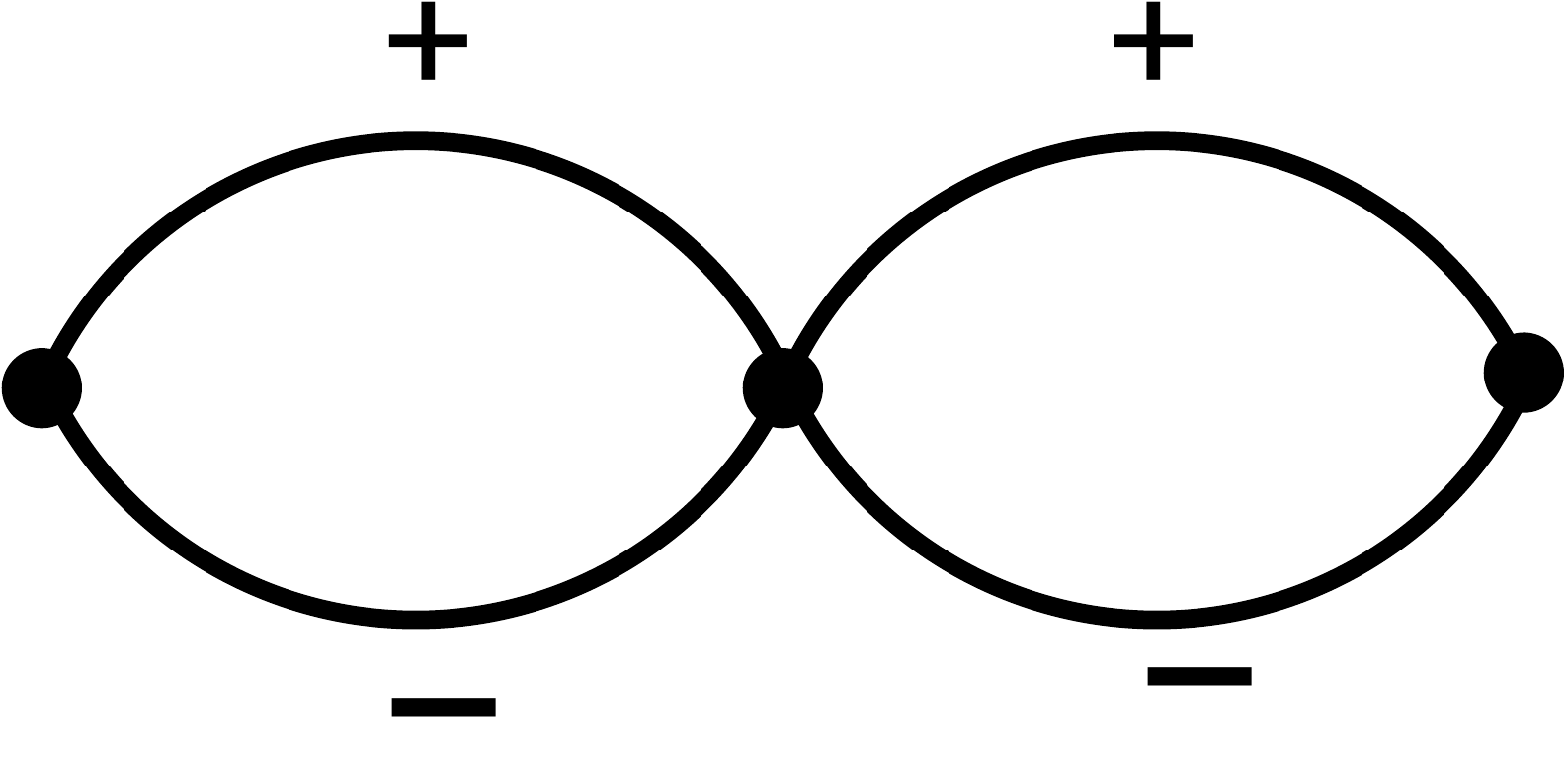} {2in} \\

This has the unfortunate effect of adding more distant unknots but since we decided to ignore them, we are done.

\end{proof} 
\subsubsection{Remarks}
Theorem \ref{all} establishes that the Thompson group is in fact as good as the braid groups at producing unoriented 
knots and links. The theorem is the analogue of the "Alexander theorem" for braids and links.This leads to the following projects:\\

1) Define the F-index of a link as the smallest number of leaves required for an element of $F$ to give that link, and similarly
the T-index. Obviously the F-index is less than or equal to the T-index and examples show that the inequality may be strict.

Perhaps the best result we have on the F-index is that the $(3,n)$ torus link has F-index less than or equal to $n+3$. This link
may be obtained as $\langle g\xi,\xi\rangle$ where $g=\omega^n$, $\omega$ being the Thompson group element:\\

\vpic {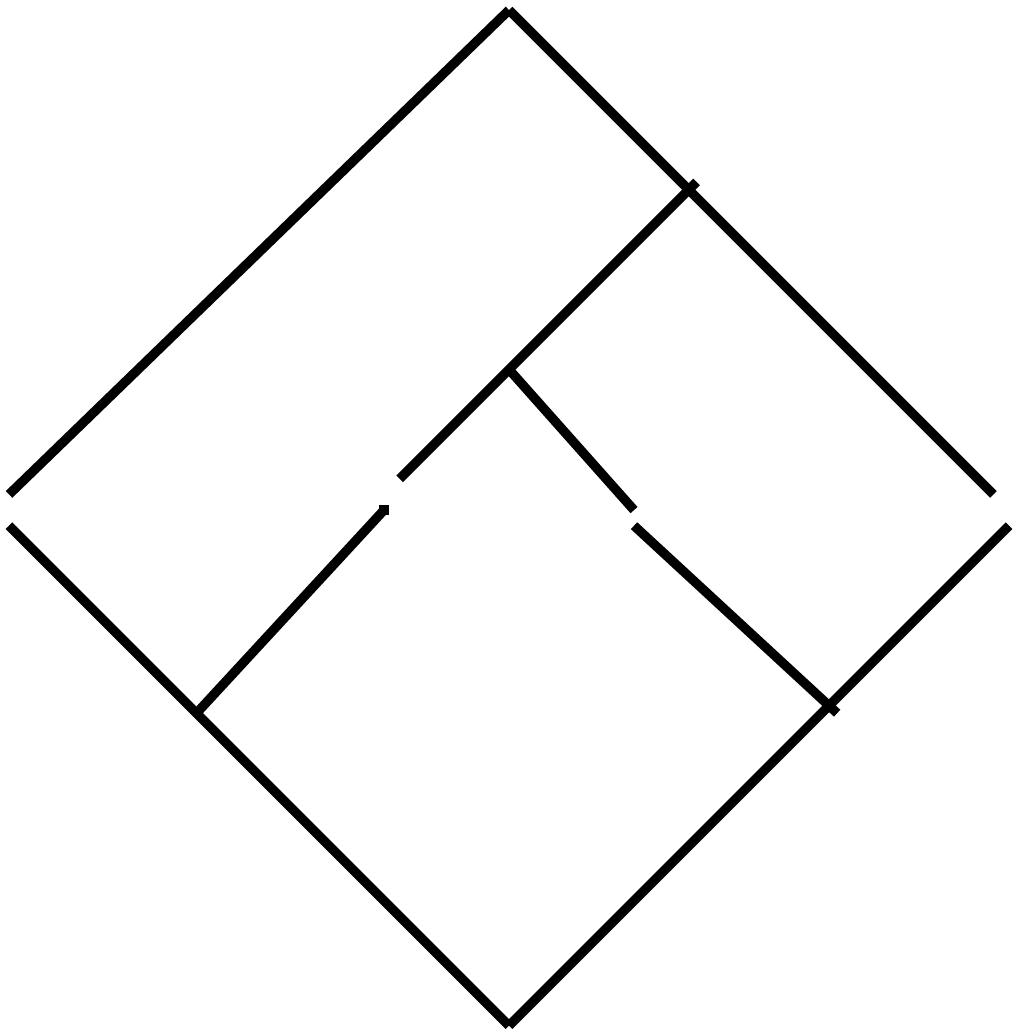} {2in}

The F-index of the Borromean rings is surely considerably less than the upper bound of 20 established above.

2) Prove a "Markov theorem" which determines when two different elements of $F$ (or $T$) give the same knot. It is not hard
to give moves on pairs of trees which effect the Reidemeister moves on the corresponding link diagrams. These moves
are quite likely enough.

3) The braided Thompson group could also be investigated in this regard.

\subsection{More quasi-regular examples.}

By a quasi-regular representation of a discrete group $G$ we mean the natural unitary representation of $G$ by  on $\ell^2(X)$ 
where $G$ acts transitively on the set $X$. In section \ref{chromatic} we unearthed the groups $\overrightarrow F$ and $\overleftarrow F$
(similarly for $T$). This example can be generalized in several ways, even staying within spin model planar algebras.


\subsubsection{More than two spins.}

We restrict ourselves to the case of $F$, leaving $T$ to the reader.

\begin{definition} Let $h:X\rightarrow X$ be a function where $X$ is a set with $Q$ elements.
Define the matrix $\hat{R}_h(x,y)$ as the matrix of the linear transformation induced by $h$, thus 
$$\hat {R}_h(x,y)=\begin{cases} 0 &\mbox{   if   } x\neq h(y)\\
                                             1&\mbox{   if   } y=h(x) \end{cases}  $$

\end{definition}

We know from \cite{jo2}, \cite{Ba2} that matrices indexed by the set of spins define $R$-matrices in spin models. 
Observe that $\hat{R}_h$ is  normalized (definition  \ref{normalized}) if $\sum_{y}R_h(x,y)$. 


We may then fix an element $x_0\in X$  and consider the representation of $F$ generated by the vector $\xi_0$ which is the element of $P_{0,+}$ defined
by $$\xi_0(x)=\begin{cases} 1 & \mbox{  if   } x = x_0\\
                                             0   & \mbox{  otherwise  }\end{cases}.$$  . Under the embeddings defined by $T_{\mathcal I}^{\mathcal J}$ where $\mathcal I$ is the element of $\ds$ with two intervals,
the vector $T_{\mathcal I}^{\mathcal J}(\xi_0) $ is always a basis vector for the usual basis of $P_n$ for spin models (elementary tensors). Thus $F$ acts
on the set of all basis vectors of the dyadic limit Hilbert space $\mathcal V$ of \ref{hilbert}.

The $R$ of section \ref{chromatic} with $Q=2$ is what we would obtain from this construction with $|X|=2$ and $h$ exchanging the elements of $X$.

Now let $g\in F$ be given by two rooted planar trees $\Gamma_\pm$ as in proposition \ref{conditions}. 
Then we define $f_\pm:\{\mbox{ vertices of  }\Gamma_\pm \} \rightarrow X$ by $$f_\pm(x)=h^{d_\pm (x_0,x)}$$ where $d(x_0,x)$ is the
distance from $x_0$ to $x$ on $\Gamma_\pm$. 
\begin{proposition} If $g$ is as above, then $g\xi_0=\xi_0  \iff  f_+=f_-$.

\end{proposition}
\begin{proof} Just follow the discussion after proposition \ref{weak}, taking into account that $h$ need no longer be an involution.
This also allows the enumeration of the stabilizers of the $\xi_0$.
\end{proof}

\end{document}